\documentclass[a4paper,smallmargins,12pt]{article}

\usepackage[]{philip_arxiv_layout}
\usepackage{bm}
\usepackage{dsfont}
\usepackage{amssymb}
\usepackage{mathtools}
\usepackage{microtype}
\usepackage{nicefrac}

\newcommand{\ot}{\leftarrow}

\newcommand{\tvto}{\overset{tv}{\to}}

\newcommand{\diff}{\mathrm{d}}
\newcommand{\angs}[1]{\left( #1 \right)}

\DeclareMathOperator*{\marg}{\mathrm{marg}}

\DeclareMathOperator*{\Indep}{\perp\mkern-11mu\perp}
\DeclareMathOperator*{\nIndep}{\not\mkern-2mu\perp\mkern-11mu\perp}
\DeclareMathOperator*{\Perp}{\perp}
\DeclareMathOperator*{\nPerp}{\not\perp}
\DeclareMathOperator*{\given}{|}
\DeclareMathOperator*{\Do}{\mathrm{do}}

\newcommand{\I}{\mathds{1}}

\newcommand{\EE}{\mathbb{E}}

\newcommand{\NN}{\mathbb{N}}
\newcommand{\PP}{\mathbb{P}}
\newcommand{\QQ}{\mathbb{Q}}
\newcommand{\SSb}{\mathbb{S}}
\newcommand{\RR}{\mathbb{R}}

\newcommand{\Bcal}{\mathcal{B}}

\newcommand{\Dcal}{\mathcal{D}}
\newcommand{\Ecal}{\mathcal{E}}
\newcommand{\Fcal}{\mathcal{F}}
\newcommand{\Gcal}{\mathcal{G}}

\newcommand{\Ical}{\mathcal{I}}

\newcommand{\Mcal}{\mathcal{M}}

\newcommand{\Pcal}{\mathcal{P}}

\newcommand{\Scal}{\mathcal{S}}

\newcommand{\Xcal}{\mathcal{X}}

\usepackage{tikz}
\usetikzlibrary{positioning, arrows}
\usetikzlibrary{automata}
\usetikzlibrary{arrows}
\usetikzlibrary{arrows.meta, calc}
\usetikzlibrary{shapes,positioning,decorations.markings}
\usetikzlibrary{decorations.pathreplacing}
\usetikzlibrary{arrows,fit}
\usetikzlibrary{matrix,decorations.pathmorphing}
\usetikzlibrary{patterns}

\tikzstyle{var}=[circle,draw,thick,minimum size=20pt,inner sep=0pt]
\tikzstyle{vare}=[ellipse,draw,thick,minimum size=20pt,inner sep=0pt]
\tikzstyle{vari}=[shape=rectangle,draw,thick,minimum size=20pt,inner sep=0pt]
\tikzstyle{varh}=[circle,draw,thick,minimum size=20pt,inner sep=0pt,dashed]

\tikzstyle{arr}=[->,>=stealth',draw,thick]
\tikzstyle{arrh}=[->,>=stealth',draw,fill,thick,dashed]
\tikzstyle{biarr}=[<->,>=stealth',draw,thick]

\tikzstyle{ndint} = [draw, line width=1pt, color=teal, shape=rectangle, minimum size=20pt,inner sep=0pt, text=black]
\tikzstyle{ndout} = [draw, line width=1pt, shape=circle, color=blue, minimum size=20pt,inner sep=0pt, text=black]
\tikzstyle{ndlat} = [draw, line width=1pt, shape=circle, color=red, minimum size=20pt,inner sep=0pt, text=black]

\tikzstyle{arint} = [style={->,>=Latex,thick,teal}]
\tikzstyle{arout} = [style={->,>=Latex,thick,blue}]
\tikzstyle{arlout} = [style={->,>=Latex,thick,red}]
\tikzstyle{arlat} = [style={<->,>=Latex,thick,red}]
\tikzstyle{car} = [style={o->,>=Latex,thick}]
\tikzstyle{carc} = [style={o-o,>=Latex,thick}]
\tikzstyle{sarr}=[style={{Rays[n=6]}->,>=Latex,thick}]
\tikzstyle{sars}=[style={{Rays[n=6]}-{Rays[n=6]},thick}]

\tikzstyle{implies}=[double,double equal sign distance,-implies]
\tikzstyle{iff}=[double,double equal sign distance,-implies,postaction={draw,thin,double,double equal sign distance,implies-}]

\tikzset{
    loop arc/.style n args={3}{
        /utils/exec={
            \pgfmathsetmacro{\outAngle}{#1+#2/2}
            \pgfmathsetmacro{\inAngle}{#1-#2/2}
        },
        out=\outAngle,
        in=\inAngle,
        min distance=#3
    }
}

\tikzset{
    loop arc rev/.style n args={3}{
        /utils/exec={
            \pgfmathsetmacro{\outAngle}{#1-#2/2}
            \pgfmathsetmacro{\inAngle}{#1+#2/2}
        },
        out=\outAngle,
        in=\inAngle,
        min distance=#3
    }
}

\usepackage{tikz}
\usetikzlibrary{positioning, arrows}
\usetikzlibrary{automata}
\usetikzlibrary{arrows}
\usetikzlibrary{arrows.meta, calc}
\usetikzlibrary{shapes,positioning,decorations.markings}
\usetikzlibrary{decorations.pathreplacing}
\usetikzlibrary{arrows,fit}
\usetikzlibrary{matrix,decorations.pathmorphing}
\usetikzlibrary{patterns}
\tikzstyle{var}=[circle,draw,thick,minimum size=20pt,inner sep=0pt]
\tikzstyle{vare}=[ellipse,draw,thick,minimum size=20pt,inner sep=0pt]
\tikzstyle{vari}=[shape=rectangle,draw,thick,minimum size=20pt,inner sep=0pt]
\tikzstyle{varh}=[circle,draw,thick,minimum size=20pt,inner sep=0pt,dashed]
\tikzstyle{arr}=[->,>=stealth',draw,thick]
\tikzstyle{arrh}=[->,>=stealth',draw,fill,thick,dashed]
\tikzstyle{biarr}=[<->,>=stealth',draw,thick]

\tikzstyle{ndint} = [draw, line width=1pt, color=teal, shape=rectangle, minimum size=20pt,inner sep=0pt, text=black]
\tikzstyle{ndout} = [draw, line width=1pt, shape=circle, color=blue, minimum size=20pt,inner sep=0pt, text=black]
\tikzstyle{ndlat} = [draw, line width=1pt, shape=circle, color=red, minimum size=20pt,inner sep=0pt, text=black]

\DeclareMathOperator*{\Anc}{\mathrm{Anc}}

\DeclareMathOperator*{\pa}{\mathrm{pa}}

\usepackage{thmtools}
\usepackage{thm-restate}

\declaretheorem[name=Theorem,sibling=theorem,numberwithin=section]{theoremapp}
\declaretheorem[name=Lemma,sibling=theorem,numberwithin=section]{lemapp}
\declaretheorem[name=Definition,sibling=theorem,numberwithin=section,style=definition]{definitionapp}

\declaretheorem[name=Corollary,sibling=theorem,numberwithin=section]{corollaryapp}

\author{Philip Boeken\thanks{Department of Mathematics, VU Amsterdam, \url{p.a.boeken@vu.nl}} \and Joris M.\ Mooij\thanks{Korteweg-de Vries Institute for Mathematics, University of Amsterdam, \url{j.m.mooij@uva.nl}}}
\date{\today}

\title{Causal Graphs, Markov Properties and Do-calculus for Stochastic Differential Equations}

\begin{document}
\maketitle
\begin{abstract}
	Stochastic differential equations (SDEs) are widely used to model continuous-time dynamical systems, but graphical causal models for them are not yet well-understood. We consider \emph{systems of causal SDEs} that are equipped with an explicit causal semantics. We pose solvability conditions for systems of causal SDEs such that they have well-defined observational and interventional distributions --- even after marginalisation --- and provide a general class of Lipschitz semimartingale SDEs that satisfies these conditions. As core results we establish the $\sigma$-separation Markov property and the do-calculus in terms of the system's \emph{causal graph} for probabilistic independence and interventions on the level of sample paths. For a class of additive-noise SDEs we prove a stronger $d$-separation Markov property, even if the system is cyclic. As a corollary of the do-calculus, we obtain an explicit causal interpretation of the graph: that the absence of a directed path implies the absence of a causal effect. We further introduce \emph{time-split} systems, which consider the causal relations between the processes when evaluated on disjoint intervals or time-points, and use them to reason about subsampled time-series, continuous-time Granger non-causality and local independence. Finally, we discuss how constraint-based causal discovery algorithms (PC, FCI, CCD, CCI) apply directly to SDEs within our framework when conditional independence between sample paths can be consistently tested.

\end{abstract}


\section{Introduction}
Stochastic differential equations (SDEs) are widely used to model continuous-time dynamical systems subject to randomness, with applications across the natural sciences and engineering.
Causal graphs for these continuous-time models are not yet well understood.
Celebrated tools from causal modelling are the causal graph, the \emph{Markov Property} (graphical separations between vertices imply conditional independencies between the variables), the \emph{do-calculus} (graphical criteria that relate observational distributions to interventional distributions), and \emph{causal discovery} (estimating the causal graph from data). These tools have not yet been worked out for general (possibly cyclic) systems of SDEs.

Causal modelling of continuous-time systems has been worked out in specific settings. Interventional semantics for equilibrium states of ODEs are developed in \cite{mooij2013ordinary}, and for the dynamic regime of ODEs in \cite{rubenstein2018deterministic} and \cite{bongers2022causal}. \cite{hansen2014causal} defined causal semantics for SDEs.
Graphical models with a Markov property for Granger non-causality have been introduced by \cite{eichler2010granger,eichler2012graphical}.
A continuous-time notion of Granger non-causality called \emph{local independence} has been introduced by \cite{aalen1978nonparametric,aalen2008survival,florens1996noncausality,comte1996noncausality} and developed graphically by \cite{didelez2008graphical,mogensen2020markov,mogensen2022graphical} with causal interpretations by \cite{aalen2012causality,roysland2024graphical} and nonparametric testing in \cite{christgau2023nonparametric}.
Causal discovery for discrete-time systems has been investigated by \cite{malinsky2018causala,runge2019detecting,reiter2024causal}; see \cite{assaad2022survey} for an overview. Causal discovery in continuous time has likewise been carried out for specific graphs: dynamic-independence graphs for stochastic kinetic networks \citep{bowsher2010stochastic}, local-independence graphs via an FCI-type algorithm \citep{mogensen2018causal}, and the directed graph of a stable process from its characteristic exponents \citep{bruck2026graph}. Other approaches to causal discovery for systems of SDEs include \cite{manten2024signature}, \cite{guan2024identifying}, \cite{engelke2024levy} and \cite{nathaniel2025deep}. None of these lines of work establishes a graphical Markov property in terms of $\sigma$- or $d$-separation for probabilistic independence between sample paths, or provides a do-calculus for general (possibly cyclic) SDEs.

The framework of Structural Causal Models (SCMs) provides many tools for causal reasoning in static and discrete-time settings: causal graphs encode conditional independence structure, the do-calculus can be employed to identify causal effects from observational distributions, and constraint-based algorithms recover causal structure from data.
Building on the cyclic-SCM theory of \cite{forre2017markov,bongers2021foundations,forre2025mathematical} and the pathwise solution function for SDEs of \cite{przybylowicz2024skorohod}, this paper develops such a framework for general systems of causal SDEs, allowing cycles, instantaneous relations, jumps, and non-Markovianity.

Causal models that abstract away time in an unspecified way are inherently ambiguous \citep{reisach2025case}. Most existing approaches that explicitly model time do so in discrete time, e.g., via structural vector autoregression models or dynamic Bayesian networks. We give a formal interpretation of these models as projections of underlying continuous-time systems observed at specific points in time. This projection can distort causal inference: conditional independencies may vanish, estimates of causal effects may be invalid, and inferred causal structure may not correspond to its true continuous-time counterpart.
By grounding causal semantics directly in continuous-time dynamics, our framework provides a principled foundation for interpreting discretely observed data while remaining mindful of the underlying processes.

\subsection{Contributions}
We equip systems of SDEs with a structural causal semantics via the notion of perfect intervention, following \cite{mooij2013ordinary,hansen2014causal,rubenstein2018deterministic,peters2020causal}, and we give a definition of the \emph{causal graph} of such systems of causal SDEs. Although not strictly necessary for our analysis, we also give a pathwise, functional formulation of the stochastic integral (Definition~\ref{def:sde_pathwise}): it lets one read the system of SDEs as a structural causal model (SCM), which we find conceptually clarifying, and it supplies the notation $\Phi_v$ for each causal mechanism (structural equation) that we rely on throughout the proofs. We then provide the following results:
\begin{enumerate}[label=(\roman*)]
    \item \emph{Solvability.} We analyse the \emph{essential unique solvability} (Definition \ref{def:ess_unique_solvable}) of a system of causal SDEs with respect to a subset of variables, which is a solvability condition for a subset of the system, under arbitrary adapted inputs for the remaining variables. We provide two model classes of particular interest that satisfy these conditions: a class of semimartingale SDEs that satisfies standard Lipschitz and linear-growth conditions (Assumption~\ref{ass:uniquely_solvable}), and an additive noise SDE driven by Brownian motion (Assumption~\ref{ass:additive_noise}).
    \item \emph{Marginalisation.} We define the marginalisation of a system of causal SDEs on a subset of the variables (Definition~\ref{def:marginalisation}) that is causally consistent with the original model, and show that essential unique solvability is preserved under this operation (Theorem~\ref{thm:marg_closure_simple}).
    \item \emph{Markov properties.} Following the acyclification strategy of \cite{forre2017markov,bongers2021foundations,forre2025mathematical}, we prove the $\sigma$-separation Markov property in the causal graph $G(\Dcal)$ for certain essentially uniquely solvable systems (Theorem~\ref{thm:sigma_sep_mp}). For the class of additive-noise SDEs, we strengthen this to the $d$-separation Markov property (Theorem~\ref{thm:dsep_mp}).
    \item \emph{Do-calculus.} We establish the three rules of the do-calculus in terms of $\sigma$-separation in the causal graph for essentially uniquely solvable systems (Theorem~\ref{thm:do_calculus}). As a consequence, the absence of a directed path implies the absence of a causal effect, giving the graph a clear causal interpretation (Theorem~\ref{thm:causation_implies_path}).
    \item \emph{Time-splitting and subsampling.} We introduce \emph{time-split} systems of causal SDEs, which separately model the causal relations between processes on disjoint subintervals (Definition~\ref{def:sdes_timesplit}), and as a special case the \emph{subsampled} system that marginalises the time-split system onto a set of time-point measurements (Definition~\ref{def:subsampled}). We transfer the $\sigma$- and $d$-separation Markov properties and the do-calculus to them, which formalises continuous-time Granger non-causality \citep{granger1969investigating} and local independence \citep{schweder1970composable,didelez2008graphical,mogensen2020markov}, and yields a Markov property for local independence.
    \item \emph{Causal discovery.} We give a proof-of-concept application of the cyclic constraint-based discovery algorithms FCI \citep{spirtes1999algorithm,mooij2020constraintbased} and CCI \citep{strobl2019constraintbased} to SDEs, relying on a conditional independence oracle.
\end{enumerate}
All results are illustrated on a single running example, introduced in the following section. Proofs are given in the appendix.

\subsection{Example: the repressilator}
To make the exposition more concrete, we examine a biological system that has been extensively studied and lends itself naturally to a stochastic modelling perspective. Gene regulatory networks are a prototypical example: they consist of interacting components whose dynamics are driven both by feedback loops and by random fluctuations at the molecular level. Among such systems, the \emph{repressilator} \citep{elowitz2000synthetic} has become a canonical case study. It illustrates how cyclic interactions, noise, and causal mechanisms interact in a way that is both biologically relevant and mathematically tractable. We use this system as a running example of a system of SDEs that are \emph{causal}, which means that the variable on the left-hand side of the equation is directly caused by those on the right-hand side of the equation.
\begin{example}[Repressilator]
    The \emph{repressilator} is an oscillator of gene expressions in the E.~coli bacteria, specifically of the genes \emph{lacI, tetR, cI} and \emph{GFP}. Extending the original ODE model of \cite{elowitz2000synthetic} by adding Brownian noise, the dynamics of the mRNA abundance $M_i$ and protein values $P_i$ for each gene $i$ and its predecessor $j$ for $(i,j) \in$ $\{$(lacI, cI), (tetR, lacI), (cI, tetR), (GFP, tetR)$\}$ are given by
    \begin{align*}
        \diff M_{i}(t) & = \left(\frac{\alpha}{1+|P_{j}(t)|^n} + \alpha_0 - M_{i}(t)\right) \diff t + \sigma \diff Z^M_{i}(t) \\
        \diff P_{i}(t) & = \beta(M_{i}(t) - P_{i}(t)) \diff t + \sigma \diff Z^P_{i}(t),
    \end{align*}
    with given exogenous random variables for the initial conditions $M_i(0)$ and $P_i(0)$, independent Brownian motions $Z^M_{i}$ and $Z^P_{i}$, and values of $\alpha, \alpha_0, \beta, \sigma > 0$ and Hill coefficient $n\geq 1$. Setting $n=1$ gives the well-known Michaelis-Menten kinetics.
    In this model, protein abundance $P_j$ inhibits the mRNA $M_i$ of protein $i$, which is required for growth of the protein abundance $P_i$.
    The dependence structure is graphically depicted in Figure \ref{fig:repressilator_graph}, where $u\to v$ if variable $u$ occurs in the SDE for variable $v$. In Definition \ref{def:sde_graph} we will formalise this as a causal graph.
    A simulated sample path of the proteins in the cycle is provided in Figure \ref{fig:repressilator_obs}.
    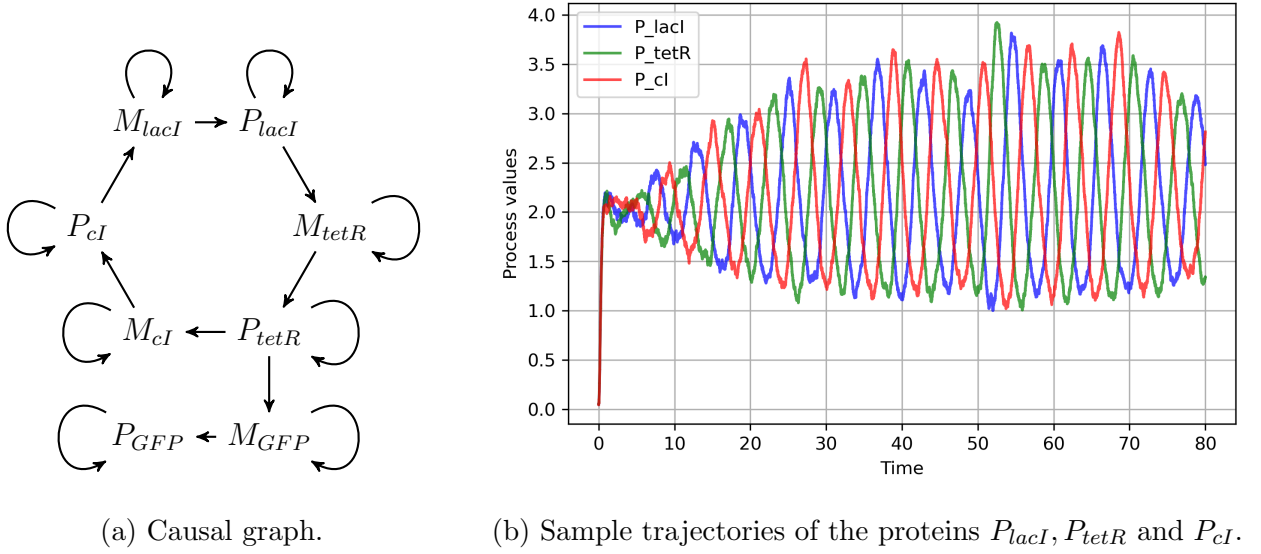
\begin{figure}[!htb]
        \centering
        \begin{subfigure}{0.38\linewidth}
            \centering
            \begin{tikzpicture}[scale=1.6]
                \node[] (X1) at (-0.5,0.866) {$M_{lacI}$};
                \node[] (Y1) at (0.5,0.866) {$P_{lacI}$};
                \node[] (X2) at (1,0) {$M_{tetR}$};
                \node[] (Y2) at (0.5,-0.866) {$P_{tetR}$};
                \node[] (X3) at (-0.5,-0.866) {$M_{cI}$};
                \node[] (Y3) at (-1,0) {$P_{cI}$};
                \node[] (X4) at (0.5,-1.732) {$M_{GFP}$};
                \node[] (Y4) at (-0.5,-1.732) {$P_{GFP}$};
                \node[] () at (0,-2) {};
                \draw[arr] (X1) to (Y1);
                \draw[arr] (Y1) to (X2);
                \draw[arr] (X2) to (Y2);
                \draw[arr] (Y2) to (X3);
                \draw[arr] (X3) to (Y3);
                \draw[arr] (Y3) to (X1);
                \draw[arr] (Y2) to (X4);
                \draw[arr] (X4) to (Y4);
                \draw[arr] (X1) to [loop arc={90}{60}{0.6cm}] (X1);
                \draw[arr] (X2) to [loop arc={0}{60}{0.6cm}] (X2);
                \draw[arr] (X3) to [loop arc rev={180}{60}{0.6cm}] (X3);
                \draw[arr] (X4) to [loop arc={0}{60}{0.6cm}] (X4);
                \draw[arr] (Y1) to [loop arc={90}{60}{0.6cm}] (Y1);
                \draw[arr] (Y2) to [loop arc={0}{60}{0.6cm}] (Y2);
                \draw[arr] (Y3) to [loop arc rev={180}{60}{0.6cm}] (Y3);
                \draw[arr] (Y4) to [loop arc rev={180}{60}{0.6cm}] (Y4);
            \end{tikzpicture}
            \caption{Causal graph.}
            \label{fig:repressilator_graph}
        \end{subfigure}
        \hfill
        \begin{subfigure}{0.6\linewidth}
            \centering
            \includegraphics[width=1\linewidth]{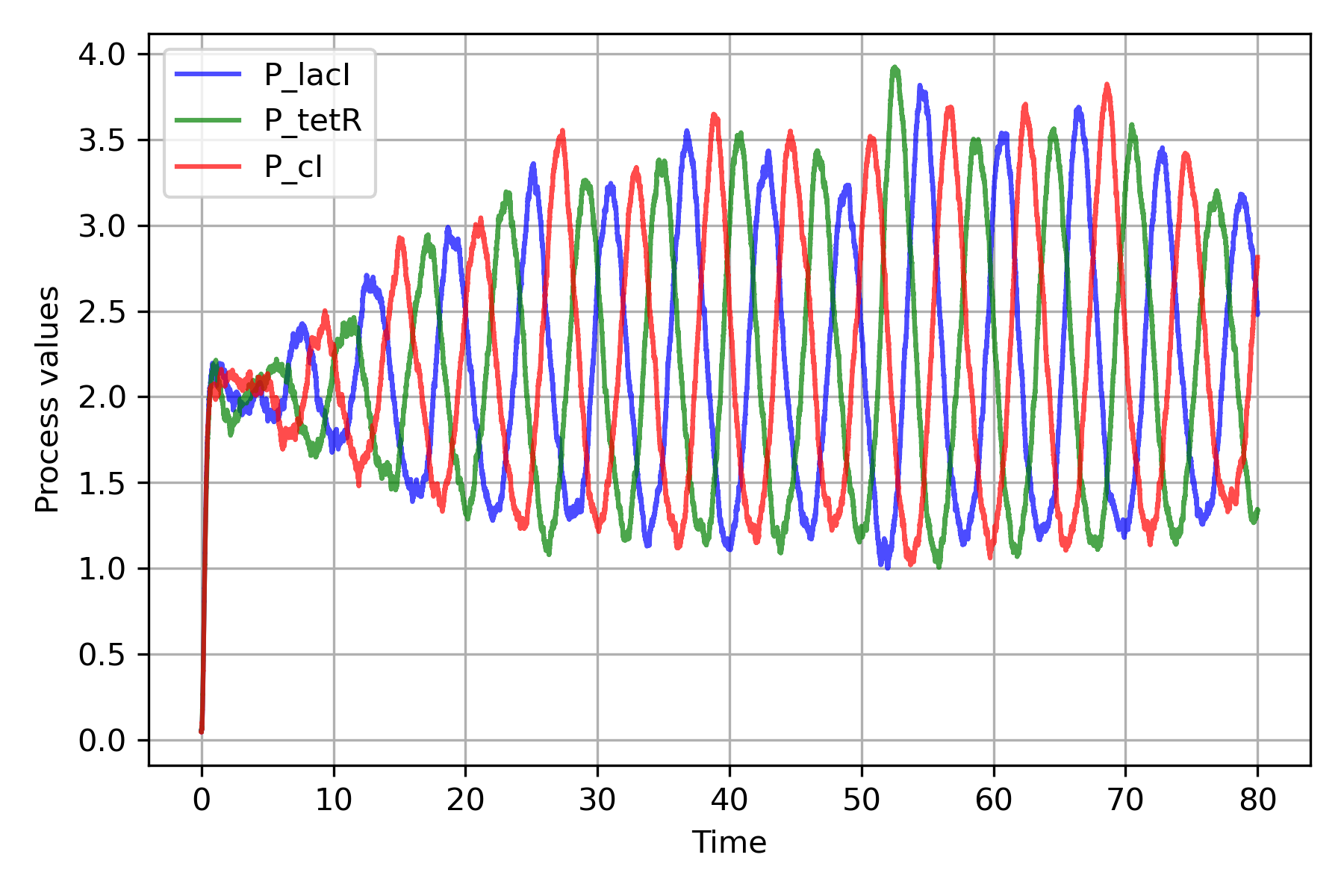}
            \caption{Sample trajectories of the proteins $P_{lacI}, P_{tetR}$ and $P_{cI}$.}
            \label{fig:repressilator_obs}
        \end{subfigure}
        \caption{The repressilator of the E.\ coli bacteria, modelled with additive noise.}
    \end{figure}
\end{example}

Experimentation is a cornerstone of the discovery of causal relations: by deliberately perturbing a system and observing the resulting changes, one can distinguish genuine cause-effect relationships from mere correlations. In the context of stochastic dynamics, this means modifying certain components of the governing equations and analysing how the system responds over time. Importantly, interventions are not only of interest for discovering causal structure, they also arise naturally in practice, where one aims to steer a system toward a desired state. Translating this intuition into the language of SDEs, a \emph{perfect intervention} amounts to altering or overriding the dynamics of a particular variable by fixing that component to a pre-specified sample path, without any dependency on the remainder of the system, and without any side-effects. In the following example, we illustrate this idea by intervening on the repressilator.

\begin{example}[Intervening on the repressilator]
    A typical way of experimenting with gene regulatory networks is to knock out a particular gene, disabling the production of the protein.
    Suppose we knock out the lacI gene, then we obtain the system where we set $M_{lacI}(t)=0$ for all $t$. A sample of the resulting process for the proteins is depicted in Figure \ref{fig:repressilator_knockout}. We see that tetR is not inhibited anymore and thus reaches high activity levels, highly suppressing cI abundance. Similarly, when exciting the lacI gene, this suppresses tetR, which in turn excites cI levels, as depicted in Figure \ref{fig:repressilator_excite}.
    \begin{figure}[!htb]
        \centering
        \begin{subfigure}{0.496\linewidth}
            \centering
            \includegraphics[width=1\linewidth]{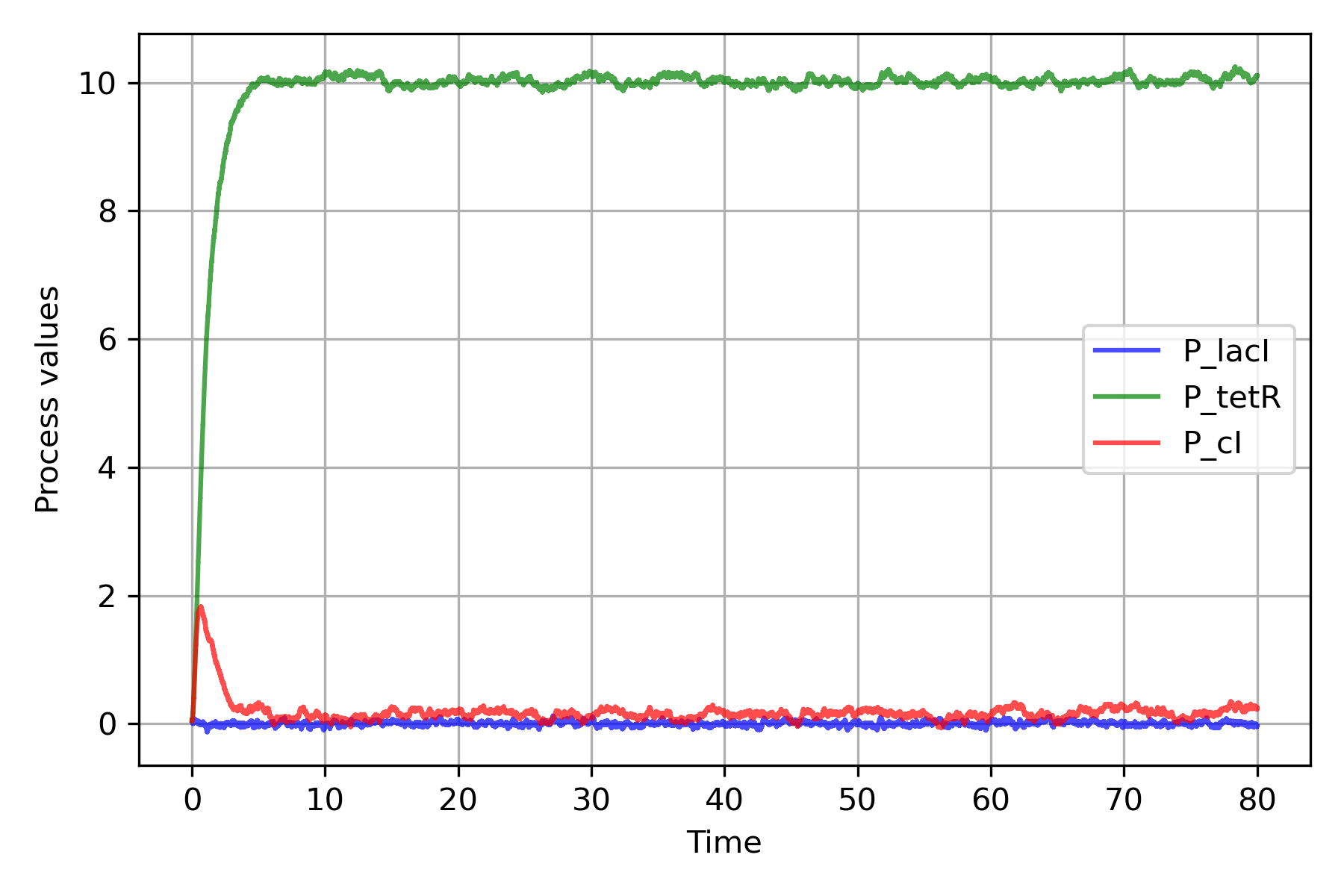}
            \caption{Gene `lacI' knocked out.}
            \label{fig:repressilator_knockout}
        \end{subfigure}
        \hfill
        \begin{subfigure}{0.496\linewidth}
            \centering
            \includegraphics[width=1\linewidth]{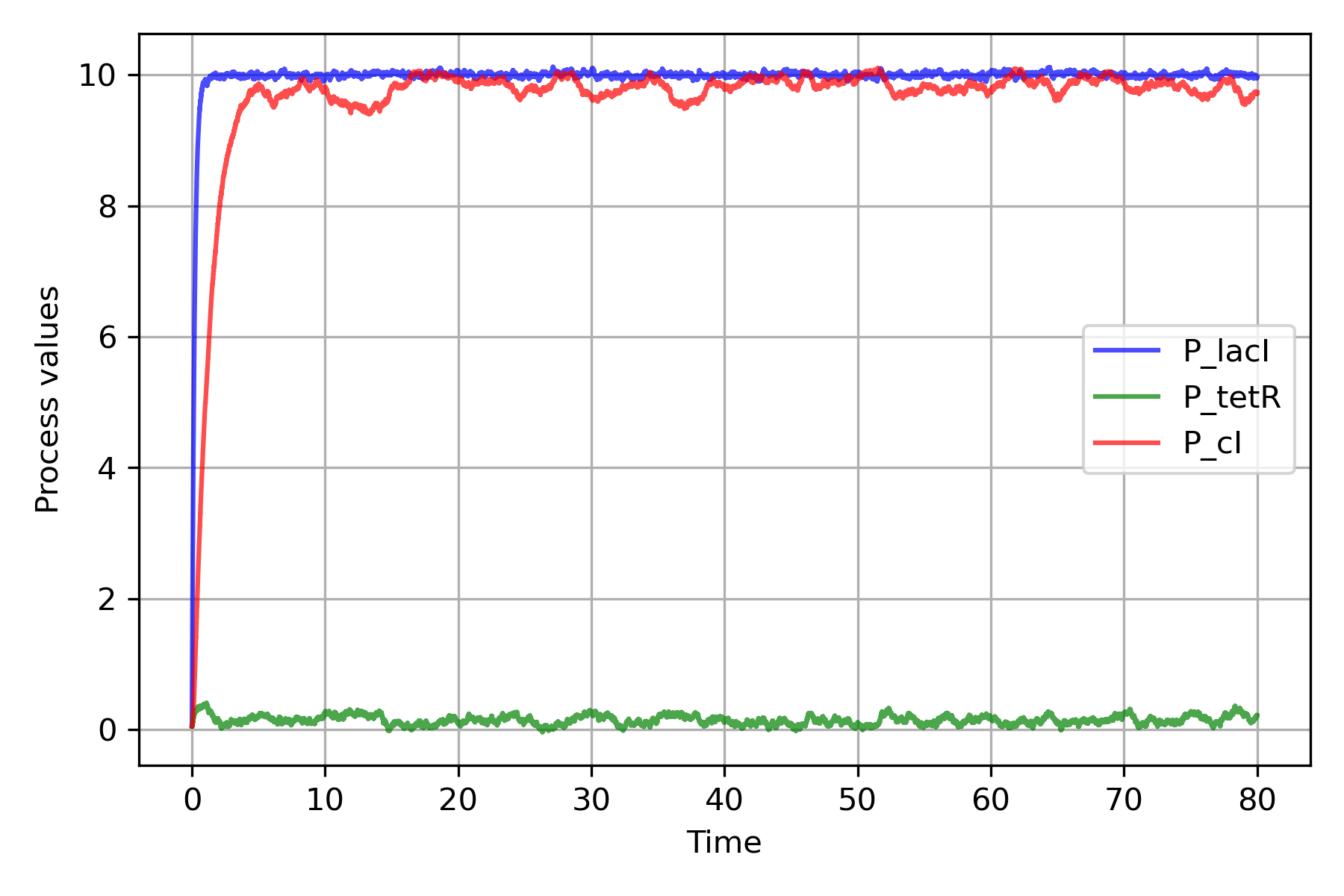}
            \caption{Gene `lacI' excited.}
            \label{fig:repressilator_excite}
        \end{subfigure}
        \caption{The effects of interventions on the `lacI' gene of the repressilator.}
        \label{fig:repressilator_int}
    \end{figure}
\end{example}



In the following section, we make systems of causal SDEs and interventions on them (as demonstrated using the repressilator) mathematically precise, allowing for jumps in the sample paths, non-Markovianity, and instantaneous dependencies.

\section{Systems of causal SDEs}\label{sec:sdes}
Let $(\Omega, \Fcal, \PP)$ denote a filtered probability space with filtration $\Fcal = (\Fcal_t)_{t\in [0,T]}$ for some $T<\infty$, satisfying \emph{the usual conditions}: $\Fcal$ is right-continuous and contains all $\PP$-null-sets.
Let $D([0,T], \RR^n)$ be the space of functions from $[0,T]$ to $\RR^n$ that are càdlàg (continue \`a droite, limité \`a gauche, i.e.\ right-continuous with left limits).\footnote{Throughout, we equip $D([0,T], \RR^n)$ with the $J_1$ topology \citep{skorokhod1956limit} and the Borel $\sigma$-algebra, making it a standard Borel space. We make no distinction between $D([0,T], \RR)^n$ and $D([0,T], \RR^n)$ since they are measurably isomorphic. The product topology on $D([0,T], \RR)^n$ is strictly weaker than the Skorokhod topology on $D([0,T], \RR^n)$ \citep[Section VI.1b]{jacod2003limit}, but this does not cause problems: the only continuity conditions are imposed in Assumptions \ref{ass:uniquely_solvable} and \ref{ass:additive_noise} via Theorem~\ref{thm:ito_map}, which are insensitive to this distinction.}
A stochastic process $X: \Omega \times [0,T] \to \RR^n$ is adapted if $X(t) \in \Fcal_t$ for all $t\in [0,T]$, and càdlàg if it has càdlàg sample paths almost surely, in which case we can represent it as a process $X:\Omega \to D([0,T], \RR^n)$.
Let $\mathbb{S}([0,T], \RR^n)$ denote the class of $\RR^n$-valued semimartingales, i.e.\ adapted processes $X : \Omega \to D([0,T], \RR^n)$ that admit a decomposition $X = \Lambda + M$ with $\Lambda$ a process of finite variation and $M$ a local martingale.\footnote{Note that $\SSb([0,T], \RR^n) = \SSb([0,T], \RR)^n$.}
Denoting with $\Pcal([0,T], \RR^m)$ the set of predictable processes on $\RR^m$, the stochastic (Itô) integral of $G$ w.r.t.\ $H$ is a mapping
\begin{equation}\label{eqn:stochastic_integral}
    J: \Pcal([0,T], \RR^m) \times \SSb([0,T], \RR^m) \to \SSb([0,T], \RR), ~~(G, H) \mapsto \int_0^{(\cdot)} G(s) \diff H(s),
\end{equation}
where $\int_0^t G(s) \diff H(s) = \sum_{i=1}^m\int_0^t G_i(s) \diff H_i(s)$ (\citealp{protter2005stochastic}, Chapter IV, Theorem 15).

We call a Borel-measurable function $f : [0,T] \times D([0,T], \RR^m) \to \RR^n$ \emph{càdlàg} if $t\mapsto f(t, x) \in D([0,T], \RR^n)$ for all $x\in D([0,T], \RR^m)$, \emph{adapted} if $f(t, x) = f(t, x^{\wedge t})$ for all $x\in D([0,T], \RR^m)$ and $t\in[0,T]$, where $x^{\wedge t}(s) := x(s\wedge t)$.
If the process $X$ is càdlàg and adapted and the function $f$ is measurable, càdlàg and adapted, then the processes $f(t, X)$ and $f(t-, X) := \lim_{s\uparrow t}f(s, X)$ are adapted, and the process $f(t-, X)$ is predictable \citep[Lemma 2.2]{przybylowicz2024skorohod}.

\begin{definition}[System of causal SDEs]\label{def:sdes}
    Given a filtered probability space $(\Omega, \Fcal, \PP)$ satisfying the usual conditions, a \emph{system of causal stochastic differential equations (SDEs)} is a tuple $\Dcal = (V, W, X_W, f, g, h)$ with $V$ and $W$ finite disjoint index sets, for each $w\in W$ a stochastic process $X_w \in \SSb([0,T], \RR)$ with the family $(X_w)_{w \in W}$ mutually independent, and for each $v\in V$ a \emph{causal SDE}
    \begin{align}\label{eqn:sde}
        X_v(t) = f_v(t, X_{\alpha(v)}) + \int_0^t g_{v}(s-, X_v, X_{\beta(v)}) \diff h_v(s, X_{\gamma(v)}) \quad \text{for all $t\in[0,T]$,}
    \end{align}
    with given sets $\alpha(v)\subseteq V\cup W$, $\beta(v) \subseteq (V \cup W)\setminus \{v\}$ and $\gamma(v)\subseteq W$, and for some $m_v \in \NN\cup \{0\}$, measurable, càdlàg and adapted functions
    \begin{align*}
        f_v & : [0,T]\times D([0,T], \RR^{|\alpha(v)|}) \to \RR         \\
        g_v & : [0,T] \times D([0,T], \RR^{|\beta(v)|+1}) \to \RR^{m_v} \\
        h_v & : [0,T] \times D([0,T], \RR^{|\gamma(v)|}) \to \RR^{m_v}
    \end{align*}
    such that $h_v(t, X_{\gamma(v)})\in \SSb([0,T], \RR^{m_v})$.
\end{definition}
A \emph{solution} (commonly referred to as a \emph{strong solution}) of a system of SDEs $\Dcal$ is an adapted process $X_V: \Omega \to D([0,T], \RR^{|V|})$ such that for each $v\in V$ the equation (\ref{eqn:sde}) holds $\PP$-a.s. The qualifier ``causal'' reflects that, like structural equations in SCMs \citep{pearl2009causality}, each equation \eqref{eqn:sde} encodes a causal mechanism rather than an algebraic constraint. Algebraically rewriting the system of SDEs might give the same solutions, but changes the causal graph and the effects of interventions (as defined below).

The function $f_v(t, X_{\alpha(v)})$ can be used to model the initial value of $X_v$ and instantaneous functional dependencies. For the former, we say that \emph{$f_v$ models the initial condition} if $f_v(t, X_{\alpha(v)})$ is constant in $t$, so that it only describes the initial value $X_v^0 := f_v(0, X_{\alpha(v)}^{\wedge 0})$. In this case, equations of the form (\ref{eqn:sde}) are often suggestively written as
\begin{equation*}
    \diff X_v(t) = g_{v}(t-, X_v, X_{\beta(v)}) \diff h_v(t, X_{\gamma(v)}) \quad\quad X_v(0) = X_v^0,
\end{equation*}
hence the name stochastic \emph{differential} equation. Since $h_v$ is typically not differentiable this has no other meaning than the integral equation.
The processes $g_v(t-, X_v, X_{\beta(v)})$ and $h_v(t, X_{\gamma(v)})$ are called \emph{integrands} and \emph{integrators} respectively. We call \emph{$g_v$ Markov} if $g_v(t, X_v, X_{\beta(v)}) = g_v(t, X_v(t), X_{\beta(v)}(t))$, and \emph{time-invariant} if $g_v(t, X_v, X_{\beta(v)}) = g_v(X_v(t), X_{\beta(v)}(t))$. If $X_{\gamma(v)}$ is a semimartingale and $h_v(t, X_{\gamma(v)}) = h_v(X_{\gamma(v)}(t))$ and the map $x\mapsto h_v(x)$ is twice continuously differentiable, then $h_v(t, X_{\gamma(v)})$ is also a semimartingale. If the process $h_v(t, X_{\gamma(v)})$ is deterministic, it is a semimartingale if and only if its path is of finite variation. If the integrator $h_v(t, X_{\gamma(v)})$ is $\PP$-almost surely continuous, one obtains the same solutions without taking the left limit in the integrand.

Definition \ref{def:sdes} generalises common definitions of SDEs, by allowing for functional relations via $f_v$.
For specific choices of integrators and integrands, systems of causal SDEs model some special cases:
\begin{enumerate}[label=\roman*)]
    \item If $h_v(s, X_{\gamma(v)}) = s$, the `stochastic integral' $\int g_v(s-, X_v, X_{\beta(v)}) \diff h_v(s, X_{\gamma(v)})$ reduces to the Riemann integral $\int g_v(s, X_v, X_{\beta(v)}) \diff s$, so Ordinary Differential Equations and Random Differential Equations are special cases of SDEs. If $h_v(s, X_{\gamma(v)})$ is of finite variation, the integral $\int g_v(s-, X_v, X_{\beta(v)}) \diff h_v(s, X_{\gamma(v)})$ reduces to the Stieltjes integral. If $g_v(s, X_v, X_{\beta(v)})$ and $h_v(s, X_{\gamma(v)})$ have no common discontinuities and finite $p$-variation and $q$-variation respectively with $p^{-1}+q^{-1} >1$, the integral reduces to the Young integral \citep[Section 1.3]{lyons2007differential}.
    \item If $f_v$ models the initial condition, $g_v$ is Markov and $h_v(\cdot, X_{\gamma(v)})$ is a vector of independent L\'evy processes for all $v\in V$, the solution $X_V$ is temporally Markov\footnote{Process $X\in \SSb$ is called temporally \emph{Markov} if $\PP(X_{t+s} \given \Fcal_{t}) = \PP(X_{t+s} \given X_{t})$ for all $s,t\in [0,T]$ such that $s+t \in [0,T]$; it is temporally \emph{strong Markov} if this holds for any stopping time $t$.} (see e.g.\ \citealp{protter2005stochastic}, Chapter V, Theorem 32 for the case that $g_v$ is Lipschitz). If additionally $g_v$ is time-invariant, then the solution is temporally strong Markov.
    \item If $f_v$ models the initial condition, $g_v$ is Markov and $h_v(t, X_{\gamma(v)}) = (t, X_{\gamma(v)}(t))$ where $X_{\gamma(v)}$ is a Brownian motion, any solution $X_v$ has continuous and possibly non-differentiable sample paths.
    Such an SDE is called an \emph{It\^o diffusion}.
    \item If $f_v$ models the initial condition, $g_v$ is Markov and $h_v(t, X_{\gamma(v)})$ is a jump process (e.g.\ a Poisson process), the solution $X_v$ can have jumps as well. If $h_v(t, X_{\gamma_1(v)}, X_{\gamma_2(v)}) = (t, X_{\gamma_1(v)}(t), X_{\gamma_2(v)}(t))$ with $X_{\gamma_1(v)}$ and $X_{\gamma_2(v)}$ a Brownian motion and a jump process respectively, such an SDE is called a \emph{jump-diffusion}.
\end{enumerate}
Given a system of causal SDEs $\Dcal$, we define its \emph{causal graph} as follows:
\begin{definition}[Causal graph]\label{def:sde_graph}
    Given a system of causal SDEs $\Dcal = \angs{V, W, X_W, f, g, h}$,
    the \emph{augmented causal graph} is the directed graph $G^+(\Dcal) = (V\cup W, E)$ where
    \begin{align*}
        E := & \left\{u \to v : v \in V, u \in \alpha(v) \text{ and $f_v$ is not constant in } X_u\right\}               \\
             & \cup \left\{u \to v : v \in V, u \in \{v\}\cup \beta(v) \text{ and $g_v$ is not constant in } X_u\right\} \\
             & \cup \left\{u \to v : v \in V, u \in \gamma(v) \text{ and $h_v$ is not constant in } X_u\right\}
    \end{align*}
    The \emph{causal graph} is the directed mixed graph $G(\Dcal)=(V, E', L)$, where $L = \{i \leftrightarrow j : i \neq j \in V, (i \leftarrow k \to j) \in G^+(\Dcal) \text{ for some $k\in W$}\}$ and $E'$ is the restriction of $E$ to $V$.
\end{definition}
For a subset $A$ of the vertices of a directed mixed graph $G$, we write $\pa(A) := \{u\notin A : u\to v \text{ in } G \text{ for some } v\in A\}$ for the \emph{parents} of $A$ (excluding $A$) and $\Anc(A) := A\cup\{u : u\to\cdots\to v \text{ in } G \text{ for some } v\in A\}$ for its \emph{ancestors} (by definition including $A$), where $u$ ranges over the vertices of $G$. Unless another graph is specified the ambient graph is $G(\Dcal)$, so in particular $\pa(A)\subseteq V$. We identify each vertex $v\in V$ with its associated process $X_v$. We sometimes abuse notation to write the labels of the random variables for the vertices in the graphs (as in Figure \ref{fig:repressilator_graph}) and for graphical separation statements (e.g.\ in Section \ref{sec:time-evaluations}).
For the repressilator, the causal graph $G(\Dcal)$ is depicted in Figure \ref{fig:repressilator_graph}, where the noise processes and initial values are considered to be exogenous.

In line with \cite{mooij2013ordinary}, \cite{hansen2014causal} and \cite{peters2020causal}, we define perfect interventions on systems of causal SDEs as follows.
\begin{definition}[Perfect intervention]\label{def:intervened_SDE}
    Given a system of causal SDEs $\Dcal = \angs{V, W, X_W, f, g, h}$, intervention target $S \subseteq V$ and intervention value $x_S \in D([0,T], \RR)^{|S|}$, the \emph{perfectly intervened system of causal SDEs} is defined as $\Dcal_{\Do(X_S = x_S)} = \angs{V, W, X_W, f^\circ, g^\circ, h^\circ}$ with for each $v\in V$ the mechanisms $f_v^\circ(t, \cdot) := x_v(t)$, $g_v^\circ := 0$ and $h_v^\circ := 0$ if $v\in S$ and $f_v^\circ := f_v$, $g_v^\circ := g_v$ and $h_v^\circ := h_v$ otherwise, that is, we have for each $v\in V$ the \emph{intervened causal SDE}
    \begin{align*}
        \begin{cases}
            X_v(t) = x_v(t)                                                                                    & \text{if $v \in S$}            \\
            X_v(t) = f_v(t, X_{\alpha(v)}) + \int_0^t g_{v}(s-, X_v, X_{\beta(v)}) \diff h_v(s, X_{\gamma(v)}) & \text{if $v\in V\setminus S$}.
        \end{cases}
    \end{align*}
\end{definition}

If one wants to model interventions on parameters that determine the dynamics of other variables, those parameters should be modelled as an endogenous process, which can be intervened upon.

Mathematically `intervening' on a system of causal SDEs should be done with caution: it may not always reflect a realistic action. Suppose we can precisely determine the inflow and the outflow of atoms in the repressilator. Intervening on both inflow and outflow simultaneously could lead to physical contradictions (as the number of atoms is a conserved quantity in chemical reactions), so the mathematical intervention may have no realistic counterpart, and only interventions on a subset of the variables, for a subset of intervention values, may have a realistic interpretation.

\subsection{Pathwise interpretation of systems of causal SDEs}
The integrals in Definitions \ref{def:sdes} and \ref{def:intervened_SDE} are interpreted as It\^o integrals as given by \eqref{eqn:stochastic_integral}, defining each $X_v$ as a random variable on the underlying probability space $(\Omega, \Fcal, \PP)$. This `direct' dependence on $\omega$ of all variables conceals an independence structure that is convenient for causal inference.
The pathwise stochastic integral of \cite{karandikar1995pathwise} --- improved upon by \cite{przybylowicz2024skorohod} --- makes this independence structure explicit by replacing the $\omega$-dependent stochastic integral with an explicit functional dependence between sample paths. In particular, \cite{przybylowicz2024skorohod} (Appendix~\ref{app:przybylowicz}, Corollary~\ref{cor:pathwise_integral}) provides an adapted measurable map
\begin{equation*}
    \Psi : D([0,T], \RR^m) \times D([0,T], \RR^m) \to D([0,T], \RR),
\end{equation*}
with the property that for every càdlàg adapted process $G$ and every semimartingale $H \in \SSb([0,T], \RR^m)$, the process $\Psi(G, H)$ is a càdlàg version of the It\^o integral $\int_0^\cdot G(s{-}) \diff H(s)$.
This allows for a pathwise expression of the SDE, where all endogenous variables only depend on $\omega$ though the exogenous variables, and not directly.

\begin{definition}[Pathwise interpretation of systems of causal SDEs]\label{def:sde_pathwise}
    Given a system of causal SDEs $\Dcal$, for each variable $v \in V$ the SDE for $X_v$ can be represented by the equation
    \begin{equation*}
        X_v = \Phi_v(X_v, X_{\alpha(v)}, X_{\beta(v)}, X_{\gamma(v)}) \quad \text{$\PP$-a.s.}
    \end{equation*}
    where for all $t\in [0,T]$,
    \begin{equation*}
        \Phi_v(x_v, x_{\alpha(v)}, x_{\beta(v)}, x_{\gamma(v)})(t) := f_v\bigl(t, x_{\alpha(v)}\bigr) + \Psi\bigl(g_v(\cdot, x_v, x_{\beta(v)}), h_v(\cdot, x_{\gamma(v)})\bigr)(t).
    \end{equation*}
    For $O\subseteq V, x_V\in D([0,T], \RR)^{|V|}$ and $x_W\in D([0,T], \RR)^{|W|}$, write
    \begin{equation*}
        \Phi_O(x_V, x_W) := (\Phi_v(x_v, x_{\alpha(v)}, x_{\beta(v)}, x_{\gamma(v)}))_{v\in O}.
    \end{equation*}
\end{definition}
Although not strictly necessary for our purposes, this functional interpretation may guide conceptual understanding of how interventions propagate through the system, and it provides a notation $\Phi_v$ for each SDE that is useful in the proofs.
The pathwise SDE formulation also clarifies why the processes $X_{\alpha(v)}$ and $X_{\beta(v)}$ are allowed to be endogenous (and hence intervenable), and the integrator processes $X_{\gamma(v)}$ are required to be exogenous (and hence non-intervenable). Namely, the map $\Psi$ correctly represents the stochastic integral for all adapted measurable càdlàg integrand processes (so in particular also for all deterministic càdlàg values of the integrand), but it requires the integrator to be a semimartingale, and not every deterministic càdlàg sample path is a semimartingale -- this only holds when the path is of finite variation.

Under this pathwise reading, a system of causal SDEs is itself a (possibly cyclic) \emph{Structural Causal Model} (SCM) \citep{pearl2009causality,bongers2021foundations} on path space, i.e.\ a tuple $\angs{V, W, \Xcal_V, \Xcal_W, \Phi, \PP(X_W)}$ where $V$ and $W$ are disjoint finite index sets of endogenous variables and exogenous variables respectively, the domains $\Xcal_V = \prod_{i\in V}D([0,T], \RR)$ and $\Xcal_W = \prod_{i\in W}D([0,T], \RR)$ are products of Skorokhod spaces, the exogenous distribution $\PP(X_W) = \bigotimes_{w\in W}\PP(X_w)$ is the product of the laws of the exogenous processes, and for each $v\in V$ the \emph{causal mechanism} $\Phi_v : \Xcal_V \times \Xcal_W \to \Xcal_v$ yields the \emph{structural equation}
\begin{equation*}
    X_v = \Phi_v(X_V, X_W).
\end{equation*}
Such an SCM on path space is also known as a \emph{Dynamic Structural Causal Model} \citep{rubenstein2018deterministic,boeken2024dynamic}.
The causal graph $G(\Dcal)$ of Definition~\ref{def:sde_graph} is a supergraph of the graph of the SCM as defined in \cite{bongers2021foundations}, and the perfect intervention on the system of causal SDEs coincides with the corresponding perfect intervention on the SCM. The remainder of the paper leverages this pathwise SCM interpretation to prove various results that exist for SCMs (Markov properties, the do-calculus, constraint-based causal discovery) for systems of causal SDEs.

\section{Solvability of systems of causal SDEs}\label{sec:solvability}

The primary objects of causal inference are the observational and interventional distributions. If $\Dcal$ has a solution $X_V$ then the \emph{observational distribution} is the law $\PP(X_V)$, and if for $S\subseteq V$ and $x_S\in \Xcal_S$ (write $O := V\setminus S$) the system $\Dcal_{\Do(X_S=x_S)}$ has solution $X_O$ then the \emph{interventional distribution} is the law of $X_O$, denoted by $\PP(X_O\given \Do(X_S=x_S))$. However, these distributions are not always well-defined. To resolve this, \cite{bongers2021foundations} introduced for (possibly cyclic) SCMs solvability requirements such that the observational distribution and all interventional distributions are well-defined. However, these solvability requirements are too strong for our purposes, and not necessary to derive useful results. Instead, we consider the following solvability property of systems of causal SDEs.

\begin{definition}[Essential unique solvability]\label{def:ess_unique_solvable}
    Let $O \subseteq V$ and write $S := V \setminus O$. The system of causal SDEs $\Dcal$ is \emph{essentially uniquely solvable w.r.t.\ $O$} if there exists a measurable adapted function $I^{[O]} : \Xcal_S \times \Xcal_W \to \Xcal_O$ such that for every measurable adapted function $\varphi_S : \Xcal_W \to \Xcal_S$ the following two conditions hold:
    \begin{itemize}
        \item $I^{[O]}$ is a.s.\ a fixed point of the system of causal SDEs: we have $\PP(X_W)$-almost surely
        \begin{equation}\label{eqn:ess_unique_solvable}
            I^{[O]}(\varphi_S(X_W), X_W) = \Phi_O\bigl(I^{[O]}(\varphi_S(X_W), X_W), \varphi_S(X_W), X_W\bigr);
        \end{equation}
        \item $I^{[O]}$ is an a.s.\ unique fixed point, i.e., for every measurable adapted function $\varphi_O : \Xcal_W \to \Xcal_O$ satisfying $\varphi_O(X_W) = \Phi_O\bigl(\varphi_O(X_W), \varphi_S(X_W), X_W\bigr)$ $\PP(X_W)$-a.s., we have $\PP(X_W)$-a.s.
        \begin{equation*}
            \varphi_O(X_W) = I^{[O]}(\varphi_S(X_W), X_W).
        \end{equation*}
    \end{itemize}
    We refer to such a function $I^{[O]}$ as a \emph{solution function (w.r.t.\ $O$)}.
\end{definition}
A solution function may be constant on a large part of its domain: whenever $\Dcal$ is essentially uniquely solvable w.r.t.\ $O$, one can choose a solution function $I^{[O]}(x_{\pa(O)}, x_W)$ that depends on $X_S$ (with $S := V \setminus O$) only through the endogenous parents $\pa(O)$ of $O$ (Lemma~\ref{lem:acy_factorisation}).

This notion of essential unique solvability differs from the notion of \cite{bongers2021foundations}. We require existence of the fixed point outside of a null set that may depend on the input process $\varphi_S$, and our uniqueness condition holds outside of a null set that may depend on the alternative solution $\varphi_O$. In \cite{bongers2021foundations}, existence and uniqueness must hold simultaneously on a single, uniform null set, that only depends on $\Phi_O$ and not on $\varphi_S$ and $\varphi_O$. In this regard our notion of solvability is weaker, and this weakening is suitable for obtaining solution functions for a certain class of systems of causal SDEs, as shown in Section \ref{sec:solvability_sufficient_cond} (Theorem~\ref{thm:ito_map}).\footnote{Note that \cite{boeken2024dynamic} overlooked this nuance and assumed that the SDEs satisfy the solvability requirements of \cite{bongers2021foundations}.} Due to this weakening, we cannot rely on existing results for simple SCMs (uniqueness of interventional distributions, Markov property, do-calculus) but we prove these results from scratch.

First, we indeed have that essential unique solvability allows for clear expressions of the observational and interventional distributions.

\begin{restatable}{theorem}{obsintdist}\label{thm:obs_int_dist}
    Let $\Dcal$ be a system of causal SDEs. If $\Dcal$ is essentially uniquely solvable w.r.t.\ $V$ with solution function $I^{[V]} : \Xcal_W \to \Xcal_V$, then the observational distribution of $\Dcal$ is unique and satisfies
    \begin{equation*}
        \PP(X_V) = \PP(I^{[V]}(X_W)).
    \end{equation*}
    If $L, O, S\subseteq V$ partition $V$ and $\Dcal$ is essentially uniquely solvable w.r.t.\ $O \subseteq V$ with solution function $I^{[O]} : \Xcal_S \times \Xcal_L \times \Xcal_W \to \Xcal_O$, then for any $x_S \in \Xcal_S$ the intervened system $\Dcal_{\Do(X_S = x_S)}$ is essentially uniquely solvable w.r.t.\ $O$ with solution function $I^{[O]}(x_S, \cdot, \cdot): \Xcal_L \times \Xcal_W\to \Xcal_O$. If $L = \emptyset$, the interventional distribution satisfies
    \begin{equation*}
        \PP\bigl(X_O \given \Do(X_S = x_S)\bigr) = \PP(I^{[O]}(x_S, X_W)),
    \end{equation*}
    is unique, and $x_S \mapsto \PP(X_O\given \Do(X_S=x_S))$ is a Markov kernel.
\end{restatable}

\begin{example}\label{ex:linearised_repressilator}
    Consider a linearised version of the repressilator with dynamical model
    \begin{align*}
        \diff M_{i}(t) & = \left(\gamma - M_{i}(t)  - \lambda P_{j}(t)\right) \diff t + \sigma \diff Z_{i}^M(t) \\
        \diff P_{i}(t) & = \beta(M_{i}(t) - P_{i}(t)) \diff t + \sigma \diff Z_{i}^P(t)
    \end{align*}
    for each gene and its predecessor $(i,j) \in \{\textrm{(lacI, cI), (tetR, lacI), (cI, tetR), (GFP, tetR)}\}$. This system is for example essentially uniquely solvable with respect to $M_i$ for each $i$, with solution function
    \begin{align*}
        I^{[M_i]}(M_i^0, P_j, Z_i^M)(t) & := M_i^0 e^{-t} + \int_0^t e^{-(t-s)}\bigl(\gamma - \lambda P_j(s)\bigr) \diff s + \sigma \int_0^t e^{-(t-s)} \diff Z_i^M(s),
    \end{align*}
    with the integrals interpreted as pathwise maps $\Psi$. This solution function determines the causal effect of the endogenous parent variable $P_j$ on $M_i$: for any intervention path $x_{P_j} \in \Xcal_{P_j}$ the interventional distribution is written as $\PP\bigl(X_{M_i} \given \Do(X_{P_j} = x_{P_j})\bigr) = \PP(I^{[M_i]}(M_i^0, x_{P_j}, Z_i^M))$.
\end{example}

\subsection{Sufficient conditions for solvability}\label{sec:solvability_sufficient_cond}
We provide two classes of systems of causal SDEs that are essentially uniquely solvable with respect to every $O\subseteq V$. Given a measurable map $g:[0,T]\times D([0,T], \RR^k)\times D([0,T], \RR^\ell) \to D([0,T], \RR^m)$ consider the conditions that there exists a measurable càdlàg $K:[0,T]\times D([0,T], \RR^\ell)\to (0,\infty)$ such that for all $x, x_1, x_2 \in D([0,T], \RR^k)$, $y \in D([0,T], \RR^\ell)$, $t \in [0,T]$,
\begin{align}
    \label{eqn:linear_growth}
    \|g(t, x, y)\|                  & \leq K(t, y)\bigl(1 + \sup_{0\leq s\leq t}\|x(s)\|\bigr), \\
    \label{eqn:lipschitz}
    \|g(t, x_1, y) - g(t, x_2, y)\| & \leq K(t, y) \sup_{0\leq s\leq t}\|x_1(s) - x_2(s)\|.
\end{align}
We refer to \eqref{eqn:linear_growth} and \eqref{eqn:lipschitz} respectively as the \emph{linear-growth} and \emph{Lipschitz} conditions on $g$ in $x$ given $y$. These conditions can be verified component-wise, as shown in Lemma~\ref{lem:joint_vs_componentwise} in the appendix. Under these conditions, \cite{przybylowicz2024skorohod} (see also Appendix~\ref{app:przybylowicz}, Theorem~\ref{thm:przybylowicz_app}) prove the existence of a measurable, adapted solution function.\footnote{For SDEs driven by Brownian motion, similar results have been shown by \cite{yamada1971uniqueness} and \cite{kallenberg1996existence}, and for general semimartingale SDEs this has been shown by \cite{karandikar1995pathwise}, whose solution function was proven to be measurable with respect to Borel sets generated by a (non-Polish) topology that is strictly stronger than the Skorokhod topology.} Translated to systems of causal SDEs, their result implies the following.

\begin{restatable}{theorem}{itomap}\label{thm:ito_map}
    Let $\Dcal$ be a system of causal SDEs and let $O \subseteq V$. If for every $v\in O$ we have $\alpha(v) \cap O = \emptyset$ and the integrand $g_v$ satisfies the linear-growth and Lipschitz conditions in $(X_v, X_{\beta(v) \cap O})$ given $X_{\beta(v) \setminus O}$, then $\Dcal$ is essentially uniquely solvable w.r.t.\ $O$.
\end{restatable}
For simplicity, this sufficient condition assumes that there are no functional relations between the variables in $O$.
To obtain essential unique solvability, we apply Theorem~\ref{thm:ito_map} separately to each \emph{strongly connected component} (SCC) of $G(\Dcal)$ (subsets whose every two vertices are connected by a directed path) and compose the resulting solution functions along the topological order of SCCs; essential unique solvability w.r.t.\ each SCC of $G(\Dcal)$ then lifts to essential unique solvability w.r.t.\ $V$ (Lemma~\ref{lem:scc_composition} in the appendix). However, this does not imply essential unique solvability with respect to every subset $O'\subseteq V$. The following assumption allows for this construction.

\begin{restatable}[]{ass}{uniquelysolvable}\label{ass:uniquely_solvable}
    Consider a system of causal SDEs $\Dcal = \angs{V, W, X_W, f, g, h}$ such that
    \begin{enumerate}[label=(\roman*)]
        \item for every SCC $S$ of $G(\Dcal)$ and every $v \in S$, $\alpha(v) \cap S = \emptyset$, and
        \item for every $v \in V$, writing $S_v$ for the SCC of $G(\Dcal)$ containing $v$, $g_v$ satisfies the linear-growth and Lipschitz conditions in $(X_v, X_{\beta(v) \cap S_v})$ given $X_{\beta(v) \setminus S_v}$.
    \end{enumerate}
\end{restatable}
This assumption is closed under perfect interventions, so we obtain the following result.
\begin{restatable}[]{theorem}{esssolvableunderass}\label{thm:ess_solvable_under_ass}
    If a system of causal SDEs $\Dcal$ satisfies Assumption~\ref{ass:uniquely_solvable}, then it is essentially uniquely solvable w.r.t.\ every subset $O \subseteq V$.
\end{restatable}

For the $d$-separation Markov property (in Section \ref{sec:dsep_mp} below) we will need a more restricted model class:
\begin{restatable}[]{ass}{additivenoise}\label{ass:additive_noise}
    Consider the additive-noise causal SDE
    \begin{equation}\label{eqn:sde_additive}
        X_V(t) = X_\alpha^0 + \int_0^t \mu(s, X_V(s)) \diff s + \int_0^t \Sigma \diff X_\gamma(s),
    \end{equation}
    for the endogenous index set $V$ with $|V| =: d$ and exogenous index sets $W = \alpha\cup \gamma$ with $\alpha, \gamma$ disjoint such that $|\alpha| = |\gamma| = d$, with $v\mapsto \alpha(v)$ a bijection of $V$ onto $\alpha$, where $X_\alpha^0$ is a vector of mutually independent exogenous random variables satisfying $\EE[\|X_\alpha^0\|^2] < \infty$, and $X_\gamma$ is a $d$-dimensional standard Brownian motion with independent components, independent of $X_\alpha^0$.
    The drift $\mu: [0,T] \times \RR^d \to \RR^d$ has $t\mapsto \mu(t, x_V(t))$ continuous in $t$ and satisfies the linear-growth and Lipschitz conditions in $x_V$, and $\Sigma\in \RR^{d\times d}$ is an invertible matrix.
\end{restatable}

\begin{restatable}[]{theorem}{esssolvableunderadditive}\label{thm:ess_solvable_under_additive}
    If a system of causal SDEs $\Dcal$ satisfies Assumption~\ref{ass:additive_noise}, then it satisfies Assumption~\ref{ass:uniquely_solvable}.
\end{restatable}

We now show that the repressilator satisfies Assumption~\ref{ass:additive_noise}.
\begin{example}[Repressilator is essentially uniquely solvable]\label{ex:repressilator_solvable}
    We may write the repressilator in vector form as
    \begin{equation*}
        \diff X(t) = \mu(X(t))\diff t + \sigma\diff Z(t)
    \end{equation*}
    with drift
    \begin{align*}
        \mu(X) =
        \begin{pmatrix}
            \frac{\alpha}{1 + |P_{cI}|^n} + \alpha_0 - M_{lacI}   \\
            \beta (M_{lacI} - P_{lacI})                           \\
            \frac{\alpha}{1 + |P_{lacI}|^n} + \alpha_0 - M_{tetR} \\
            \beta (M_{tetR} - P_{tetR})                           \\
            \frac{\alpha}{1 + |P_{tetR}|^n} + \alpha_0 - M_{cI}   \\
            \beta (M_{cI} - P_{cI})                               \\
            \frac{\alpha}{1 + |P_{tetR}|^n} + \alpha_0 - M_{GFP}  \\
            \beta (M_{GFP} - P_{GFP})
        \end{pmatrix},
        \quad
        X =
        \begin{pmatrix}
            M_{lacI} \\
            P_{lacI} \\
            M_{tetR} \\
            P_{tetR} \\
            M_{cI}   \\
            P_{cI}   \\
            M_{GFP}  \\
            P_{GFP}
        \end{pmatrix}
        \quad\text{and}\quad
        Z =
        \begin{pmatrix}
            Z^M_{lacI} \\
            Z^P_{lacI} \\
            Z^M_{tetR} \\
            Z^P_{tetR} \\
            Z^M_{cI}   \\
            Z^{P}_{cI} \\
            Z^M_{GFP}  \\
            Z^P_{GFP}
        \end{pmatrix}.
    \end{align*}
    We verify the linear-growth and Lipschitz conditions on $\mu$ componentwise. For each $M_i$ component (i.e.\ rows 1, 3, 5, 7) we have $\left| \frac{\alpha}{1+|P_j(t)|^n}+\alpha_0-M_i\right| \leq \max\{\alpha + \alpha_0, 1\}(1+\|(P_j(t), M_i(t))\|)$, so it satisfies the linear-growth condition in $(P_j, M_i)$, and the function $p \mapsto \frac{\alpha}{1+|p|^n}$ has Lipschitz constant upper bounded by $\alpha n$, so one readily verifies that the component is $\sqrt{2}\max\{\alpha n, 1\}$-Lipschitz in $(P_j, M_i)$. Each $P_i$ component is linear in $(M_i, P_i)$ and so satisfies both the linear-growth condition and the Lipschitz condition in $(M_i, P_i)$. Hence $\mu$ satisfies the linear-growth and Lipschitz conditions jointly. Moreover, the diffusion matrix $\Sigma = \sigma I_8$ is invertible, with inverse $\sigma^{-1} I_8$. If the initial condition $X^0$ is a vector of mutually independent variables with $\EE[\|X^0\|^2] < \infty$, independent of the driving Brownian motion $Z$, then the repressilator satisfies Assumption~\ref{ass:additive_noise}, and by Theorems \ref{thm:ess_solvable_under_ass} and \ref{thm:ess_solvable_under_additive} it is essentially uniquely solvable w.r.t.\ every $O\subseteq V$.
\end{example}

\section{Marginalisation}\label{sec:marginalisation}

When a subset $L \subseteq V$ of variables is not of direct interest or unobserved, one may want to reason about $\Dcal$ on $O := V \setminus L$. The marginalisation of a system of causal SDEs substitutes a solution function for $L$ into the SDE coefficients for $V \setminus L$.

\begin{definition}[Marginalisation]\label{def:marginalisation}
    Let $L\subseteq V$ and let $\Dcal = \angs{V, W, X_W, f, g, h}$ be a system of causal SDEs that is essentially uniquely solvable w.r.t.\ $L \subseteq V$ with solution function $I^{[L]}: \Xcal_{O} \times \Xcal_W \to \Xcal_L$, writing $O := V \setminus L$. The \emph{marginalised system of causal SDEs} $\Dcal_{\marg(L)} := \angs{O, W, X_W, \tilde f, \tilde g, \tilde h}$ has for each $v \in O$ an SDE of the form \eqref{eqn:sde} given by the following parameters:
    \begin{align*}
        \tilde f_v(t, x_O, x_W)         & := f_v\bigl(t, x_{\alpha(v) \cap O}, I^{[L]}_{\alpha(v) \cap L}(x_O, x_W), x_{\alpha(v) \cap W}\bigr),      \\
        \tilde g_v(s{-}, x_v, x_O, x_W) & := g_v\bigl(s{-}, x_v, x_{\beta(v) \cap O}, I^{[L]}_{\beta(v) \cap L}(x_O, x_W), x_{\beta(v) \cap W}\bigr), \\
        \tilde h_v(s, x_W)              & := h_v\bigl(s, x_{\gamma(v)}\bigr).
    \end{align*}
\end{definition}

Since adapted solution functions $I^{[L]}$ are not unique, the marginalised system $\Dcal_{\marg(L)}$ is itself not unique. The induced distributions are nonetheless unambiguous: whenever $\Dcal$ is essentially uniquely solvable w.r.t.\ $L$ and $V$, the observational distributions on $O := V \setminus L$ coincide, $\PP_{\Dcal}(X_O) = \PP_{\Dcal_{\marg(L)}}(X_O)$; and more generally, for $O' \subseteq O$ with $S := O \setminus O'$, whenever $\Dcal$ is essentially uniquely solvable w.r.t.\ $L$ and $L \cup O'$, the interventional distributions coincide, $\PP_{\Dcal}(X_{O'} \given \Do(X_S)) = \PP_{\Dcal_{\marg(L)}}(X_{O'} \given \Do(X_S))$ (Lemma~\ref{lem:marg_well_defined}).
Marginalisation is therefore a powerful notion of abstraction: (causal) inference on the observed part of a system remains valid in the underlying system. It is moreover compatible with itself and with intervention: sequentially marginalising over two disjoint subsets $L_1, L_2 \subseteq V$ yields the same observational and interventional distributions regardless of the order (Lemma~\ref{lem:marginalisation_steps}), and marginalisation commutes with intervention (Lemma~\ref{lem:intervention_marg_commute}). These results rest on the following closure property: marginalising over $L$ preserves essential unique solvability, in that solvability of $\Dcal$ w.r.t.\ $L$ and $L \cup O'$ transfers to solvability of $\Dcal_{\marg(L)}$ w.r.t.\ $O'$.

\begin{restatable}{theorem}{margclosuresimple}\label{thm:marg_closure_simple}
    Let $\Dcal$ be a system of causal SDEs, and let $L, O \subseteq V$ be disjoint. If $\Dcal$ is essentially uniquely solvable w.r.t.\ $L$ and w.r.t.\ $L \cup O$, then the marginalised system $\Dcal_{\marg(L)}$ is essentially uniquely solvable w.r.t.\ $O$. Moreover, for any solution function $I^{[L \cup O]}$ of $\Dcal$ w.r.t.\ $L \cup O$, its restriction $I^{[L \cup O]}_{O}$ is a solution function of $\Dcal_{\marg(L)}$ w.r.t.\ $O$.
\end{restatable}

Any system of causal SDEs satisfying Assumption~\ref{ass:uniquely_solvable} or Assumption~\ref{ass:additive_noise} is essentially uniquely solvable w.r.t.\ every $O \subseteq V$ (Theorems~\ref{thm:ess_solvable_under_ass} and~\ref{thm:ess_solvable_under_additive}). By Theorem~\ref{thm:marg_closure_simple}, any such system therefore remains essentially uniquely solvable w.r.t.\ every set of endogenous variables after any marginalisation.

\section{Graphical Markov properties}\label{sec:markov_properties}
A cornerstone of reasoning with causal graphs is the Markov property: that a $d$-separation or $\sigma$-separation in the graph implies a conditional independence in the distribution. The Markov property is the foundation of the do-calculus and constraint-based causal discovery algorithms, as will be treated in Sections \ref{sec:causal_inference} and \ref{sec:constraint_based_cd}.

Formally, a walk $\pi$ in a DMG $G =(V, E, L)$, is \emph{$d$-blocked} by $C\subseteq V$ if it has an end-point or a non-collider in $C$, or if there is a collider which is not in $\Anc(C)$. For sets of nodes $A, B, C\subseteq V$, we call $A$ and $B$ \emph{$d$-separated} given $C$, written $A\Perp_G^d B \given C$, if all walks between $A$ and $B$ are $d$-blocked by $C$.
The $d$-separation Markov property then means that for all $A, B, C \subseteq V$, we have the implication
\begin{equation*}
    A\Perp^d_G B\given C \implies X_A \Indep X_B \given X_C.
\end{equation*}
This is known to hold for acyclic SCMs, as well as for certain cyclic SCMs: if all variables are discrete and the SCM is ancestrally uniquely solvable \citep{pearl1996identifying,neal2000deducing,forre2017markov}, or if the structural equations are linear and depend on an exogenous variable whose distribution has a density with respect to Lebesgue measure \citep{spirtes1994conditional,forre2017markov}, see also \cite{bongers2021foundations}. In the next section, we will see that a class of cyclic additive-noise SDEs can be added to this list.

It is known that there are cyclic SCMs for which the $d$-separation Markov property does not hold, for example for certain nonlinear cyclic SCMs with Gaussian noise \citep{spirtes1994conditional,spirtes1995directed}. A more suitable separation criterion for cyclic SCMs is the following notion of \emph{$\sigma$-separation} \citep{forre2017markov,bongers2021foundations,forre2025mathematical}.
Given a DMG $G=(V, E)$, a set of nodes $C\subseteq V$ and a walk $\pi$ in DMG $G$:
\begin{itemize}
    \item a non-collider $v$ is called \emph{blockable} if it points towards a neighbouring node on the walk that is not in the same strongly connected component as $v$,
    \item the walk $\pi$ is called \emph{$\sigma$-blocked by $C$} if it has an end-point or a blockable non-collider in $C$, or if there is a collider on $\pi$ that is not in $\Anc(C)$.
\end{itemize}
For sets of nodes $A, B, C\subseteq V$, we call $A$ and $B$ \emph{$\sigma$-separated} given $C$, written $A\Perp_G^\sigma B \given C$, if all walks between $A$ and $B$ are $\sigma$-blocked by $C$.
For general DMGs, $\sigma$-separation implies $d$-separation, and for acyclic DMGs the two notions coincide (since every non-collider is automatically blockable).
Self-loops $v\to v$ do not affect $d$- or $\sigma$-separations (Theorem~\ref{thm:self_cyclic_graph_separations}); we may therefore ignore them when reading off separations.

\subsection{A \texorpdfstring{$\sigma$}{sigma}-separation Markov property for systems of causal SDEs}\label{sec:sigmasep_mp}

The $\sigma$-separation Markov property was established for simple SCMs by \cite{forre2017markov,bongers2021foundations,forre2025mathematical}. Here we extend this result to systems of causal SDEs that are essentially uniquely solvable w.r.t.\ every SCC of $G(\Dcal)$, following the same proof strategy: we construct an \emph{acyclification} by replacing each SCC's SDE with its solution function, show that this acyclified system is observationally equivalent to the original and has a graph that is a subgraph of any graphical acyclification, and then invoke the $d$-separation Markov property for acyclic SCMs \cite[Theorem 6.3]{bongers2021foundations}. The full proof is given in the appendix.

\begin{restatable}[]{theorem}{sigmasepmarkov}\label{thm:sigma_sep_mp}
    Let $\Dcal$ be a system of causal SDEs that is essentially uniquely solvable w.r.t.\ every SCC of $G(\Dcal)$. For all $A, B, C \subseteq V\cup W$ we have
    \begin{equation}\label{eqn:sigmasep_mp}
        A \Perp_{G^+(\Dcal)}^\sigma B \given C \implies X_A \Indep X_B \given X_C.
    \end{equation}
\end{restatable}

For sets $A, B, C \subseteq V$ restricted to endogenous variables, this implies the $\sigma$-separation Markov property in the (non-augmented) causal graph $G(\Dcal)$, via the correspondence between $\sigma$-separation in $G(\Dcal)$ and $G^+(\Dcal)$.
We call $\Dcal$ $\sigma$-\emph{faithful} if the reverse implication of \eqref{eqn:sigmasep_mp} holds; this need not hold in general.
\begin{example}[Graphical separation in the repressilator]\label{ex:repressilator_separation}
    Consider the cyclic graph in Figure \ref{fig:repressilator_graph}. In this graph we have the $d$-separation $P_{lacI} \Perp^d P_{tetR} \given M_{tetR}, M_{lacI}$, but a $\sigma$-connection $P_{lacI} \nPerp^\sigma P_{tetR} \given M_{tetR}, M_{lacI}$: on the walk $P_{lacI} \to M_{tetR} \to P_{tetR}$ the non-collider $M_{tetR}$ points to $P_{tetR}$, which lies in the same strongly connected component $\{M_{lacI}, P_{lacI}, M_{tetR}, P_{tetR}, M_{cI}, P_{cI}\}$, so $M_{tetR}$ is not blockable and conditioning on it does not block the walk. Similarly, we have a $\sigma$-connection $P_{lacI} \nPerp^\sigma P_{GFP} \given M_{tetR}, M_{lacI}$, and we have a $\sigma$-separation $P_{lacI} \Perp^\sigma P_{GFP} \given P_{tetR}$.

    For essentially uniquely solvable cyclic systems of causal SDEs the $\sigma$-separation Markov property holds (Theorem~\ref{thm:sigma_sep_mp}), but the question remains: can this be strengthened to a $d$-separation Markov property? Do we have a conditional independence $P_{lacI} \Indep P_{tetR} \given M_{tetR}, M_{lacI}$?
\end{example}

\subsection{A d-separation Markov property for additive-noise SDEs}\label{sec:dsep_mp}
In this section we will show that the $d$-separation Markov property holds for (possibly cyclic) additive noise SDEs satisfying Assumption~\ref{ass:additive_noise}.
Our proof technique relies on a time-discretised process for which the $d$-separation Markov property holds by acyclicity. By showing that the law of the discretised process -- suitably embedded in the Skorokhod space $D([0,T], \RR^d)$ -- converges in total variation to the law of the original process as a probability measure on $D([0,T], \RR^d)$, conditional independence in the time-discretised process implies a conditional independence in the solution of the original process.

\bigskip
We consider the following discrete-time Euler-Maruyama-type scheme, with initial value $X_V^\Delta(0) = X_\alpha^0$ and
\begin{equation*}
    X_V^\Delta(t_{k+1}) = X_V^\Delta(t_{k}) + \mu(t_k, X_V^\Delta(t_k))(t_{k+1}-t_k) + \Sigma\bigl(X_\gamma(t_{k+1}) - X_\gamma(t_k)\bigr)
\end{equation*}
with $t_k := Tk/n$, $n\in\NN$, $k=0, ..., n$ and independent increment processes $X_\gamma^{[t_{k}, t_{k+1}]} := X_\gamma - X_\gamma(t_k)$ considered as a random variable in $D([t_k, t_{k+1}], \RR^d)$. This induces a discrete-time \emph{Euler SCM} $\Mcal_\Dcal^\Delta$ with endogenous variables $X_v^\Delta(t_0), ..., X_v^\Delta(t_n)$ for all $v\in V$ and exogenous independent variables $X_w^0$ for all $w\in\alpha$ and $X_w^{[0, t_{1}]}, ..., X_w^{[t_{n-1}, t_{n}]}$ for all $w\in \gamma$ \citep{hansen2014causal}.
We refer to the graph $G(\Mcal_\Dcal^\Delta)$ of $\Mcal_\Dcal^\Delta$ as the \emph{Euler graph}, see Example \ref{ex:euler_graph} below.
In contrast, we call the causal graph $G(\Dcal)$ the \emph{summary graph}, since each of its vertices represents an entire process $X_v$ on $[0,T]$ rather than the individual time-points $X_v^\Delta(t_k)$ as displayed by the Euler graph (or the subinterval processes of the time-split graph in Section~\ref{sec:time-evaluations}).
An important observation is that the Euler graph is acyclic, so the Euler SCM satisfies the $d$-separation Markov property.

The following lemma relates the causal graph to the Euler graph. It is an adaptation of \cite{ferreira2024identifying}, who consider discrete-time stochastic processes with a corresponding summary graph.

\begin{restatable}[]{lem}{dsepsummaryeuler}\label{thm:dsep_summary_euler}
    Let $\Dcal$ be the SDE from Assumption~\ref{ass:additive_noise} and let $\Mcal_\Dcal^\Delta$ be the Euler SCM (for any $n\in\NN$), then for $A, B, C\subseteq V$ we have
    \begin{equation*}
        X_A \Perp_{G(\Dcal)}^d X_B \given X_C \implies X_A^\Delta \Perp_{G(\Mcal_\Dcal^\Delta)}^d X_B^{\Delta} \given X_C^{\Delta} \implies X_A^\Delta \Indep X_B^\Delta \given X_C^\Delta
    \end{equation*}
    where we write $X_A^\Delta := (X_A^\Delta(t_0), ..., X_A^\Delta(t_n))$.
\end{restatable}

\begin{example}[Euler graph of repressilator]\label{ex:euler_graph}
    Considering only the variables $P_{lacI}$, $M_{tetR}$, $P_{tetR}$ and $M_{lacI}$ of the repressilator, the causal graph of the marginalised SDE $\Dcal' := \Dcal_{\marg(M_{cI}, P_{cI}, M_{GFP}, P_{GFP})}$ is given in Figure \ref{fig:repressilator_graph_euler_summary}, from which we read off the $d$-separation $P_{lacI} \Perp^d P_{tetR} \given M_{tetR}, M_{lacI}$. The Euler graph is depicted in Figure \ref{fig:repressilator_graph_euler_unrolled}, where we observe that $P_{lacI}^\Delta \Perp^d P_{tetR}^\Delta \given M_{tetR}^\Delta, M_{lacI}^\Delta$. For example, the displayed paths from $P_{lacI}^\Delta(t_0)$ to $P_{tetR}^\Delta(t_2)$ and from $P_{tetR}^\Delta(t_2)$ to $P_{lacI}^\Delta(t_4)$ are blocked by $M_{tetR}^\Delta(t_1)$ and $M_{lacI}^\Delta(t_3)$ respectively. Since this discrete-time SCM is acyclic, we have $P_{lacI}^\Delta \Indep P_{tetR}^\Delta \given M_{tetR}^\Delta, M_{lacI}^\Delta$.
\end{example}
\begin{figure}[!htb]
    \centering
    \begin{subfigure}{.45\linewidth}
        \centering
        \begin{tikzpicture}[scale=1.6]
            \node[] (X1) at (-.75,0) {$M_{lacI}$};
            \node[] (Y1) at (0,.75) {$P_{lacI}$};
            \node[] (X2) at (.75,0) {$M_{tetR}$};
            \node[] (Y2) at (0,-.75) {$P_{tetR}$};
            \draw[arr] (X1) to (Y1);
            \draw[arr] (Y1) to (X2);
            \draw[arr] (X2) to (Y2);
            \draw[arr] (Y2) to (X1);
            \draw[arr] (X1) to [loop arc rev={180}{60}{0.6cm}] (X1);
            \draw[arr] (Y1) to [loop arc={90}{60}{0.6cm}] (Y1);
            \draw[arr] (X2) to [loop arc={0}{60}{0.6cm}] (X2);
            \draw[arr] (Y2) to [loop arc rev={-90}{60}{0.6cm}] (Y2);
        \end{tikzpicture}
        \caption{The causal graph $G(\Dcal')$.}
        \label{fig:repressilator_graph_euler_summary}
    \end{subfigure}
    \begin{subfigure}{.45\linewidth}
        \centering
        \begin{tikzpicture}
	\node[] (A) at (.3,0) {$P_{lacI}^\Delta$};
	\node[] (C1) at (.3,-1) {$M_{tetR}^\Delta$};
	\node[] (B) at (.3,-2) {$P_{tetR}^\Delta$};
	\node[] (C2) at (.3,-3) {$M_{lacI}^\Delta$};
	\node[] () at (1,.6) {$t_0$};
	\node[] () at (2,.6) {$t_1$};
	\node[] () at (3,.6) {$t_2$};
	\node[] () at (4,.6) {$t_3$};
	\node[] () at (5,.6) {$t_4$};
	\node[] () at (6,.6) {$t_5$};
	\node[] (A1) at (1,0) {$\cdot$};
	\node[] (A2) at (2,0) {$\cdot$};
	\node[] (A3) at (3,0) {$\cdot$};
	\node[] (A4) at (4,0) {$\cdot$};
	\node[] (A5) at (5,0) {$\cdot$};
	\node[] (A6) at (6,0) {$\cdot$};
	\node[] (C11) at (1,-1) {$\cdot$};
	\node[] (C12) at (2,-1) {$\cdot$};
	\node[] (C13) at (3,-1) {$\cdot$};
	\node[] (C14) at (4,-1) {$\cdot$};
	\node[] (C15) at (5,-1) {$\cdot$};
	\node[] (C16) at (6,-1) {$\cdot$};
	\node[] (B1) at (1,-2) {$\cdot$};
	\node[] (B2) at (2,-2) {$\cdot$};
	\node[] (B3) at (3,-2) {$\cdot$};
	\node[] (B4) at (4,-2) {$\cdot$};
	\node[] (B5) at (5,-2) {$\cdot$};
	\node[] (B6) at (6,-2) {$\cdot$};
	\node[] (C21) at (1,-3) {$\cdot$};
	\node[] (C22) at (2,-3) {$\cdot$};
	\node[] (C23) at (3,-3) {$\cdot$};
	\node[] (C24) at (4,-3) {$\cdot$};
	\node[] (C25) at (5,-3) {$\cdot$};
	\node[] (C26) at (6,-3) {$\cdot$};
	\node[color=gray!50] () at (6.5,0) {...};
	\node[color=gray!50] () at (6.5,-1) {...};
	\node[color=gray!50] () at (6.5,-2) {...};
	\node[color=gray!50] () at (6.5,-3) {...};
	\draw[arr,color=gray!50] (A1) to (A2);
	\draw[arr,color=gray!50] (A2) to (A3);
	\draw[arr,color=gray!50] (A3) to (A4);
	\draw[arr,color=gray!50] (A4) to (A5);
	\draw[arr,color=gray!50] (A5) to (A6);
	\draw[arr,color=gray!50] (B1) to (B2);
	\draw[arr,color=gray!50] (B2) to (B3);
	\draw[arr,color=gray!50] (B3) to (B4);
	\draw[arr,color=gray!50] (B4) to (B5);
	\draw[arr,color=gray!50] (B5) to (B6);
	\draw[arr,color=gray!50] (C11) to (C12);
	\draw[arr,color=gray!50] (C12) to (C13);
	\draw[arr,color=gray!50] (C13) to (C14);
	\draw[arr,color=gray!50] (C14) to (C15);
	\draw[arr,color=gray!50] (C15) to (C16);
	\draw[arr,color=gray!50] (C21) to (C22);
	\draw[arr,color=gray!50] (C22) to (C23);
	\draw[arr,color=gray!50] (C23) to (C24);
	\draw[arr,color=gray!50] (C24) to (C25);
	\draw[arr,color=gray!50] (C25) to (C26);
	\draw[arr,color=gray!50] (A2) to (C13);
	\draw[arr,color=gray!50] (A3) to (C14);
	\draw[arr,color=gray!50] (A4) to (C15);
	\draw[arr,color=gray!50] (A5) to (C16);
	\draw[arr,color=gray!50] (B1) to (C22);
	\draw[arr,color=gray!50] (B2) to (C23);
	\draw[arr,color=gray!50] (B4) to (C25);
	\draw[arr,color=gray!50] (B5) to (C26);
	\draw[arr,color=gray!50] (C11) to (B2);
	\draw[arr,color=gray!50] (C13) to (B4);
	\draw[arr,color=gray!50] (C14) to (B5);
	\draw[arr,color=gray!50] (C15) to (B6);
	\draw[arr,color=gray!50] (C21) to (A2);
	\draw[arr,color=gray!50] (C22) to (A3);
	\draw[arr,color=gray!50] (C23) to (A4);
	\draw[arr,color=gray!50] (C25) to (A6);
	\draw[arr] (A1) to (C12);
	\draw[arr] (C12) to (B3);
	\draw[arr] (B3) to (C24);
	\draw[arr] (C24) to (A5);
\end{tikzpicture}
        \caption{The Euler graph $G(\Mcal_{\Dcal'}^\Delta)$.}
        \label{fig:repressilator_graph_euler_unrolled}
    \end{subfigure}
    \caption{The marginalisation of the repressilator onto $\{P_{lacI}, M_{tetR}, P_{tetR}, M_{lacI}\}$, showcasing that $P_{lacI} \Perp^d P_{tetR} \given M_{tetR}, M_{lacI}$ implies $P_{lacI}^\Delta \Perp^d P_{tetR}^\Delta \given M_{tetR}^\Delta, M_{lacI}^\Delta$.}
    \label{fig:repressilator_graph_euler}
\end{figure}
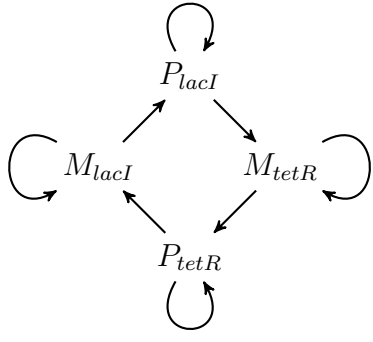
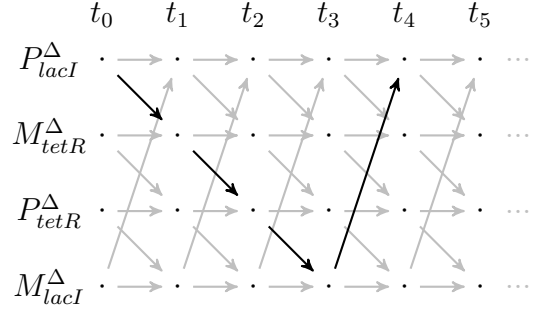

It follows from Lemma~\ref{thm:dsep_summary_euler} that the $d$-separation Markov property holds for any discrete-time SCM which does not have instantaneous cycles, similar to results shown by \cite{ferreira2024identifying}, \cite{niemiro2024causal} and \cite{reiter2024causal}.

\bigskip
To obtain a $d$-separation Markov property for the SDE, we aim to deduce the conditional independence $X_A\Indep X_B \given X_C$ from the conditional independence $X_A^\Delta \Indep X_B^\Delta \given X_C^\Delta$. To do so, we first extend the discrete-time process $X_V^\Delta$ to a continuous-time process $X_V^n$, for which we show that $X_A^n \Indep X_B^n \given X_C^n$ holds as well. Finally, we show that the distribution $\PP(X_V^n)$ converges in total variation to $\PP(X_V)$, which by \cite{lauritzen2024total} implies the desired conditional independence $X_A \Indep X_B \given X_C$.

Obtaining total variation convergence of the law of the Euler scheme to the law of the solution of the SDE is not straightforward. For example, if we extend $X_V^\Delta$ to continuous time by setting $X^n(t) := X^\Delta(t_k)$ for $t\in[t_k, t_{k+1})$, then $\PP(X_V^n)$ will never converge in total variation to the law of the SDE since the set of continuous paths is a Borel measurable set in the Skorokhod space, giving total variation distance $d_{TV}(X_V^\Delta, X_V^n) = 1$. To mitigate this problem we consider for $t\in[t_k, t_{k+1})$ the continuously interpolated Euler scheme
\begin{equation}\label{eqn:euler_continuous}
    X^n_V(t) = X^n_V(t_{k}) + \mu(t_k, X^n_V(t_k))(t-t_k) + \Sigma\bigl(X_\gamma(t) - X_\gamma(t_k)\bigr).
\end{equation}
This Euler approximation has the same noise structure as the original SDE, which ensures that the processes both have a density with respect to a (transformed) Wiener measure on $D([0,T], \RR^{d})$ via Girsanov's theorem. Convergence of the density of $\PP(X^n_V)$ to the density of $\PP(X_V)$ finally implies the desired total variation convergence, via Scheffé's Theorem \citep{scheffe1947useful}.

Having introduced a suitable continuous-time extension of the Euler scheme, we show that it inherits the conditional independencies implied by $d$-separation in the causal graph $G(\Dcal)$. The proof is in the appendix, but relies on the insight that one can write $X_V^n$ as a function of $X_V^\Delta$ and a vector $\tilde X_V$ of Brownian bridges, with $X_V^\Delta\Indep \tilde X_V$ (Lemma~\ref{lem:euler_decomposition}). By Lemma~\ref{thm:dsep_summary_euler} and the $d$-separation Markov property we have $X_A^\Delta \Indep X_B^\Delta \given X_C^\Delta$, and by Lemma~\ref{lem:mp_brownian_bridge} we have $\tilde X_A\Indep \tilde X_B\given \tilde X_C$. Combining these conditional independencies then gives the following result:
\begin{restatable}[]{theorem}{continuouseulerdsepmp}\label{thm:continuous_euler_dsep_mp}
    For all $A, B, C \subseteq V$, the continuous Euler scheme satisfies
    \begin{equation*}
        X_A \Perp_{G(\Dcal)}^d X_B \given X_C \implies X_A^n \Indep X_B^n \given X_C^n.
    \end{equation*}
\end{restatable}

The following result gives the desired total variation convergence of the Euler scheme.

\begin{restatable}[]{theorem}{eulertvconvergence}\label{thm:euler_tv_convergence}
    Let $\Dcal$ satisfy Assumption~\ref{ass:additive_noise}, then there is a subsequence $n_m$ such that $\PP(X^{n_m}_V)$ converges in total variation to $\PP(X_V)$.
\end{restatable}

Finally, the conditional independence $X_A^n \Indep X_B^n \given X_C^n$ implies the desired conditional independence $X_A \Indep X_B \given X_C$ by the following result.
\begin{restatable}[\cite{lauritzen2024total}]{theorem}{citvclosed}\label{thm:ci_tv_closed}
    Given a sequence of random variables $(X_A^n, X_B^n, X_C^n)_{n\in\NN}$ such that $X_A^n\Indep X_B^n \given X_C^n$ for all $n\in\NN$ and given a random variable $(X_A, X_B, X_C)$ such that $\PP(X_A^n, X_B^n, X_C^n)$ converges in total variation to $\PP(X_A, X_B, X_C)$, we have $X_A\Indep X_B\given X_C$.
\end{restatable}

We now state the main result of this section, that the additive-noise SDE satisfies the $d$-separation Markov property.
\begin{theorem}\label{thm:dsep_mp}
    Let the system of causal SDEs $\Dcal$ satisfy Assumption~\ref{ass:additive_noise}. Then for all $A, B, C \subseteq V$,
    \begin{equation*}
        X_A \Perp_{G(\Dcal)}^d X_B \given X_C \implies X_A \Indep X_B \given X_C.
    \end{equation*}
\end{theorem}
\begin{proof}
    Let $A \Perp_{G(\Dcal)}^d B \given C$, then by Theorem~\ref{thm:continuous_euler_dsep_mp} the continuous Euler scheme satisfies $X_A^n \Indep X_B^n \given X_C^n$. By Theorem~\ref{thm:euler_tv_convergence} there is a subsequence $n_m$ such that $\PP(X_{A, B, C}^{n_m}) \tvto \PP(X_{A, B, C})$. By Theorem~\ref{thm:ci_tv_closed} this implies $X_A \Indep X_B \given X_C$.
\end{proof}

Theorem~\ref{thm:dsep_mp} holds for $A, B, C\subseteq V$, so for conditional independencies between endogenous variables. While Theorem~\ref{thm:continuous_euler_dsep_mp} (and the underlying lemmas) extend to $A, B, C\subseteq V\cup W$ on the augmented causal graph $G^+(\Dcal)$, the total variation convergence of Theorem~\ref{thm:euler_tv_convergence} does not extend to the joint law of endogenous and exogenous variables, since the measures $\PP(X_V^n, X_W)$ and $\PP(X_V, X_W)$ are mutually singular and hence have total variation distance 1. Hence Theorem~\ref{thm:dsep_mp} is not straightforwardly extended to the augmented graph. For conditional independencies involving exogenous variables, the $\sigma$-separation Markov property (Theorem~\ref{thm:sigma_sep_mp}) still applies.



\begin{example}[$d$-separation Markov property]
    By Theorem~\ref{thm:dsep_mp}, the system of causal SDEs satisfies the $d$-separation Markov property, so the $d$-separation $P_{lacI} \Perp^d P_{tetR} \given M_{tetR}, M_{lacI}$ in the graph from Example \ref{ex:repressilator_separation} implies $P_{lacI} \Indep P_{tetR} \given M_{tetR}, M_{lacI}$.
\end{example}

\section{Do-calculus}\label{sec:causal_inference}

A question central to causal inference is whether a causal effect (an interventional distribution) can be identified from the observational distribution.
When the underlying causal graph is known, such questions can be answered by the rules of do-calculus. We employ a formulation that uses the notion of a causal graph with \emph{intervention variables} \citep{spirtes1993causation,pearl1993comment,forre2020causal,dawid2021decisiontheoretic}, that is equivalent to the well-known formulation of the do-calculus in terms of \emph{mutilated graphs} \citep{pearl1995causal,pearl2009causality}.

\begin{definition}[Graph with intervention variables]\label{def:intervention_variables_graph}
    Given a DMG $G = (V, E, L)$ and intervention target $S\subseteq V$, let $G_{\Do(R_S)}$ be the graph $G$ appended with the vertex $R_v$ and edge $R_v \to v$ for each $v\in S$.
\end{definition}
More details on \emph{systems of causal SDEs with intervention variables} are given in the appendix (cf.\ Definition \ref{def:intervention_variables}).
The following formulation is inspired by \cite{forre2025mathematical}, Theorem 5.1.2. In the main text we only focus on the graphical aspect of the intervention variables, as this is sufficient for formulating the do-calculus.

\begin{restatable}[]{theorem}{docalculus}\label{thm:do_calculus}
    Let $\Dcal$ be a system of causal SDEs that is essentially uniquely solvable w.r.t.\ every $O \subseteq V$, and let $A, B, C \subseteq V$ be pairwise disjoint. Then:
    \begin{enumerate}
        \item \emph{Insertion/deletion of observation.} If $A \Perp^\sigma_{G(\Dcal)} B \given C$, then
        \begin{equation*}
            \PP(X_A \given X_B, X_C) = \PP(X_A \given X_C) \quad \text{ $\PP(X_B, X_C)$-a.s.}
        \end{equation*}
        \item \emph{Action/observation exchange.} If $A \Perp^\sigma_{G(\Dcal)_{\Do(R_B)}} R_B \given B \cup C$ and $\PP(X_C \mid X_B = x_B) \ll \PP(X_C \mid \Do(X_B = x_B))$ for $\PP(X_B)\text{-a.a.\ } x_B$, then
        \begin{equation*}
            \PP(X_A \given \Do(X_B), X_C) = \PP(X_A \given X_B, X_C) \quad \text{  $\PP(X_B, X_C)$-a.s.}
        \end{equation*}
        \item \emph{Insertion/deletion of action.} If $A \Perp^\sigma_{G(\Dcal)_{\Do(R_B)}} R_B \given C$, then for every $x_B \in \Xcal_B$ such that $\PP(X_C) \ll \PP(X_C \mid \Do(X_B = x_B))$,
        \begin{equation*}
            \PP(X_A \given \Do(X_B = x_B), X_C) = \PP(X_A \given X_C) \quad \text{$\PP(X_C)$-a.s.}
        \end{equation*}
    \end{enumerate}
\end{restatable}
In particular, any system of causal SDEs satisfying Assumption~\ref{ass:uniquely_solvable} is essentially uniquely solvable w.r.t.\ every $O\subseteq V$ (Theorem~\ref{thm:ess_solvable_under_ass}), and hence the $\sigma$-separation do-calculus applies in the causal graph $G(\Dcal)$.

A $d$-separation analogue of Theorem~\ref{thm:do_calculus} does not follow automatically from the $d$-separation Markov property for systems of causal SDEs that satisfy Assumption~\ref{ass:additive_noise}: the proof of Theorem~\ref{thm:do_calculus} for Rules 2 and 3 uses an auxiliary system $\Dcal_{\Do(R_B \sim \nu)}$ (cf.\ Definition \ref{def:intervention_variables}) for which the required $d$-separation Markov property does not automatically follow. We conjecture this $d$-separation do-calculus to hold. For acyclic graphs, $\sigma$-separation and $d$-separation coincide, so on acyclic systems the $d$-separation do-calculus follows trivially from Theorem~\ref{thm:do_calculus}.

We obtain a fundamental consequence of the do-calculus for inferring non-causation in systems of causal SDEs, and thereby establish a clear causal interpretation of the causal graph.
\begin{restatable}[]{theorem}{causationimpliespath}\label{thm:causation_implies_path}
    Let $\Dcal$ be a system of causal SDEs that is essentially uniquely solvable w.r.t.\ every $O\subseteq V$ and let $u, v \in V$. If there is no directed path from $u$ to $v$ in $G(\Dcal)$, then $\PP(X_v \given \Do(X_u)) = \PP(X_v)$.
\end{restatable}

Another consequence of Theorem~\ref{thm:do_calculus} is that (generalised) adjustment formulae like backdoor adjustment, and the ID algorithm -- which are all derived from the do-calculus -- are also valid, provided the appropriate absolute continuity conditions hold \citep{pearl2009causality,forre2020causal,forre2025mathematical}.

Similar results have been derived in other frameworks. \cite{roysland2024graphical} provide identifiability criteria for causal effects based on local independence graphs. \cite{gill2001causal} provide a g-formula for continuous-time processes in the potential outcomes framework, that \cite{ryalen2026causal} partially extend to a g-formula for time-varying treatment regimes for marked point processes.
The do-calculus and subsequent adjustment formulae consider the identification of certain estimands in terms of observational distributions. We note that the construction of \emph{estimators} for such expressions can be highly non-trivial when the variables take values in function spaces, see e.g.\ \cite{gill2001causal}, \cite{lok2008statistical}, \cite{rytgaard2022continuoustime}, and \cite{schwank2026nonparametric}.

\section{Time-splitting and subsampling}\label{sec:time-evaluations}
When reasoning about causal dynamical systems, it can be useful to evaluate subprocesses on distinct time intervals, for example, when analysing how local modifications propagate through the system. Formally, this requires extending the framework of causal SDEs to time-split systems, in which the global time interval is partitioned into subintervals, each equipped with its own structural dynamics while remaining consistent with the overall system.

\begin{example}[Time-split repressilator]\label{ex:repressilator_ts}
    Suppose that for some $t\in (0,T)$ we want to consider the repressilator on the time points $0, t, T$, and the intermediate intervals $(0,t)$ and $(t, T)$.
    This can straightforwardly be modelled by considering the variables $X^0, X^{(0,t)}, X^t, X^{(t,T)}, X^T$ with dynamics described by the SDEs
    \begin{align*}
        X^{(0,t)}(s) & = X^0  \hspace{40pt}            + \int_0^s\mu(u, X^0, X^{(0,t)}) du + \sigma \tilde{Z}^{(0,t)}(s) & s\in (0,t) \\
        X^t          & = X^{(0, t)}(t-)                                                                                               \\
        X^{(t,T)}(s) & = X^t  \hspace{40pt}          + \int_{t}^s\mu(u, X^t, X^{(t,T)}) du + \sigma \tilde{Z}^{(t,T)}(s) & s\in (t,T) \\
        X^T          & = X^{(t,T)}(T-)
    \end{align*}
    where $\tilde{Z}^{(0,t)}(s) := Z(s)-Z(0)$ for $s\in(0,t)$ and $\tilde{Z}^{(t,T)}(s) := Z(s)-Z(t)$ for $s\in(t,T)$ are independent exogenous variables. The graph of the time-split system (marginalised onto $P_{lacI}, P_{tetR}$ and $P_{GFP}$ for visual clarity, denoted by $\bar\Dcal$) is shown in Figure \ref{fig:timesplit_subsampled}.

    To model the inhibition of $P_{lacI}$ during the first half of the process, we can consider the time-split SDE with $t=T/2$ and doing the perfect intervention $\Do(M_{lacI}^0 = M_{lacI}^{(0,t)} = 0)$. After releasing the intervention the system returns to its stable behaviour, as depicted in Figure \ref{fig:repressilator_knockout_half}.
    \begin{figure*}[htb]
        \centering
        \begin{tikzpicture}[scale=1,xscale=1.4,yscale=1.3]
            \node[] (X1) at (0, 0) {$P_{lacI}^{[0,T]}$};
            \node[] (X2) at (0, -1) {$P_{tetR}^{[0,T]}$};
            \node[] (X3) at (0, -2) {$P_{GFP}^{[0,T]}$};
            \draw[arr,bend right=20] (X1) to (X2);
            \draw[arr,bend right=20] (X2) to (X1);
            \draw[arr] (X2) to (X3);
            \draw[arr] (X1) to [loop arc={90}{30}{0.4cm}] (X1);
            \draw[arr] (X2) to [loop arc rev={180}{30}{0.4cm}] (X2);
            \draw[arr] (X3) to [loop arc rev={270}{30}{0.4cm}] (X3);
        \end{tikzpicture}
        \begin{tikzpicture}[scale=1,xscale=1.4,yscale=1.3]
            \node[] () at (-1.8, -.9) {$\overset{\textrm{\tiny Def.\ \ref{def:sdes_timesplit}}}{\implies}$};
            \node[] (X1[0]) at (-1, 0) {$P_{lacI}^0$};
            \node[] (X1[0t]) at (0, 0) {$P_{lacI}^{(0,t)}$};
            \node[] (X1[t]) at (1, 0) {$P_{lacI}^{t}$};
            \node[] (X1[tT]) at (2, 0) {$P_{lacI}^{(t,T)}$};
            \node[] (X1[T]) at (3, 0) {$P_{lacI}^{T}$};
            \node[] (X2[0]) at (-1, -1) {$P_{tetR}^0$};
            \node[] (X2[0t]) at (0, -1) {$P_{tetR}^{(0,t)}$};
            \node[] (X2[t]) at (1, -1) {$P_{tetR}^{t}$};
            \node[] (X2[tT]) at (2, -1) {$P_{tetR}^{(t,T)}$};
            \node[] (X2[T]) at (3, -1) {$P_{tetR}^{T}$};
            \node[] (X3[0]) at (-1, -2) {$P_{GFP}^{0}$};
            \node[] (X3[0t]) at (0, -2) {$P_{GFP}^{(0,t)}$};
            \node[] (X3[t]) at (1, -2) {$P_{GFP}^{t}$};
            \node[] (X3[tT]) at (2, -2) {$P_{GFP}^{(t,T)}$};
            \node[] (X3[T]) at (3, -2) {$P_{GFP}^{T}$};
            \draw[arr] (X1[0]) to (X1[0t]);
            \draw[arr] (X2[0]) to (X2[0t]);
            \draw[arr] (X1[0t]) to (X1[t]);
            \draw[arr] (X1[t]) to (X1[tT]);
            \draw[arr] (X1[tT]) to (X1[T]);
            \draw[arr] (X2[0t]) to (X2[t]);
            \draw[arr] (X2[t]) to (X2[tT]);
            \draw[arr] (X2[tT]) to (X2[T]);
            \draw[arr] (X1[0]) to (X2[0t]);
            \draw[arr] (X2[0]) to (X1[0t]);
            \draw[arr,bend right=20] (X1[0t]) to (X2[0t]);
            \draw[arr,bend right=20] (X2[0t]) to (X1[0t]);
            \draw[arr] (X1[0t]) to (X2[t]);
            \draw[arr] (X2[0t]) to (X1[t]);
            \draw[arr] (X1[t]) to (X2[tT]);
            \draw[arr] (X2[t]) to (X1[tT]);
            \draw[arr,bend right=20] (X1[tT]) to (X2[tT]);
            \draw[arr,bend right=20] (X2[tT]) to (X1[tT]);
            \draw[arr] (X1[t]) to (X2[tT]);
            \draw[arr] (X2[t]) to (X1[tT]);
            \draw[arr] (X1[tT]) to (X2[T]);
            \draw[arr] (X2[tT]) to (X1[T]);
            \draw[arr,bend right=20] (X1[tT]) to (X2[tT]);
            \draw[arr,bend right=20] (X2[tT]) to (X1[tT]);
            \draw[arr] (X3[0]) to (X3[0t]);
            \draw[arr] (X3[0t]) to (X3[t]);
            \draw[arr] (X3[t]) to (X3[tT]);
            \draw[arr] (X3[tT]) to (X3[T]);
            \draw[arr] (X2[0]) to (X3[0t]);
            \draw[arr] (X2[0t]) to (X3[0t]);
            \draw[arr] (X2[0t]) to (X3[t]);
            \draw[arr] (X2[t]) to (X3[tT]);
            \draw[arr] (X2[tT]) to (X3[tT]);
            \draw[arr] (X2[tT]) to (X3[T]);
            \draw[arr] (X1[0t]) to [loop arc={90}{30}{0.4cm}] (X1[0t]);
            \draw[arr] (X2[0t]) to [loop arc rev={230}{20}{0.3cm}] (X2[0t]);
            \draw[arr] (X3[0t]) to [loop arc rev={270}{30}{0.4cm}] (X3[0t]);
            \draw[arr] (X1[tT]) to [loop arc={90}{30}{0.4cm}] (X1[tT]);
            \draw[arr] (X2[tT]) to [loop arc rev={230}{20}{0.3cm}] (X2[tT]);
            \draw[arr] (X3[tT]) to [loop arc rev={270}{30}{0.4cm}] (X3[tT]);
        \end{tikzpicture}
        \begin{tikzpicture}[scale=1,xscale=1.4,yscale=1.3]
            \node[] () at (-1.8, -.9) {$\overset{\textrm{\tiny Def.\ \ref{def:subsampled}}}{\implies}$};
            \node[] () at (0, -2.61) {};
            \node[] (X1[0]) at (-1, 0) {$P_{lacI}^0$};
            \node[] (X1[t]) at (0, 0) {$P_{lacI}^{t}$};
            \node[] (X1[T]) at (1, 0) {$P_{lacI}^{T}$};
            \node[] (X2[0]) at (-1, -1) {$P_{tetR}^0$};
            \node[] (X2[t]) at (0, -1) {$P_{tetR}^{t}$};
            \node[] (X2[T]) at (1, -1) {$P_{tetR}^{T}$};
            \node[] (X3[0]) at (-1, -2) {$P_{GFP}^{0}$};
            \node[] (X3[t]) at (0, -2) {$P_{GFP}^{t}$};
            \node[] (X3[T]) at (1, -2) {$P_{GFP}^{T}$};
            \draw[arr,color=gray] (X1[0]) to (X3[t]);
            \draw[arr,color=gray] (X1[t]) to (X3[T]);
            \draw[arr] (X1[0]) to (X1[t]);
            \draw[arr] (X1[0]) to (X2[t]);
            \draw[arr] (X1[t]) to (X1[T]);
            \draw[arr] (X1[t]) to (X2[T]);
            \draw[arr] (X2[0]) to (X1[t]);
            \draw[arr] (X2[0]) to (X2[t]);
            \draw[arr] (X2[0]) to (X3[t]);
            \draw[arr] (X2[t]) to (X1[T]);
            \draw[arr] (X2[t]) to (X2[T]);
            \draw[arr] (X2[t]) to (X3[T]);
            \draw[arr] (X3[0]) to (X3[t]);
            \draw[arr] (X3[t]) to (X3[T]);
        \end{tikzpicture}
        \caption{The causal graph $G(\bar\Dcal)$, the time-split causal graph $G(\bar\Dcal^\Ecal)$ with $\Ecal = \{\{0\}, (0,t), \{t\}, (t, T), \{T\}\}$, and the subsampled causal graph $G(\bar\Dcal^{\{0, t, T\}})$ (see Definition~\ref{def:subsampled} below).}
        \label{fig:timesplit_subsampled}
    \end{figure*}
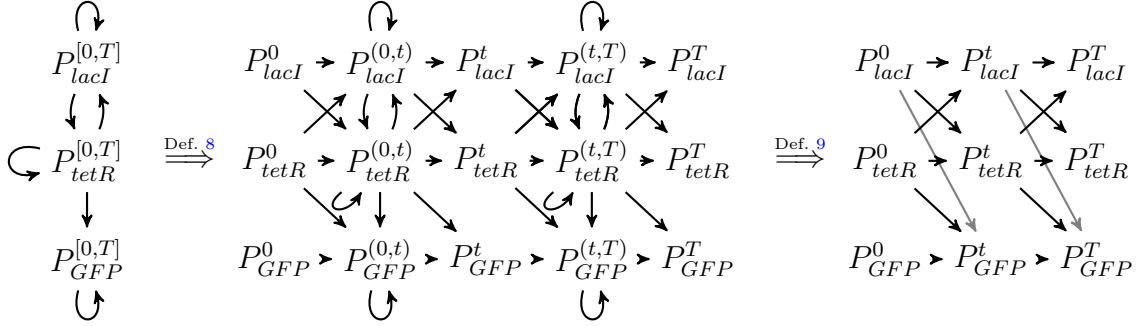
    \begin{figure}[!htb]
        \centering
        \includegraphics[width=0.7\linewidth]{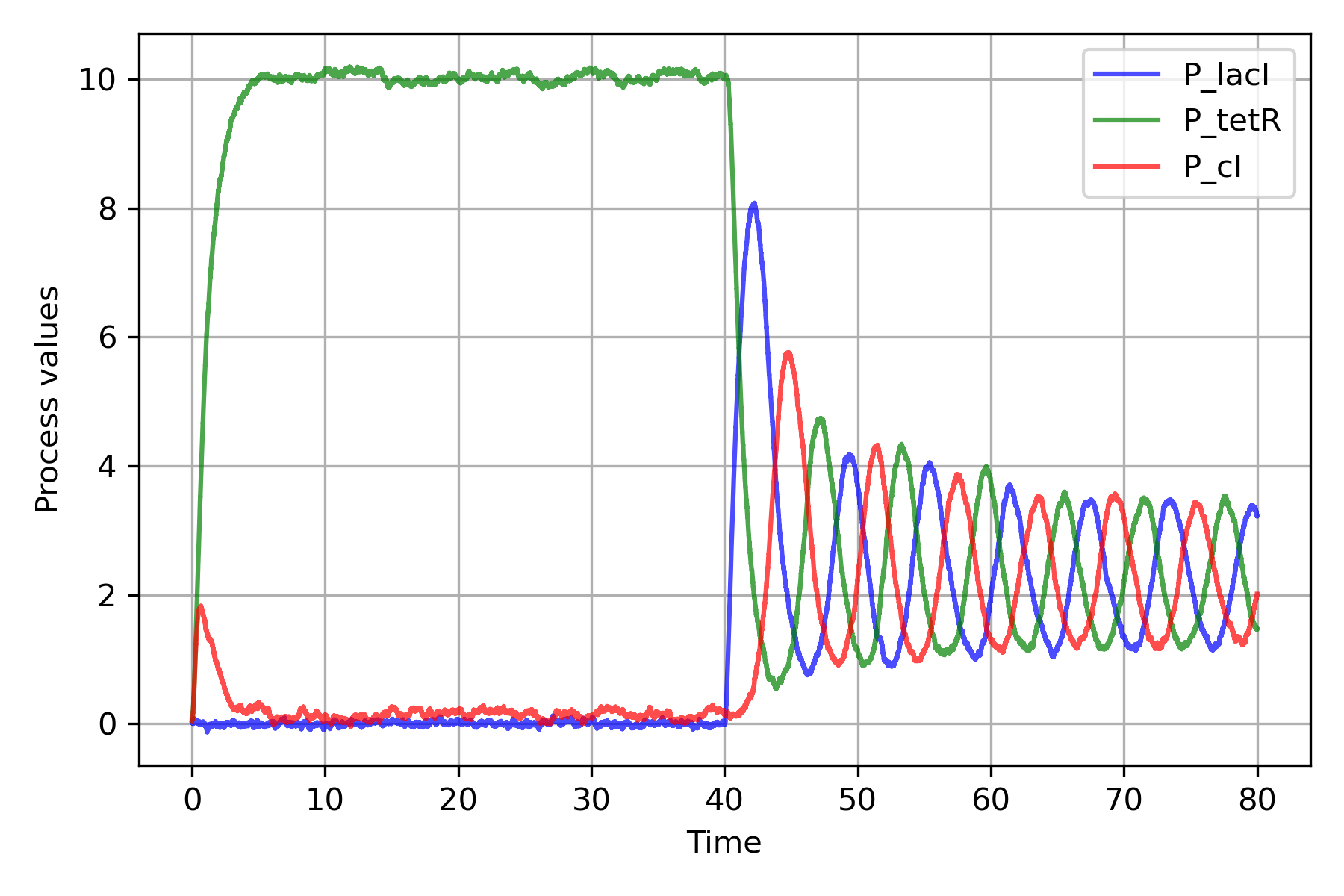}
        \caption{Repressilator with gene `lacI' knocked out in the first half of the simulation.}
        \label{fig:repressilator_knockout_half}
    \end{figure}
\end{example}

We now formalise this construction as follows. Let $\Ecal = \{\Ical_1, ..., \Ical_m\}$ be a finite partition of $[0,T]$ into intervals with \emph{temporal ordering}, i.e., for $i < j$ we have $s < t$ for all $s\in \Ical_i$ and $t\in \Ical_j$. For càdlàg path $x \in D([0,T], \RR^n)$, if we let $x^\Ical := x|_\Ical \in D(\Ical, \RR^n)$ be the restriction of $x$ to $\Ical$ and write $x^\Ecal = (x^{\Ical_1}, ..., x^{\Ical_m})$, then let $\mathrm{cat}(x^\Ecal) \in D([0,T], \RR^n)$ be the concatenation $\mathrm{cat}(x^\Ecal)(t) := \sum_{\Ical_k}\I_{\Ical_k}(t) x^{\Ical_k}(t)$ that retrieves $x$. For consecutive intervals $\Ical_k, \Ical_{k+1}$ in $\Ecal$, let
\begin{equation*}
    b_{k, k+1} :=
    \begin{cases}
        (\sup\Ical_{k})- & \text{if } \Ical_{k} \text{ is right-open},  \\
        \sup\Ical_{k}    & \text{if } \Ical_{k} \text{ is right-closed}
    \end{cases}
\end{equation*}
and $b_{0,1}:=0$ and $b_{m, m+1}:=T$.
If for the càdlàg process $X(t) = \int_0^tg\diff h$ and partition $(\Ical_1, \Ical_2)$ of $[0,T]$ we set $X^{\Ical_1}(t) = \int_0^{t}g\diff h$ for all $t\in \Ical_1$ and $X^{\Ical_2}(t) = X^{\Ical_1}(b_{1,2}) + \int_{b_{1,2}}^t g\diff h$ for all $t\in\Ical_2$, then if $X$ has a jump at the boundary point $\sup\Ical_1$, it will be contained in $X^{\Ical_1}(b_{1,2})$ when $\Ical_1$ is right-closed and in $\int_{b_{1,2}}^t g\diff h$ when $\Ical_1$ is right-open. If for $y\in D([0,T], \RR^n)$ we let $y^{\Ical_1}(t) := y(t) \in D(\Ical_1, \RR^n)$ and $y^{\Ical_k}(t) := y(t) - y(b_{k-1,k}) \in D(\Ical_k, \RR^n)$ for $k\geq 2$ be the increment path (not counting jumps at $\inf \Ical_k$ if $\Ical_k$ is left-open), then let $\mathrm{rec}(y^\Ecal)(t) := \sum_{k}\I_{\Ical_k}(t)\cdot (\sum_{j<k} y^{\Ical_j}(b_{j, j+1}) + y^{\Ical_k}(t))$ be the reconstruction that retrieves $y$.

\begin{definition}[Time-split system of SDEs]\label{def:sdes_timesplit}
    Let $\Dcal = \angs{V, W, X_W, f, g, h}$ be a system of causal SDEs on $[0,T]$ and let $\Ecal = \{\Ical_1, ..., \Ical_m\}$ be a finite partition of $[0,T]$ with temporal ordering.
    The \emph{time-split system of SDEs} $\Dcal^\Ecal := \angs{V\times\Ecal, W\times\Ecal, X_{W\times\Ecal}, f_V^\Ecal, g_V^\Ecal, h_V^\Ecal}$ is defined with for each $(w,\Ical_k)\in W\times \Ecal$ the \emph{exogenous increment process} $X_w^{\Ical_k} \in D(\Ical_k, \RR)$ defined as $X_w^{\Ical_1}(t):=X_w(t)$ for $t\in\Ical_1$ and $X_w^{\Ical_k}(t) := X_w(t) - X_w(b_{k-1, k})$ for $t\in\Ical_k$ for all $k\geq 2$, and for each $(v, \Ical_k) \in V\times \Ecal$ the causal SDE\footnote{The pathwise integration map $\Psi$ is stated for stochastic integrals on $[0,T]$, but can be extended to any interval $\Ical_k$.}
    \begin{align*}
        X_v^{\Ical_k}(s) = X_v^{\Ical_{k-1}}(b_{k-1,k}) & + f_v^{\Ical_k}\bigl(s, X_{\alpha(v)}^\Ecal\bigr) + \int_{b_{k-1,k}}^s g_v^{\Ical_k}\bigl(u-,  X_{\{v\}\cup\beta(v)}^\Ecal\bigr) \diff h_v^{\Ical_k}\bigl(u, X_{\gamma(v)}^\Ecal\bigr)
    \end{align*}
    for $s\in\Ical_k$, with
    \begin{align*}
        f_v^{\Ical_k}(s, x_{\alpha(v)}^\Ecal)          & := f_v\bigl(s, \mathrm{cat}(x_{\alpha(v)\cap V}^\Ecal), \mathrm{rec}(x_{\alpha(v)\cap W}^\Ecal)\bigr) - f_v\bigl(b_{k-1,k}, \mathrm{cat}(x_{\alpha(v)\cap V}^\Ecal), \mathrm{rec}(x_{\alpha(v)\cap W}^\Ecal)\bigr), \\
        g_v^{\Ical_k}(u-, x_{\{v\}\cup\beta(v)}^\Ecal) & := g_v\bigl(u-, \mathrm{cat}(x_{\{v\}\cup(\beta(v)\cap V)}^\Ecal), \mathrm{rec}(x_{\beta(v)\cap W}^\Ecal)\bigr),                                                                                                    \\
        h_v^{\Ical_k}(u, x_{\gamma(v)}^\Ecal)          & := h_v\bigl(u, \mathrm{rec}(x_{\gamma(v)}^\Ecal)\bigr),
    \end{align*}
    and where we set $X_v^{\Ical_0}(b_{0,1}) := f_v(0, X_{\alpha(v)}^{\Ical_1})$.
\end{definition}
The time-split system is observationally equivalent to the original: concatenating any solution of $\Dcal^\Ecal$ yields a solution of $\Dcal$, and conversely the restrictions of any solution of $\Dcal$ form a solution of $\Dcal^\Ecal$; see Lemma~\ref{lem:timesplit_equivalent}. Note however that after intervening on a variable $X_v^\Ical$ with $\Ical$ left-open, the resulting path need not be càdlàg. This imposes an additional condition on the notion of essential unique solvability w.r.t.\ $O\subseteq V\times \Ecal$, which is stated in terms of input processes $\varphi_{S}:\Xcal_W \to \Xcal_S$ where $S = V\times \Ecal \setminus O$. Only $\varphi_S$ should be considered such that for $(v, \Ical)\in S$ the path $\varphi_{(v,\Ical)}$ has a left limit at $\sup \Ical$ (which is not automatic when $\Ical$ is right-open), and if $(v, \Ical_k), (v, \Ical_{k+1})\in S$, the concatenation of $\varphi_{(v, \Ical_k)}$ and $\varphi_{(v, \Ical_{k+1})}$ must be right-continuous at $\sup\Ical_k$. We refer to such $\varphi_S$ as \emph{admissible}.
The causal graph $G(\Dcal^\Ecal)$ has vertex set $V\times\Ecal$ and edges $(u, \Ical') \to (v, \Ical)$ derived exactly as in Definition \ref{def:sde_graph}. We note three sources of graphical sparsity:
\begin{enumerate}[label=(\roman*)]
    \item \emph{Adaptedness}: by adaptedness of $f_v, g_v, h_v$, the values of $f_v^{\Ical_k}, g_v^{\Ical_k}, h_v^{\Ical_k}$ only depend on $X_{V\cup W}^{\Ical_j}$ with $j \leq k$, so there are no edges into $(v, \Ical_k)$ from later intervals.
    \item \emph{Markovianity}: If in the original SDE $g_v$ is Markov and $f_v$ either models the initial condition or is Markov then
    there are no edges from any $X_u^{\Ical_j}$ with $u\in V$ and $j<k-1$ to $X_v^{\Ical_k}$.
    \item \emph{Independent increments}: if an exogenous process $X_w$ has independent increments, then any two increment processes $X_w^{\Ical_j}$ and $X_w^{\Ical_k}$ with $j\neq k$ are independent. Consequently, if $X_v$ depends on an exogenous process $X_w$, then in the graph $G(\Dcal^\Ecal)$ there is no bidirected edge between $X_v^{\Ical_{k-1}}$ and $X_v^{\Ical_k}$ because of the common cause $X_w$.
\end{enumerate}

If $\Dcal$ satisfies Assumption~\ref{ass:uniquely_solvable} and its exogenous processes have independent increments, then by Lemma~\ref{lem:timesplit_solvable} in the Appendix, $\Dcal^\Ecal$ is essentially uniquely solvable w.r.t.\ every $O \subseteq V \times \Ecal$. By Lemma~\ref{lem:timesplit_embedding} in the Appendix, every time-split system of causal SDEs $\Dcal^\Ecal$ whose exogenous processes have independent increments can be embedded as a system of causal SDEs on $[0,T]$, so the Markov properties (Theorems \ref{thm:sigma_sep_mp} and \ref{thm:dsep_mp}) and the do-calculus (Theorem~\ref{thm:do_calculus}) apply directly to $\Dcal^\Ecal$, with $\sigma$- and $d$-separation read off from the time-split causal graph $G(\Dcal^\Ecal)$.

\begin{restatable}[Markov properties for time-split systems of causal SDEs]{theorem}{markovtimesplit}\label{thm:markov_timesplit}
    Let $\Ecal$ be a finite partition of $[0,T]$ into intervals with temporal ordering, and let $\Dcal$ be a system of causal SDEs whose exogenous processes $X_w$ ($w \in W$) have independent increments.
    \begin{enumerate}[label=(\roman*)]
        \item If $\Dcal$ satisfies Assumption~\ref{ass:uniquely_solvable}, then $\Dcal^\Ecal$ satisfies the $\sigma$-separation Markov property: for all $A, B, C_1, ..., C_r\subseteq V$ and $\Ical_A, \Ical_B, \Ical_{C_1}, ..., \Ical_{C_r} \subseteq \Ecal$,
        \begin{equation*}
            X_A^{\Ical_A}\Perp^\sigma_{G(\Dcal^\Ecal)} X_B^{\Ical_B} \given X_{C_1}^{\Ical_{C_1}}, ..., X_{C_r}^{\Ical_{C_r}} \implies X_A^{\Ical_A} \Indep X_B^{\Ical_B} \given X_{C_1}^{\Ical_{C_1}}, ..., X_{C_r}^{\Ical_{C_r}}.
        \end{equation*}
        \item If $\Dcal$ satisfies Assumption~\ref{ass:additive_noise}, then $\Dcal^\Ecal$ satisfies the $d$-separation Markov property: for all $A, B, C_1, ..., C_r\subseteq V$ and $\Ical_A, \Ical_B, \Ical_{C_1}, ..., \Ical_{C_r} \subseteq \Ecal$,
        \begin{equation*}
            X_A^{\Ical_A}\Perp^d_{G(\Dcal^\Ecal)} X_B^{\Ical_B} \given X_{C_1}^{\Ical_{C_1}}, ..., X_{C_r}^{\Ical_{C_r}} \implies X_A^{\Ical_A} \Indep X_B^{\Ical_B} \given X_{C_1}^{\Ical_{C_1}}, ..., X_{C_r}^{\Ical_{C_r}}.
        \end{equation*}
    \end{enumerate}
\end{restatable}

\begin{restatable}[Do-calculus for time-split systems of causal SDEs]{theorem}{docalctimesplit}\label{thm:do_calculus_timesplit}
    Let $\Dcal$ be a system of causal SDEs satisfying Assumption~\ref{ass:uniquely_solvable} whose exogenous processes $X_w$ ($w \in W$) have independent increments, and let $\Ecal$ be a finite partition of $[0,T]$ into intervals with temporal ordering. For $A, B, C_1, \ldots, C_r \subseteq V$ and $\Ical_A, \Ical_B, \Ical_{C_1}, \ldots, \Ical_{C_r} \subseteq \Ecal$ such that the blocks $A^{\Ical_A}, B^{\Ical_B}, C_1^{\Ical_{C_1}}, \ldots, C_r^{\Ical_{C_r}}$ are pairwise disjoint in $V\times\Ecal$, write $X_{C}^{\Ical_C} := (X_{C_1}^{\Ical_{C_1}}, \ldots, X_{C_r}^{\Ical_{C_r}})$. Then:
    \begin{enumerate}
        \item \emph{Insertion/deletion of observation.} If $X_A^{\Ical_A} \Perp^\sigma_{G(\Dcal^\Ecal)} X_B^{\Ical_B} \given X_{C}^{\Ical_C}$, then
        \begin{equation*}
            \PP(X_A^{\Ical_A} \given X_B^{\Ical_B}, X_{C}^{\Ical_C}) = \PP(X_A^{\Ical_A} \given X_{C}^{\Ical_C}) \quad \PP(X_B^{\Ical_B}, X_{C}^{\Ical_C})\text{-a.s.}
        \end{equation*}
        \item \emph{Action/observation exchange.} If $X_A^{\Ical_A} \Perp^\sigma_{G(\Dcal^\Ecal)_{\Do(R_B^{\Ical_B})}} R_B^{\Ical_B} \given X_B^{\Ical_B}, X_{C}^{\Ical_C}$ and $\PP(X_{C}^{\Ical_C} \given X_B^{\Ical_B} = x_B) \ll \PP(X_{C}^{\Ical_C} \given \Do(X_B^{\Ical_B} = x_B))$ for $\PP(X_B^{\Ical_B})$-a.a.\ $x_B$, then
        \begin{equation*}
            \PP(X_A^{\Ical_A} \given \Do(X_B^{\Ical_B}), X_{C}^{\Ical_C}) = \PP(X_A^{\Ical_A} \given X_B^{\Ical_B}, X_{C}^{\Ical_C}) \quad \PP(X_B^{\Ical_B}, X_{C}^{\Ical_C})\text{-a.s.}
        \end{equation*}
        \item \emph{Insertion/deletion of action.} If $X_A^{\Ical_A} \Perp^\sigma_{G(\Dcal^\Ecal)_{\Do(R_B^{\Ical_B})}} R_B^{\Ical_B} \given X_{C}^{\Ical_C}$, then for every $x_B \in \Xcal_{B^{\Ical_B}}$ such that $\PP(X_{C}^{\Ical_C}) \ll \PP(X_{C}^{\Ical_C} \given \Do(X_B^{\Ical_B} = x_B))$,
        \begin{equation*}
            \PP(X_A^{\Ical_A} \given \Do(X_B^{\Ical_B} = x_B), X_{C}^{\Ical_C}) = \PP(X_A^{\Ical_A} \given X_{C}^{\Ical_C}) \quad \PP(X_{C}^{\Ical_C})\text{-a.s.}
        \end{equation*}
    \end{enumerate}
\end{restatable}

\begin{example}
    The summary, time-split, and subsampled causal graphs of the repressilator are depicted in Figure \ref{fig:timesplit_subsampled}.
    In the time-split graph, the parents of $P_{GFP}^{(0,t)}$ are $\{P_{GFP}^0, P_{tetR}^0, P_{tetR}^{(0,t)}\}$. Let $Z := (P_{tetR}^0, P_{tetR}^{(0,t)})$. Since $P_{GFP}^0$ is a source with no other outgoing edges, every backdoor path from $P_{GFP}^{(0,t)}$ to $P_{GFP}^T$ must enter $P_{GFP}^{(0,t)}$ via $P_{tetR}^0$ or $P_{tetR}^{(0,t)}$, so conditioning on $Z$ blocks all such paths. Moreover, $Z$ contains no descendants of $P_{GFP}^{(0,t)}$. Assuming the absolute-continuity conditions of Theorem~\ref{thm:do_calculus_timesplit} hold, Rule 3 gives $\PP(Z \given \Do(P_{GFP}^{(0,t)})) = \PP(Z)$, and Rule 2 gives $\PP(P_{GFP}^T \given \Do(P_{GFP}^{(0,t)}), Z) = \PP(P_{GFP}^T \given P_{GFP}^{(0,t)}, Z)$, which combine to yield the backdoor adjustment formula
    \begin{equation}\label{eqn:adjustment}
        \PP\left(P_{GFP}^T \given \Do(P_{GFP}^{(0,t)})\right) = \int \PP\left(P_{GFP}^T \given P_{GFP}^{(0,t)}, P_{tetR}^0, P_{tetR}^{(0,t)}\right) \diff\PP(P_{tetR}^0, P_{tetR}^{(0,t)}).
    \end{equation}
\end{example}

Finally, we consider time-split systems where the endogenous variables only consist of time-points.

\begin{definition}[Subsampled system of causal SDEs]\label{def:subsampled}
    Let $\Dcal$ be a system of causal SDEs satisfying Assumption~\ref{ass:uniquely_solvable} whose exogenous processes have independent increments, so that the time-split system $\Dcal^\Ecal$ is essentially uniquely solvable w.r.t.\ every $O \subseteq V \times \Ecal$ by Lemma~\ref{lem:timesplit_solvable}. Given a set of time-points $\Scal \subseteq [0,T]$, let $\Ecal'$ be a set of intervals such that $\Ecal := \Scal \cup \Ecal'$ partitions $[0,T]$. The \emph{subsampled system of causal SDEs} is defined as
    \begin{equation*}
        \Dcal^\Scal := \bigl(\Dcal^\Ecal\bigr)_{\marg(V \times \Ecal')},
    \end{equation*}
    the marginalisation of the time-split system $\Dcal^\Ecal$ onto the time-point components $V \times \Scal$. Since $\Dcal^\Ecal$ is essentially uniquely solvable w.r.t.\ every $O \subseteq V \times \Ecal$, Theorem~\ref{thm:marg_closure_simple} gives that $\Dcal^\Scal$ is again essentially uniquely solvable w.r.t.\ every $O' \subseteq V \times \Scal$.
\end{definition}

\begin{example}\label{ex:time-split}
    From the Markov property in the summary graph (Figure \ref{fig:timesplit_subsampled}) we infer that $P_{lacI}^{[0,T]} \Indep P_{GFP}^{[0,T]} \given P_{tetR}^{[0,T]}$. If we consider the subsampled time-series, however, we might have $P_{lacI}^{\{0,t,T\}} \nIndep P_{GFP}^{\{0,t,T\}} \given P_{tetR}^{\{0,t,T\}}$.
\end{example}

Analysis of the information loss of the subsampling operation could be particularly interesting in the light of (im)possibility results for inferring features of $\Dcal$ from the subsampled system. In particular, inferences from subsampled time series need not represent features of the underlying continuous-time system. For example, conditional independence relations which hold in $\Dcal$ might not hold in $\Dcal^\Scal$, as portrayed in the example above; see also \cite{aalen2016can}, who show a similar phenomenon for local independence graphs. Another example is that of perfect adaptation: if $A \to B$ but $B$ perfectly adapts between two sampling points, no edge from $A$ to $B$ is present in the graph of the subsampled system \citep{blom2023causality,weinberger2026homeostasis}.

\subsection{Granger causality in time-split systems}\label{sec:granger_causality}
Having a time-splitting operation at our disposal, we can consider predicting future values of a process from past values of the same and other processes. For discrete-time stochastic processes, \cite{granger1969investigating,granger1980testing} called $A$ a cause of $B$ if, given the history of all other variables, the history of $A$ contains useful information for predicting $B$. \cite{florens1996noncausality} extended this notion to continuous time. We adopt the following slightly more flexible variant, allowing arbitrary conditioning sets:
\begin{definition}[Global Granger non-causation]\label{def:global_granger}
    Let $A, B, C\subseteq V$. We say that $X_A$ does not \emph{globally Granger cause $X_B$ given $X_C$} if
    \begin{equation*}
        X_A^{[0,s]} \Indep X_B^{(s,t]} \given X_C^{[0,s]} \quad \text{for all } 0 \leq s < t \leq T,
    \end{equation*}
    equivalently, $X_A^{[0,s]} \Indep X_B^{(s, T]} \given X_C^{[0,s]}$ for all $0 \leq s < T$.
\end{definition}
The classical notion of \emph{$X_A$ does not (globally) Granger-cause $X_B$} is recovered as the special case $C = V\setminus A$. If $\Dcal$ satisfies Assumption~\ref{ass:uniquely_solvable}, global Granger non-causation can be read off from the time-split causal graph via the $\sigma$-separation Markov property:
\begin{equation*}
    X_A^{[0,s]}\Perp^\sigma_{G(\Dcal^{\{[0,s],(s,T]\}})} X_B^{(s,t]} \given X_C^{[0,s]} \implies X_A^{[0,s]} \Indep X_B^{(s,t]} \given X_C^{[0,s]}.
\end{equation*}
When $\Dcal$ additionally satisfies Assumption~\ref{ass:additive_noise}, the same conclusion follows from the stronger $d$-separation Markov property in the time-split graph (Theorem~\ref{thm:markov_timesplit}(ii)).

The following result relates Granger causation to causation in $G(\Dcal)$.
\begin{restatable}[]{theorem}{grangeriscausation}\label{thm:granger_is_causation}
    Let $\Dcal$ be a system of causal SDEs satisfying Assumption~\ref{ass:uniquely_solvable} whose exogenous processes $X_w$ ($w\in W$) have independent increments, and let $u, v\in V$ with $u\neq v$.
    \begin{enumerate}[label=(\roman*)]
        \item Let $\Dcal^{\{[0,s], (s,T]\}}$ be faithful for every $0\leq s < T$. If there is a directed path from $u$ to $v$ in $G(\Dcal)$, then $u$ is a Granger cause of $v$.
        \item Let $G(\Dcal)$ have no bidirected edges. If $u$ is a Granger cause of $v$, then there is a directed path from $u$ to $v$ in $G(\Dcal)$.
    \end{enumerate}
\end{restatable}
When combining the above assumptions of faithfulness and no bidirected edges, Granger-causation is equivalent to the existence of a directed path.
For discrete-time models, a similar link between Granger causation and causation in SCMs has also been remarked by \cite{peters2013causal} (Theorem 10.3), \cite{white2010granger} and \cite{eichler2012graphical}.

\subsubsection{Local independence}\label{sec:local_independence}
Besides the `global' Granger non-causality, much interest has been shown in a local notion called \emph{local independence} \citep{schweder1970composable,aalen1978nonparametric,didelez2008graphical,mogensen2020markov}. It has been conjectured that it is equivalent to the conditional independence $X_A^{[0,t)}\Indep X_B^{t} \given X_C^{[0,t)}$ \citep{didelez2008graphical}, but if $B\subseteq C$ and $X_B$ is continuous, then $X_B(t-) = X_B(t)$ a.s.\ so this conditional independence holds trivially even if there is a local dependence \citep{christgau2023nonparametric}. In this section we investigate the relation between Granger non-causality and local independence, and we derive a local independence Markov property based on $\sigma$-separation (and $d$-separation when appropriate) in the time-split graph.

Defining local independence requires some technical background. A \emph{special semimartingale} $X(t)$ adapted to a filtration $\Fcal$ on compact interval $[0,T]$ is a stochastic process that has a Doob-Meyer decomposition $X(t)=\Lambda(t) + M(t)$, where $\Lambda$ is a $\Fcal$-\emph{predictable} finite variation process with $\Lambda(0) = 0$, and $M$ is a local martingale with respect to $\Fcal$; this decomposition is unique up to indistinguishability. Throughout this section, let $\Dcal$ be a system of causal SDEs satisfying Assumption~\ref{ass:uniquely_solvable} such that for every $B, C \subseteq V$ the \emph{optional projection} (the unique càdlàg version of) $\EE[X_B(t)\given \Fcal^C_t]$ exists and is a special semimartingale, where $\Fcal^C$ denotes the filtration generated by $X_C$. A convenient sufficient condition is that $X_B$ is a \emph{quasimartingale} with respect to $\Fcal^V$ (see Definition~\ref{def:quasimartingale} and Lemma~\ref{thm:bounded_implies_special_projection} in the Appendix).

\begin{definition}[Local independence]\label{def:local_independence}
    Let $A, B, C\subseteq V$ and consider the following optional projections and their Doob-Meyer decompositions
    \begin{align}\label{eqn:local_independence}
        \begin{split}
            \EE[X_B(t)\given \Fcal_t^C]   & = \Lambda^C(t) + M^C(t)      \\
            \EE[X_B(t)\given \Fcal_t^{A,C}] & = \Lambda^{A,C}(t) + M^{A,C}(t).
        \end{split}
    \end{align}
    We say that $X_B$ is \emph{locally independent}\footnote{\cite{florens1996noncausality} refer to this as \emph{weak instantaneous non-causality}, \cite{comte1996noncausality} call it \emph{local Granger non-causality}.} from $X_A$ given $X_C$, written $X_A \not\to X_B \given X_C$, if $\Lambda^{A,C} = \Lambda^C$ a.s.
\end{definition}

Local independence is equivalent to $\Lambda^{A,C}$ having a $\Fcal^C$-predictable version, and this version is then necessarily a.s.\ equal to $\Lambda^C$ (see \citealp{mogensen2020markov}, Appendix E). Some authors assume that the trajectory $t\mapsto \Lambda^{A,C}(t)$ is absolutely continuous, i.e.\ $\Lambda^{A,C}(t) = \int_0^t \lambda^{A,C}(s) \diff s$ for some $\Fcal^{A,C}$-adapted process $\lambda^{A,C}$ called the \emph{intensity} process, in which case local independence is equivalent to the intensity process $\lambda^{A,C}$ having a $\Fcal^C$-adapted version.

\begin{remark}
    In some cases where the causal SDE for $X_v$ directly gives the Doob-Meyer decomposition w.r.t.\ $\Fcal^V$, the local independence $X_u \not\to X_v \given X_{V\setminus\{u\}}$ with $u\neq v$ can be read off from the causal graph. Since $v \in V\setminus\{u\}$, both optional projections in (\ref{eqn:local_independence}) equal $X_v$ itself, and the local independence is the statement that the compensators $\Lambda^{V}$ and $\Lambda^{V\setminus \{u\}}$ coincide.
    Suppose that the mechanism of $v$ has the martingale-driven form
    \begin{equation}\label{eqn:sde_martingale_integrator}
        X_v(t) = X_v^0 + \int_0^t g_v^1(s{-}, X_v, X_{\beta(v)}) \diff s + \int_0^t g_v^2(s{-}, X_v, X_{\beta(v)}) \diff m_v(s, X_{\gamma(v)}),
    \end{equation}
    where $\beta(v)\subseteq V$, $f_v$ models the initial condition, $g_v^2$ is bounded and $m_v(\cdot, X_{\gamma(v)})$ is a square-integrable martingale with respect to $\Fcal^W$ (this for example holds under Assumption~\ref{ass:additive_noise}). The stochastic integral $M^V := \int_0^\cdot g_v^2(s{-}, X_v, X_{\beta(v)}) \diff m_v(s, X_{\gamma(v)})$ is then a martingale wrt $\Fcal^W$. Since $M^V = X_v - X_v^0 - \int_0^\cdot g_v^1(s{-}, X_v, X_{\beta(v)})\diff s$ it is $\Fcal^V$-adapted, and by \cite{follmer2011local}, Theorem 2.2 it is a martingale with respect to $\Fcal^V\subseteq \Fcal^W$ as well, so (\ref{eqn:sde_martingale_integrator}) is the Doob-Meyer decomposition of $X_v$ in $\Fcal^V$, with martingale part $M^V$ and compensator $\Lambda^V := \int_0^\cdot g_v^1(s{-}, X_v, X_{\beta(v)})\diff s$.

    If $(u\to v)\notin G(\Dcal)$, then $g_v$ does not depend on $X_u$, so both $\Lambda^V$ and $M^V$ are adapted to $\Fcal^{V\setminus\{u\}}$. By the same argument $M^V$ is an $\Fcal^{V\setminus\{u\}}$-martingale, so (\ref{eqn:sde_martingale_integrator}) is the Doob-Meyer decomposition of $X_v$ in $\Fcal^{V\setminus\{u\}}$ as well. The two compensators thus coincide, giving $X_u \not\to X_v\given X_{V\setminus\{u\}}$, so $G(\Dcal)$ is a \emph{local independence graph} in the sense of \cite{didelez2008graphical,mogensen2018causal} and \cite{mogensen2020markov}.
\end{remark}

Some authors (e.g.\ \citealp{aalen1978nonparametric,florens1996noncausality,comte1996noncausality,didelez2008graphical,christgau2023nonparametric}) always condition on the history of the target variable, i.e.\ they consider $B\subseteq C$, in which case we get $X_B(t) = \EE[X_B(t) \given \Fcal_t^C] = \EE[X_B(t) \given \Fcal_t^{A,C}]$, so (\ref{eqn:local_independence}) reads $X_B = M^{A,C} + \Lambda^{A,C}$ and $X_B = M^C + \Lambda^C$, in which case local independence $X_A\not\to X_B \given X_C$ is equivalent to $M^{A,C} = M^C$ a.s. If $B\subseteq C$, local independence trivially holds if $X_B$ is almost surely differentiable, since then $X_B(t) = X_B(0) + \int_0^t X_B'(s) \diff s$, so $X_B' = \lambda^C = \lambda^{A,C}$ a.s. \citep{comte1996noncausality}.
Other authors \citep{mogensen2018causal,mogensen2020markov} allow for general conditioning sets $C$; we follow this convention.

The following result, due to \cite{florens1996noncausality} for the special case where $B\subseteq C$, relates Granger non-causality to local independence.
\begin{restatable}[]{theorem}{globalgrangerimplieslocalindependence}\label{thm:global_granger_implies_local_independence}
    Let $\Dcal$ be a system of causal SDEs satisfying Assumption~\ref{ass:uniquely_solvable} whose exogenous processes $X_w$ ($w\in W$) have independent increments, such that $X_V$ is a quasimartingale with respect to $\Fcal^V$ with $\EE\bigl[\sup_{t \in [0,T]} \|X_V(t)\|\bigr] < \infty$, and let $A, B, C \subseteq V$. If for all $0< s < t\leq T$ we have $X_A^{[0,s]}\Indep X_B^{(s,t]} \given X_C^{[0,s]}$ in the time-split system $\Dcal^{\{[0,s], (s,t], (t,T]\}}$, then $X_A\not\to X_B \given X_C$ and $M^C = M^{A,C}$.
    Conversely, if $B=C$, then $X_A\not\to X_B \given X_C$ implies $X_A^{[0,s]}\Indep X_B^{(s,t]} \given X_C^{[0,s]}$ for all $0\leq s < t\leq T$.
\end{restatable}

The conditions of Theorem~\ref{thm:global_granger_implies_local_independence} are satisfied by the additive-noise class --- see Lemma~\ref{lem:additive_noise_quasimartingale} in the Appendix. \cite{florens1996noncausality} show that equivalence between local independence and Granger non-causality also holds for certain point processes and Markov processes for the case that $B$ is a strict subset of $C$.
As a corollary of Theorem~\ref{thm:global_granger_implies_local_independence}, we obtain the following Markov property for local independence in terms of the time-split graph.
\begin{corollary}[Markov property for local independence]\label{cor:local_independence_mp}
    Let $\Dcal$ be a system of causal SDEs satisfying Assumption~\ref{ass:uniquely_solvable} whose exogenous processes $X_w$ ($w\in W$) have independent increments, such that $X_V$ is a quasimartingale with respect to $\Fcal^V$ with $\EE\bigl[\sup_{t \in [0,T]} \|X_V(t)\|\bigr] < \infty$. Let $A, B, C \subseteq V$.
    \begin{enumerate}[label=(\roman*)]
        \item If $X_A^{[0,s]}\Perp^\sigma_{G(\Dcal^{\{[0,s],(s,T]\}})} X_B^{(s,T]}\given X_C^{[0,s]}$ for all $0< s < T$, then $X_A\not\to X_B \given X_C$.
        \item If $g_v$ is Markov, $f_v$ either models the initial condition or is Markov, and we have $X_A^{[0,s]}\Perp^\sigma_{G(\Dcal^{\{[0,s],(s,T]\}})} X_B^{(s,T]}\given X_C^{[0,s]}$ for some $0< s < T$, then $X_A\not\to X_B \given X_C$.
        \item If $\Dcal$ satisfies Assumption~\ref{ass:additive_noise} and $g_v$ is Markov, and $X_A^{[0,s]}\Perp^d_{G(\Dcal^{\{[0,s],(s,T]\}})} X_B^{(s,T]}\given X_C^{[0,s]}$ for some $0< s < T$, then $X_A\not\to X_B \given X_C$.
    \end{enumerate}
\end{corollary}
Note that in clauses (ii) and (iii), it suffices to verify the separation for a single $s$ since $g_v$ is Markov and $f_v$ either models the initial condition or is Markov, and hence the time-split graph $G(\Dcal^{\{[0,s],(s,T]\}})$ does not depend on $s\in(0,T)$.

\bigskip
Various other Markov properties have been derived for local independence in terms of the causal graph $G(\Dcal)$, see e.g.\ \cite{didelez2008graphical} for a Markov property for point processes in terms of $\delta$-separation (assuming $B\subseteq C$, excluding latent confounding), \cite{mogensen2018causal} for additive noise diffusion processes in terms of $\mu$-separation (not assuming $B\subseteq C$, allowing for latent confounding), and \cite{mogensen2022graphical} for Ornstein-Uhlenbeck processes (not assuming $B\subseteq C$, allowing for latent confounding).

\section{Discussion}

In this work we developed a formal framework that equips systems of stochastic differential equations (SDEs) with causal semantics and lifts the SCM toolkit --- causal graphs, graphical Markov properties and the do-calculus --- to this setting via a pathwise interpretation of the causal SDEs. We introduced the notion of \emph{essential unique solvability} of a system of causal SDEs, and provided an SCC-wise Lipschitz condition (Assumption~\ref{ass:uniquely_solvable}) and a subclass of additive noise SDE (Assumption~\ref{ass:additive_noise}) that are sufficient for this property to hold. We also introduced a \emph{marginalisation} operation that integrates out latent processes while preserving the observational and interventional distributions on the retained processes, and showed that systems of causal SDEs remain essentially uniquely solvable after marginalisation. Following the acyclification proof strategy of \cite{forre2017markov,bongers2021foundations,forre2025mathematical}, we established the $\sigma$-separation Markov property for essentially uniquely solvable systems of causal SDEs, and the stronger $d$-separation Markov property for the class of (cyclic) additive-noise SDEs via total variation convergence of continuous Euler approximations. The $\sigma$-separation Markov property implies the do-calculus and with that an explicit causal interpretation of the \emph{causal graph} of a system of causal SDEs. Combining marginalisation with time-split systems yields the \emph{subsampled} system (Definition~\ref{def:subsampled}), which models the dynamics observed at a finite set of time-points. Time-split systems give a Markov property for continuous-time Granger non-causality (equivalent to the absence of a directed path in the causal graph under faithfulness and no latent confounding) and an analogous Markov property for local independence. We also argue that time-split systems of causal SDEs provide a natural framework for constraint-based causal discovery from time-series data, as showcased in the following section.

\subsection{Constraint-based Causal Discovery}\label{sec:constraint_based_cd}
As mentioned earlier, the cornerstone of the discovery of causal relations is experimentation: if $\PP(Y|\Do(X))$ depends on $X$, then $X$ is a cause of $Y$. However, in many practical scenarios experimentation is impossible (e.g.\ expensive or unethical), so one might want to infer causal relations from observational data. Causal discovery algorithms estimate the underlying causal graph of an SCM from observational data by exploiting statistical patterns in the data. \emph{Constraint-based} causal discovery algorithms specifically rely on conditional independence patterns, the Markov property and faithfulness assumption of the SCM.

A classical example is the PC algorithm \citep{spirtes2000causation}, which assumes that the SCM is acyclic, faithful, and has no latent confounding. To allow for latent confounding, the FCI algorithm has been proposed by \cite{spirtes1999algorithm}. It has later been shown to be sound and complete for acyclic graphs (with the $d$-separation Markov property and faithfulness assumption) by \cite{zhang2008completeness}, and subsequently has been shown to be sound and complete for cyclic SCMs as well with regard to the $\sigma$-separation Markov property and faithfulness assumption, in the setting without selection bias \citep{mooij2020constraintbased}. FCI does not directly estimate the underlying graph of the SCM, but instead estimates a \emph{Partial Ancestral Graph} (PAG), representing ancestral relations and the Markov equivalence class. We apply FCI to the repressilator in Example \ref{ex:cd} below, assuming a conditional independence oracle.
Other constraint-based causal discovery algorithms that are known to be sound for (possibly cyclic) systems with the $\sigma$-separation Markov property and faithfulness assumption are Local Causal Discovery \citep{cooper1997simple,mooij2020joint} and Y-structures \citep{mani2006bayesian,mooij2020joint}.

For cyclic systems with a $d$-separation Markov property and faithfulness assumption, as in the setting of Theorem~\ref{thm:dsep_mp}, the FCI algorithm is not sound. For this scenario (excluding latent confounding and selection bias), \cite{richardson1996discovery} introduced the CCD algorithm, which was subsequently extended to the CCI algorithm by \cite{strobl2019constraintbased} to allow for latent confounding and selection bias. The CCI algorithm outputs a \emph{Maximal Almost Ancestral Graph} (MAAG). A main drawback is that we do not know whether any sparsity in the MAAG represents conditional independencies in the data, since no Markov property has been proven for MAAGs. Until that question is answered, we can only read off the absence and presence of ancestral relations from the MAAG.

\begin{example}\label{ex:cd}
    Consider the repressilator $\bar\Dcal$ after projection onto $Z^{P}_{cI}, P_{cI}, P_{lacI}, P_{tetR}$ and $P_{GFP}$. After time-splitting into the intervals $\Ecal = \{[0,s), [s,t), [t,T]\}$, the time-split causal graph $G(\bar\Dcal^\Ecal)$ is depicted in Figure \ref{fig:cd:ground_truth}. When applying CCI to this system with a $d$-separation oracle (emulating a consistent conditional independence test with infinite data, under the $d$-separation Markov property and the faithfulness assumption), we obtain the MAAG as depicted in Figure \ref{fig:cd:cci}. When applying FCI (assuming no selection bias) to this system with a $\sigma$-separation oracle (so under the $\sigma$-separation Markov property and the faithfulness assumption), we obtain the PAG as depicted in Figure \ref{fig:cd:fci}.
\end{example}

\begin{figure}[!htb]
    \centering
    \begin{subfigure}{0.6\linewidth}
        \centering
        \begin{tikzpicture}[xscale=1.8,yscale=1.35]
            \node[] (13) at (0, 0) {$Z^{P,[0,s)}_{cI}$};
                \node[] (1) at (0, -1) {$P_{cI}^{[0,s)}$};
            \node[] (2) at (0.5, -2) {$P_{lacI}^{[0,s)}$};
                \node[] (3) at (0, -3) {$P_{tetR}^{[0,s)}$};
            \node[] (4) at (0, -4) {$P_{GFP}^{[0,s)}$};
                \node[] (14) at (1.5, 0) {$Z^{P,[s,t)}_{cI}$};
            \node[] (5) at (1.5, -1) {$P_{cI}^{[s,t)}$};
                \node[] (6) at (2, -2) {$P_{lacI}^{[s,t)}$};
            \node[] (7) at (1.5, -3) {$P_{tetR}^{[s,t)}$};
                \node[] (8) at (1.5, -4) {$P_{GFP}^{[s,t)}$};
            \node[] (15) at (3, 0) {$Z^{P,[t,T]}_{cI}$};
            \node[] (9) at (3, -1) {$P_{cI}^{[t,T]}$};
            \node[] (10) at (3.5, -2) {$P_{lacI}^{[t,T]}$};
            \node[] (11) at (3, -3) {$P_{tetR}^{[t,T]}$};
            \node[] (12) at (3, -4) {$P_{GFP}^{[t,T]}$};

            \draw[->,>=Latex,thick] (1) to (2);
            \draw[<-,>=Latex,thick] (1) to (3);
            \draw[->,>=Latex,thick] (1) to (5);
            \draw[<-,>=Latex,thick] (1) to (13);
            \draw[->,>=Latex,thick] (2) to (3);
            \draw[->,>=Latex,thick] (2) to (6);
            \draw[->,>=Latex,thick] (3) to (4);
            \draw[->,>=Latex,thick] (3) to (7);
            \draw[->,>=Latex,thick] (4) to (8);
            \draw[->,>=Latex,thick] (5) to (6);
            \draw[<-,>=Latex,thick] (5) to (7);
            \draw[->,>=Latex,thick] (5) to (9);
            \draw[<-,>=Latex,thick] (5) to (14);
            \draw[->,>=Latex,thick] (6) to (7);
            \draw[->,>=Latex,thick] (6) to (10);
            \draw[->,>=Latex,thick] (7) to (8);
            \draw[->,>=Latex,thick] (7) to (11);
            \draw[->,>=Latex,thick] (8) to (12);
            \draw[->,>=Latex,thick] (9) to (10);
            \draw[<-,>=Latex,thick] (9) to (11);
            \draw[<-,>=Latex,thick] (9) to (15);
            \draw[->,>=Latex,thick] (10) to (11);
            \draw[->,>=Latex,thick] (11) to (12);
            \draw[->,>=latex,thick] (1) to [loop arc rev={135}{60}{0.6cm}] (1);
            \draw[->,>=latex,thick] (2) to [loop arc={55}{40}{0.6cm}] (2);
            \draw[->,>=latex,thick] (3) to [loop arc rev={135}{60}{0.6cm}] (3);
            \draw[->,>=latex,thick] (4) to [loop arc rev={135}{60}{0.6cm}] (4);

            \draw[->,>=latex,thick] (5) to [loop arc rev={135}{60}{0.6cm}] (5);
            \draw[->,>=latex,thick] (6) to [loop arc={55}{40}{0.6cm}] (6);
            \draw[->,>=latex,thick] (7) to [loop arc rev={135}{60}{0.6cm}] (7);
            \draw[->,>=latex,thick] (8) to [loop arc rev={135}{60}{0.6cm}] (8);

            \draw[->,>=latex,thick] (9) to [loop arc rev={135}{60}{0.6cm}] (9);
            \draw[->,>=latex,thick] (10) to [loop arc={55}{40}{0.6cm}] (10);
            \draw[->,>=latex,thick] (11) to [loop arc rev={135}{60}{0.6cm}] (11);
            \draw[->,>=latex,thick] (12) to [loop arc rev={135}{60}{0.6cm}] (12);
        \end{tikzpicture}
        \caption{The ground truth.}
        \label{fig:cd:ground_truth}
    \end{subfigure}
    \begin{subfigure}{0.48\linewidth}
        \centering
        \begin{tikzpicture}[xscale=1.8,yscale=1.35]
            \node[] (13) at (0, 0) {$Z^{P,[0,s)}_{cI}$};
                \node[] (1) at (0, -1) {$P_{cI}^{[0,s)}$};
            \node[] (2) at (0.5, -2) {$P_{lacI}^{[0,s)}$};
                \node[] (3) at (0, -3) {$P_{tetR}^{[0,s)}$};
            \node[] (4) at (0, -4) {$P_{GFP}^{[0,s)}$};
                \node[] (14) at (1.5, 0) {$Z^{P,[s,t)}_{cI}$};
            \node[] (5) at (1.5, -1) {$P_{cI}^{[s,t)}$};
                \node[] (6) at (2, -2) {$P_{lacI}^{[s,t)}$};
            \node[] (7) at (1.5, -3) {$P_{tetR}^{[s,t)}$};
                \node[] (8) at (1.5, -4) {$P_{GFP}^{[s,t)}$};
            \node[] (15) at (3, 0) {$Z^{P,[t,T]}_{cI}$};
            \node[] (9) at (3, -1) {$P_{cI}^{[t,T]}$};
            \node[] (10) at (3.5, -2) {$P_{lacI}^{[t,T]}$};
            \node[] (11) at (3, -3) {$P_{tetR}^{[t,T]}$};
            \node[] (12) at (3, -4) {$P_{GFP}^{[t,T]}$};

            \draw[o-o,>=Latex,thick] (1) to (2);
            \draw[o-o,>=Latex,thick] (1) to (3);
            \draw[o->,>=Latex,thick] (1) to (5);
            \draw[o->,>=Latex,thick,bend left=10] (1) to (7);
            \draw[o-o,>=Latex,thick] (1) to (13);
            \draw[o-o,>=Latex,thick] (2) to (3);
            \draw[o->,>=Latex,thick] (2) to (5);
            \draw[o->,>=Latex,thick] (2) to (6);
            \draw[o-o,>=Latex,thick] (3) to (4);
            \draw[o->,>=Latex,thick] (3) to (6);
            \draw[o->,>=Latex,thick] (3) to (7);
            \draw[o-o,>=Latex,thick,bend left=18] (3) to (13);
            \draw[->,>=Latex,thick] (4) to (8);
            \draw[-,>=Latex,thick] (5) to (6);
            \draw[-,>=Latex,thick] (5) to (7);
            \draw[->,>=Latex,thick] (5) to (9);
            \draw[->,>=Latex,thick,bend left=10] (5) to (11);
            \draw[<-o,>=Latex,thick] (5) to (14);
            \draw[-,>=Latex,thick] (6) to (7);
            \draw[->,>=Latex,thick] (6) to (9);
            \draw[->,>=Latex,thick] (6) to (10);
            \draw[->,>=Latex,thick] (7) to (8);
            \draw[->,>=Latex,thick] (7) to (10);
            \draw[->,>=Latex,thick] (7) to (11);
            \draw[<-o,>=Latex,thick,bend left=18] (7) to (14);
            \draw[->,>=Latex,thick] (8) to (12);
            \draw[-,>=Latex,thick] (9) to (10);
            \draw[-,>=Latex,thick] (9) to (11);
            \draw[<-o,>=Latex,thick] (9) to (15);
            \draw[-,>=Latex,thick] (10) to (11);
            \draw[->,>=Latex,thick] (11) to (12);
            \draw[<-o,>=Latex,thick,bend left=18] (11) to (15);
        \end{tikzpicture}
        \caption{The MAAG computed by CCI.}
        \label{fig:cd:cci}
    \end{subfigure}
    \begin{subfigure}{0.48\linewidth}
        \centering
        \begin{tikzpicture}[xscale=1.8,yscale=1.35]
            \node[] (13) at (0, 0) {$Z^{P,[0,s)}_{cI}$};
                \node[] (1) at (0, -1) {$P_{cI}^{[0,s)}$};
            \node[] (2) at (0.5, -2) {$P_{lacI}^{[0,s)}$};
                \node[] (3) at (0, -3) {$P_{tetR}^{[0,s)}$};
            \node[] (4) at (0, -4) {$P_{GFP}^{[0,s)}$};
                \node[] (14) at (1.5, 0) {$Z^{P,[s,t)}_{cI}$};
            \node[] (5) at (1.5, -1) {$P_{cI}^{[s,t)}$};
                \node[] (6) at (2, -2) {$P_{lacI}^{[s,t)}$};
            \node[] (7) at (1.5, -3) {$P_{tetR}^{[s,t)}$};
                \node[] (8) at (1.5, -4) {$P_{GFP}^{[s,t)}$};
            \node[] (15) at (3, 0) {$Z^{P,[t,T]}_{cI}$};
            \node[] (9) at (3, -1) {$P_{cI}^{[t,T]}$};
            \node[] (10) at (3.5, -2) {$P_{lacI}^{[t,T]}$};
            \node[] (11) at (3, -3) {$P_{tetR}^{[t,T]}$};
            \node[] (12) at (3, -4) {$P_{GFP}^{[t,T]}$};

            \draw[o-o,>=Latex,thick] (1) to (2);
            \draw[o-o,>=Latex,thick] (1) to (3);
            \draw[o->,>=Latex,thick] (1) to (5);
            \draw[o->,>=Latex,thick] (1) to (6);
            \draw[o->,>=Latex,thick,bend left=10] (1) to (7);
            \draw[o-o,>=Latex,thick] (1) to (13);
            \draw[o-o,>=Latex,thick] (2) to (3);
            \draw[o->,>=Latex,thick] (2) to (5);
            \draw[o->,>=Latex,thick] (2) to (6);
            \draw[o->,>=Latex,thick] (2) to (7);
            \draw[o-o,>=Latex,thick,bend right=5] (2) to (13);
            \draw[o-o,>=Latex,thick] (3) to (4);
            \draw[o->,>=Latex,thick,bend right=10] (3) to (5);
            \draw[o->,>=Latex,thick] (3) to (6);
            \draw[o->,>=Latex,thick] (3) to (7);
            \draw[o-o,>=Latex,thick,bend left=18] (3) to (13);
            \draw[->,>=Latex,thick] (4) to (8);
            \draw[o-o,>=Latex,thick] (5) to (6);
            \draw[o-o,>=Latex,thick] (5) to (7);
            \draw[->,>=Latex,thick] (5) to (9);
            \draw[->,>=Latex,thick] (5) to (10);
            \draw[->,>=Latex,thick,bend left=10] (5) to (11);
            \draw[<-o,>=Latex,thick] (5) to (14);
            \draw[o-o,>=Latex,thick] (6) to (7);
            \draw[->,>=Latex,thick] (6) to (9);
            \draw[->,>=Latex,thick] (6) to (10);
            \draw[->,>=Latex,thick] (6) to (11);
            \draw[<-o,>=Latex,thick,bend right=5] (6) to (14);
            \draw[->,>=Latex,thick] (7) to (8);
            \draw[->,>=Latex,thick,bend right=10] (7) to (9);
            \draw[->,>=Latex,thick] (7) to (10);
            \draw[->,>=Latex,thick] (7) to (11);
            \draw[<-o,>=Latex,thick,bend left=18] (7) to (14);
            \draw[->,>=Latex,thick] (8) to (12);
            \draw[o-o,>=Latex,thick] (9) to (10);
            \draw[o-o,>=Latex,thick] (9) to (11);
            \draw[<-o,>=Latex,thick] (9) to (15);
            \draw[o-o,>=Latex,thick] (10) to (11);
            \draw[<-o,>=Latex,thick,bend right=5] (10) to (15);
            \draw[->,>=Latex,thick] (11) to (12);
            \draw[<-o,>=Latex,thick,bend left=18] (11) to (15);
        \end{tikzpicture}
        \caption{The PAG computed by FCI.}
        \label{fig:cd:fci}
    \end{subfigure}
    \caption{Outputs of the constraint-based causal discovery algorithms CCI and FCI applied to a subset of the variables of the repressilator. The exogenous noise process $Z^{P}_{cI}$ is included as a parent of $P_{cI}$ in the cycle, for sake of exposing the difference between CCI and FCI.}
    \label{fig:cd}
\end{figure}
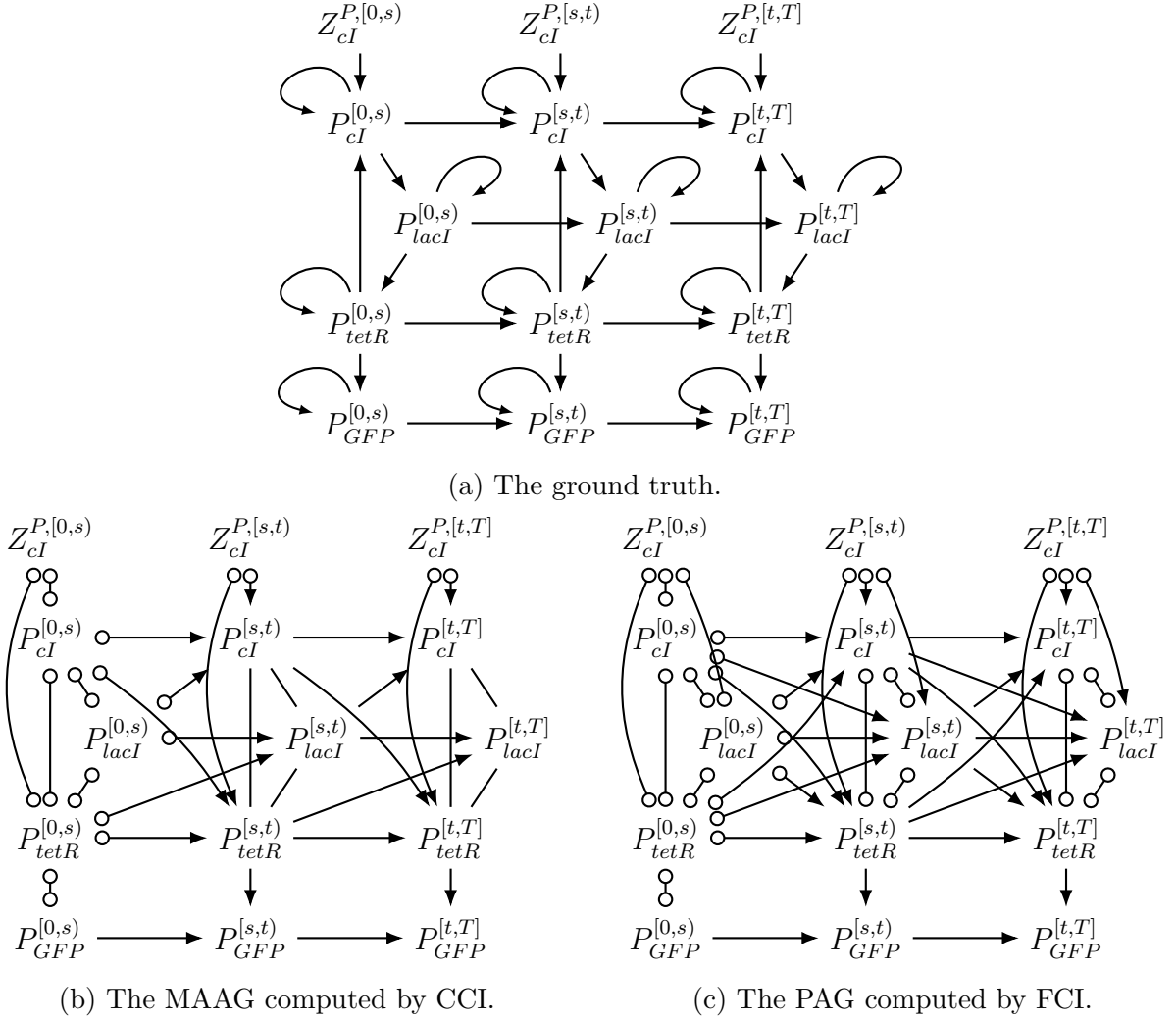

Other approaches to causal discovery for discrete-time systems include \cite{malinsky2018causala,runge2019detecting,reiter2024causal}; see \cite{assaad2022survey} for an overview.
Other approaches for causal discovery for systems of SDEs include \cite{guan2024identifying}, \cite{engelke2024levy}, \cite{nathaniel2025deep} and \cite{bruck2026graph}. Closely related to our framework, \cite{manten2025asymmetric} establish an asymmetric Markov property for the conditional independence $X_A^{[0,s]} \Indep X_B^{(s,t]} \given X_C^{[0,s]}, X_{C\setminus B}^{(s,t]}$, based on $\sigma$-separation in a `lifted graph' constructed from the causal graph $G(\Dcal)$; this lifted graph corresponds to the time-split causal graph $G(\Dcal^{\{[0,s], (s,t], (t, T]\}})$ if $f_v$ models the initial condition and $g_v$ is Markov. When Assumption~\ref{ass:additive_noise} holds, our $d$-separation Markov property in the time-split graph (Theorem~\ref{thm:markov_timesplit}(ii)) shows that their Markov property can be strengthened.


\subsection{Limitations and open problems}\label{sec:limitations}
Our framework assumes well-posedness of the underlying SDE system through a global Lipschitz condition, which might be extended to include more general classes of dynamics. We proved the $d$-separation Markov property for a class of additive-noise SDEs (Theorem~\ref{thm:dsep_mp}); we conjecture that the result holds for a larger class of systems of causal SDEs that excludes instantaneous cycles.

A central open direction is proving the $d$-separation Markov property for marginalised systems. This would boil down to finding a model class like Assumption~\ref{ass:additive_noise} for which we have a $d$-separation Markov property and which is itself closed under marginalisation.

Another open problem concerns the $d$-separation analogue of the do-calculus for cyclic SDEs. Theorem~\ref{thm:do_calculus} establishes the $\sigma$-separation do-calculus for essentially uniquely solvable systems of causal SDEs via an auxiliary system $\Dcal_{\Do(R_B\sim\nu)}$ augmented with intervention variables, and on acyclic systems the $d$-separation analogue follows trivially since $d$-separation and $\sigma$-separation coincide. For cyclic additive-noise SDEs the analogue does not follow from Theorem~\ref{thm:dsep_mp}, because the auxiliary system falls outside the additive-noise class. We conjecture that the $d$-separation do-calculus nevertheless holds for additive-noise SDEs satisfying Assumption~\ref{ass:additive_noise}.

On the causal-discovery side, several open problems are especially relevant for time-split systems of causal SDEs. Immediate ones are the extension of FCI with $\sigma$-separation to allow for selection bias, and investigating Markov properties for MAAGs and the completeness of CCI. An especially relevant extension is \emph{tiered FCI} (tFCI) by \cite{andrews2020completeness}, where one specifies a temporal ordering between the variables as background knowledge, excluding causal relations from future to past. Soundness of tFCI has only been proven for acyclic graphs with $d$-separation Markov property and faithfulness \citep{andrews2020completeness} without selection bias; it is unknown whether it is complete in this setting. Moreover, it is unknown whether tiered background knowledge can be used in FCI for cyclic systems with the $\sigma$-separation Markov property and faithfulness assumption, with and without selection bias. In the cyclic setting with the $d$-separation and faithfulness assumption, no extension of CCI has been proposed to incorporate tiered background knowledge. Lastly, it seems reasonable to believe that by using (tiered) FCI/CCI on a time-split system and then mapping back to the `summary' PAG/MAAG for $\Dcal$, we can identify more edges than straightforward FCI/CCI on $\Dcal$.

For causal discovery, a main bottleneck is consistent conditional independence testing for sample paths: if variables $X, Y, Z$ take values in $D(\Ical_X, \RR^k)$, $D(\Ical_Y, \RR^\ell)$ and $D(\Ical_Z, \RR^m)$ respectively, testing $X\Indep Y\given Z$ is not straightforward. Recent results for functional CI testing are proposed in \cite{lundborg2022conditional}, \cite{laumann2023kernelbased} and \cite{manten2024signature}. \cite{boeken2026topological} provide topological criteria for the existence of consistent conditional independence tests that may be useful for analysing this problem.

Beyond causal discovery, future work should explore statistical aspects of causal effect estimation; there is an opportunity to combine this formalism with existing statistical theory in this field. Finally, experimental validation on real-world dynamical systems -- such as gene regulatory networks, climate models, and financial systems -- will be essential to assess the practical utility of this approach.


\section*{Acknowledgements}
We thank Claude Code, Sonja Cox and Patrick Forré for helpful remarks and suggestions. Part of this work was carried out while Philip Boeken was at the Korteweg–de Vries Institute for Mathematics, University of Amsterdam. Philip Boeken was supported by Booking.com.

\begin{appendix}
	\section{Functional representation of semimartingale SDEs}\label{app:przybylowicz}
The pathwise interpretation of systems of causal SDEs (Definition~\ref{def:sde_pathwise}) and the solvability results of Section~\ref{sec:solvability} rest on the following functional-representation theorem of \cite{przybylowicz2024skorohod}, which we restate here for completeness. Recall that $x^{\wedge t}(s) := x(s \wedge t)$.

\begin{theoremapp}[\citealp{przybylowicz2024skorohod}, Theorem 3.1]\label{thm:przybylowicz_app}
    Let $m, d, r \in \NN$ and let $g : [0,T] \times D([0,T], \RR^d) \times D([0,T], \RR^r) \to \RR^{d \times m}$ be a matrix-valued map such that
    \begin{enumerate}[label=(\roman*)]
        \item $t \mapsto g(t, x, y)$ is càdlàg for all $(x, y) \in D([0,T], \RR^d) \times D([0,T], \RR^r)$;
        \item $g$ is $\Bcal([0,T]) \otimes \Bcal(D([0,T], \RR^d)) \otimes \Bcal(D([0,T], \RR^r))$-measurable;
        \item $g(t, x, y) = g(t, x^{\wedge t}, y^{\wedge t})$ for all $x, y$ and $t \in [0,T]$;
        \item $g$ satisfies the linear-growth and Lipschitz conditions \eqref{eqn:linear_growth}--\eqref{eqn:lipschitz} in $x$ given $y$: there exists a measurable $K : [0,T] \times D([0,T], \RR^r) \to (0, \infty)$ with $t \mapsto K(t, y)$ càdlàg, such that for all $x, x_1, x_2 \in D([0,T], \RR^d)$, $y \in D([0,T], \RR^r)$ and $t \in [0,T]$,
        \begin{align*}
            \|g(t, x, y)\|                  & \leq K(t, y)\bigl(1 + \sup_{0\leq s\leq t}\|x(s)\|\bigr), \\
            \|g(t, x_1, y) - g(t, x_2, y)\| & \leq K(t, y) \sup_{0\leq s\leq t}\|x_1(s) - x_2(s)\|.
        \end{align*}
    \end{enumerate}
    Then there exists a map
    \begin{equation*}
        \Psi : D([0,T], \RR^d) \times D([0,T], \RR^r) \times D([0,T], \RR^m) \to D([0,T], \RR^d),
    \end{equation*}
    measurable with respect to the Borel $\sigma$-algebras of the Skorokhod topology, such that for every $\RR^d$-valued càdlàg adapted process $F$, every $\RR^r$-valued càdlàg adapted process $G$, and every $\RR^m$-valued semimartingale $H$, the process $X := \Psi(F, G, H)$ satisfies $\PP$-almost surely the SDE
    \begin{equation*}
        X_t = F_t + \int_0^t g(s-, X, G) \diff H_s, \qquad t \in [0,T].
    \end{equation*}
    If $Y$ is another solution of the SDE, then $\PP(X = Y) = 1$.
\end{theoremapp}

By condition (iii) the map $\Psi$ is non-anticipative: $\Psi(F, G, H)^{\wedge t}$ depends on $(F, G, H)$ only through $(F^{\wedge t}, G^{\wedge t}, H^{\wedge t})$, so the induced solution is adapted. The pathwise integral map of Definition~\ref{def:sde_pathwise} arises as the special case $F \equiv 0$ and $g(s-, X, G) = G(s-)$.

\begin{corollaryapp}[Pathwise It\^o integral map]\label{cor:pathwise_integral}
    There exists an adapted map
    \begin{equation*}
        \Psi : D([0,T], \RR^m) \times D([0,T], \RR^m) \to D([0,T], \RR),
    \end{equation*}
    measurable with respect to the Borel $\sigma$-algebras of the Skorokhod topology, such that for every $\RR^m$-valued càdlàg adapted process $G$ and every $\RR^m$-valued semimartingale $H$, the process $\Psi(G, H)$ is a version of the It\^o integral $\int_0^\cdot G(s-) \diff H(s)$.
\end{corollaryapp}

\section{Proofs of Section \ref{sec:solvability}}
\obsintdist*
\begin{proof}
    Let $\Phi_O$ be as in Definition~\ref{def:sde_pathwise} and $I^{[V]}$ a solution function of $\Dcal$ w.r.t.\ $V$. Since $X_V := I^{[V]}(X_W)$ satisfies $X_V = \Phi_V(X_V, X_W)$ $\PP$-a.s., it is an adapted solution of $\Dcal$. Let $\varphi_V:\Xcal_W\to \Xcal_V$ be an adapted measurable map with $\varphi_V(X_W) = \Phi_V(\varphi_V(X_W), X_W)$ $\PP$-a.s. By essential unique solvability we have $\varphi_V = I^{[V]}(X_W)$ $\PP$-a.s. The law of any adapted solution therefore equals $\PP(I^{[V]}(X_W)) = \PP_{\Dcal}(\varphi_V(X_W))$, independent of the choice of $I^{[V]}$.

    For the interventional case, let $L, O, S \subseteq V$ partition $V$ and $x_S \in \Xcal_S$. Essential unique solvability of $\Dcal$ w.r.t.\ $O$ provides the fixed-point and uniqueness conditions of Definition~\ref{def:ess_unique_solvable} for every adapted input process $\varphi_{L\cup S} : \Xcal_W \to \Xcal_{L\cup S}$. Considering instead the map $\varphi_{L\cup S}^* = (\varphi_L, x_S)$ leaves both conditions intact: for each such $\varphi_L$ the process $X_O := I^{[O]}(x_S, \varphi_L(X_W), X_W)$ is the $\PP$-a.s.\ unique adapted solution of $X_O = \Phi_O(X_O, x_S, \varphi_L(X_W), X_W)$, which is exactly the $O$-fixed-point equation of $\Dcal_{\Do(X_S = x_S)}$. Hence $I^{[O]}(x_S, \cdot, \cdot)$ is a solution function of $\Dcal_{\Do(X_S = x_S)}$ w.r.t.\ $O$. Taking $L=\emptyset$, the fully intervened system $\Dcal_{\Do(X_S = x_S)}$ has the $\PP$-a.s.\ unique adapted solution $I^{[O]}(x_S, X_W)$, and hence
    \begin{equation*}
        \PP\bigl(X_O \given \Do(X_S = x_S)\bigr) = \PP(I^{[O]}(x_S, X_W)),
    \end{equation*}
    independent of the choice of $I^{[O]}$.

    The map $I^{[O]} : \Xcal_S \times \Xcal_W \to \Xcal_O$ is jointly measurable by Definition~\ref{def:ess_unique_solvable}. For every Borel set $A \subseteq \Xcal_O$, the indicator $(x_S, x_W) \mapsto \I_A(I^{[O]}(x_S, x_W))$ is jointly measurable, so by Tonelli's theorem the integral
    \begin{equation*}
        x_S \mapsto \PP_{\Dcal}\bigl(X_O \in A \given \Do(X_S = x_S)\bigr) = \int_{\Xcal_W} \I_A\bigl(I^{[O]}(x_S, x_W)\bigr) \diff\PP(X_W)(x_W)
    \end{equation*}
    is measurable in $x_S$. Since $A \mapsto \PP_{\Dcal}(X_O \in A \mid \Do(X_S = x_S))$ is a probability measure on $\Xcal_O$ for each fixed $x_S$ (as a pushforward of $\PP(X_W)$), the map $x_S \mapsto \PP_{\Dcal}(X_O \given \Do(X_S = x_S))$ is a Markov kernel from $\Xcal_S$ to $\Xcal_O$.
\end{proof}

\begin{lemapp}\label{lem:joint_vs_componentwise}
    Let $g : [0,T] \times D([0,T], \RR^k) \times D([0,T], \RR^\ell) \to \RR^m$ be measurable, and write $g = (g^{(i)})_{i=1}^m$ with each $g^{(i)} : [0,T] \times D([0,T], \RR^k) \times D([0,T], \RR^\ell) \to \RR$. Then $g$ satisfies the linear-growth condition \eqref{eqn:linear_growth} (resp.\ the Lipschitz condition \eqref{eqn:lipschitz}) in $x$ given $y$ if and only if each $g^{(i)}$ satisfies the corresponding condition in $x$ given $y$.
\end{lemapp}
\begin{proof}
    `$\Rightarrow$' For each $i$ and every $(t, x, y)$, $|g^{(i)}(t, x, y)| \leq \|g(t, x, y)\| \leq K(t, y)\bigl(1 + \sup_{0\leq s\leq t}\|x(s)\|\bigr)$, so $g^{(i)}$ satisfies \eqref{eqn:linear_growth} with the same $K$; the Lipschitz analogue is identical.

    `$\Leftarrow$' Suppose each $g^{(i)}$ satisfies \eqref{eqn:linear_growth} with measurable càdlàg $K_i$. Then
    \begin{equation*}
        \|g(t, x, y)\|^2 = \sum_{i=1}^m \bigl(g^{(i)}(t, x, y)\bigr)^2 \leq \sum_{i=1}^m K_i(t, y)^2 \bigl(1 + \sup_{0\leq s\leq t}\|x(s)\|\bigr)^2,
    \end{equation*}
    so $\|g(t, x, y)\| \leq K(t, y)\bigl(1 + \sup_{0\leq s\leq t}\|x(s)\|\bigr)$ with $K(t, y) := \bigl(\sum_{i=1}^m K_i(t, y)^2\bigr)^{1/2}$, which is measurable and càdlàg in $(t, y)$ as a finite combination of measurable càdlàg functions. The Lipschitz analogue is identical, with the same $K$.
\end{proof}

\itomap*
\begin{proof}
    Since $\alpha(v) \cap O = \emptyset$ for every $v \in O$, the joint SDE for $X_O$ reads
    \begin{equation}\label{eqn:itomap_joint_sde}
        X_O(t) = f_O\bigl(t, X_{\alpha(O)}\bigr) + \int_0^t g_O\bigl(s{-}, X_O, X_{\beta(O) \setminus O}\bigr) \diff h_O\bigl(s, X_{\gamma(O)}\bigr),
    \end{equation}
    with $\alpha(O) \subseteq (V \setminus O) \cup W$ (the assumption $\alpha(v) \cap O = \emptyset$), $\beta(O) \setminus O \subseteq (V \setminus O) \cup W$, and $\gamma(O) \subseteq W$ (Definition~\ref{def:sdes}). The componentwise Lipschitz and linear-growth conditions on $g_v$ in $(X_v, X_{\beta(v) \cap O})$ given $X_{\beta(v) \setminus O}$ extend trivially (by ignoring extra coordinates) to componentwise conditions on $g_v$ in $X_O$ given $X_{\beta(O) \setminus O}$, with the same $K_v$. Lemma~\ref{lem:joint_vs_componentwise} then gives joint Lipschitz and linear-growth conditions on $g_O := (g_v)_{v \in O}$ in $X_O$ given $X_{\beta(O) \setminus O}$, with joint constant $K_O(t, y) := \bigl(\sum_{v \in O} K_v(t, y_{\beta(v) \setminus O})^2\bigr)^{1/2}$. Theorem~\ref{thm:przybylowicz_app} then yields a measurable map $I^{[O]}$ such that $X_O = I^{[O]}(X_{\alpha(O)}, X_{\beta(O) \setminus O}, X_{\gamma(O)})$ is the $\PP$-a.s.\ unique adapted solution of \eqref{eqn:itomap_joint_sde}. Since this holds for every $X_W$-measurable adapted càdlàg process $X_{(\alpha(O) \cup (\beta(O)\setminus O)) \cap V}$, $I^{[O]}$ is thus a solution function as in Definition \ref{def:ess_unique_solvable}, and $\Dcal$ is essentially uniquely solvable w.r.t.\ $O$.
\end{proof}

\begin{lemapp}\label{lem:acy_factorisation}
    Let $\Dcal$ be a system of causal SDEs and $O \subseteq V$ such that $\Dcal$ is essentially uniquely solvable w.r.t.\ $O$. Then there exists an adapted solution function $\tilde I^{[O]} : \Xcal_{\pa(O)} \times \Xcal_W \to \Xcal_O$.
\end{lemapp}
\begin{proof}
    By essential unique solvability of $\Dcal$ w.r.t.\ $O$ there exists an adapted solution function $I^{[O]} : \Xcal_S \times \Xcal_W \to \Xcal_O$ where $S := V\setminus O$. Define $L := (V\setminus O)\setminus\pa(O)$, fix any constant $x_L^0 \in \Xcal_L$, and define
    \begin{equation*}
        \tilde I^{[O]}(x_{\pa(O)}, x_W) := I^{[O]}(x_{\pa(O)}, x_L^0, x_W),
    \end{equation*}
    which is measurable and inherits adaptedness from $I^{[O]}$.

    Let $\varphi_S$ be adapted and measurable and set $\varphi_S^0 := (\varphi_{\pa(O)}, x_L^0)$.
    Since none of $f_O, g_O, h_O$ depend on $x_L$, $\Phi_O$ does not depend on $x_L$. In particular, for every $x_O$ and $x_W$,
    \begin{equation*}
        \Phi_O\bigl(x_O, \varphi_S(x_W), x_W\bigr) = \Phi_O\bigl(x_O, \varphi_S^0(x_W), x_W\bigr).
    \end{equation*}
    Since $\tilde I^{[O]}(\varphi_{\pa(O)}(x_W), x_W) = I^{[O]}(\varphi_S^0(x_W), x_W)$ is an essentially unique fixed point of $x_O = \Phi_O\bigl(x_O, \varphi_S^0(x_W), x_W\bigr)$ and thus also for $x_O = \Phi_O\bigl(x_O, \varphi_S(x_W), x_W\bigr)$, we obtain the result.
\end{proof}

\begin{lemapp}[Composition of SCC solution functions]\label{lem:scc_composition}
    Let $\Dcal$ be a system of causal SDEs that is essentially uniquely solvable w.r.t.\ every strongly connected component (SCC) of $G(\Dcal)$. Then $\Dcal$ is essentially uniquely solvable w.r.t.\ $V$.
\end{lemapp}
\begin{proof}
    Let $(C_1, \ldots, C_k)$ be a topological ordering of the SCCs of $G(\Dcal)$, and for each $i$ let $I^{[C_i]}$ be a solution function of $\Dcal$ w.r.t.\ $C_i$. By Lemma~\ref{lem:acy_factorisation} the solution function $I^{[C_i]}$ may be taken to depend only on $X_W$ and the endogenous parents $X_{\pa(C_i)}$ that all lie in $C_1 \cup \cdots \cup C_{i-1}$. Writing $I^{[C_i]}(x_{\pa(C_i)}, x_W)$ accordingly, define $I^{[V]} : \Xcal_W \to \Xcal_V$ recursively along the topological order by
    \begin{equation*}
        I^{[V]}_{C_i}(x_W) := I^{[C_i]}\bigl(I^{[V]}_{\pa(C_i)}(x_W), x_W\bigr), \qquad i = 1, \ldots, k.
    \end{equation*}
    We verify by induction on the SCCs that $I^{[V]}$ is an essentially unique solution function of $\Dcal$ w.r.t.\ $V$. To that end, let $\varphi_V : \Xcal_W \to \Xcal_V$ be adapted with $\varphi_V(X_W) = \Phi_V(\varphi_V(X_W), X_W)$ $\PP$-a.s.
    Since $I^{[V]}_{C_1}(x_W) = I^{[C_1]}(x_W)$ we have
    \begin{equation*}
        I^{[V]}_{C_1}(X_W) = \Phi_{C_1}\bigl(I^{[V]}_{C_1}(X_W), X_W\bigr)
    \end{equation*}
    by the fixed-point property of $I^{[C_1]}$ and $\varphi_{C_1}(X_W) = I^{[V]}_{C_1}(X_W)$ $\PP$-a.s.\ by the essential uniqueness of $I^{[C_1]}$. If for $i \in \{2, ..., k\}$ we have $I^{[V]}_{C_j}(X_W) = \Phi_{C_j}\bigl(I^{[V]}_{C_j}(X_W), X_W\bigr)$ and $\varphi_{C_j}(X_W) = I^{[V]}_{C_j}(X_W)$ $\PP$-a.s. for all $j<i$, then, since $\pa(C_{i}) \subseteq C_1 \cup \cdots \cup C_{i-1}$ and $I^{[V]}_{C_i}(X_W) = I^{[C_i]}\bigl(I^{[V]}_{\pa(C_i)}(X_W), X_W\bigr)$, the fixed-point property of $I^{[C_i]}$ gives
    \begin{equation*}
        I^{[V]}_{C_i}(X_W) = \Phi_{C_i}\bigl(I^{[V]}_{C_i}(X_W), I^{[V]}_{\pa(C_i)}(X_W), X_W\bigr) \quad \PP\text{-a.s.}
    \end{equation*}
    Since $\varphi_{C_i}(X_W) = \Phi_{C_i}(\varphi_{C_i}(X_W), \varphi_{\pa(C_i)}(X_W), X_W)$ $\PP$-a.s., substituting the essential uniqueness $\varphi_{\pa(C_i)}(X_W) = I^{[V]}_{\pa(C_i)}(X_W)$ from the induction hypothesis, and since we have $I^{[V]}_{C_i}(x_W) = I^{[C_i]}\bigl(I^{[V]}_{\pa(C_i)}(x_W), x_W\bigr)$, we obtain by the essential uniqueness of $I^{[C_i]}$ that
    \begin{equation*}
        \varphi_{C_i}(X_W) = \Phi_{C_i}\bigl(\varphi_{C_i}(X_W), I^{[V]}_{\pa(C_i)}(X_W), X_W\bigr) = I^{[V]}_{C_i}(X_W) \quad \PP\text{-a.s.,}
    \end{equation*}
    which proves the result.
\end{proof}

\esssolvableunderass*
\begin{proof}
    Fix $O \subseteq V$ and write $S := V \setminus O$. By Definition~\ref{def:ess_unique_solvable}, a solution function of $\Dcal$ w.r.t.\ $O$ is a single measurable adapted map $I^{[O]} : \Xcal_S \times \Xcal_W \to \Xcal_O$ such that, for every measurable adapted $\varphi_S : \Xcal_W \to \Xcal_S$, the process $I^{[O]}(\varphi_S(X_W), X_W)$ is the essentially-unique adapted solution of
    \begin{equation*}
        x_O = \Phi_O\bigl(x_O, \varphi_S(X_W), X_W\bigr).
    \end{equation*}
    This can equivalently be interpreted as the fixed-point equation of the \emph{intervened system} $\Dcal_{\Do(X_S = \varphi_S(X_W))}$ w.r.t.\ its full endogenous set $O$: the system of causal SDEs on $O$ obtained by overriding the mechanisms of the variables in $S$ with the adapted process $\varphi_S(X_W)$, leaving each mechanism $\Phi_v$ ($v \in O$) unchanged. Thus $\Dcal$ is essentially uniquely solvable w.r.t.\ $O$ if and only if there is an essentially unique solution function $I^{[O]}$ for $\Dcal_{\Do(X_S = \varphi_S(X_W))}$ w.r.t.\ $O$ for every adapted $\varphi_S$; we establish the latter by verifying the conditions of Theorem~\ref{thm:ito_map} on each SCC of $G(\Dcal_{\Do(X_S = \varphi_S(X_W))})$ and composing via Lemma~\ref{lem:scc_composition}.

    Since $\varphi_S$ is a function only of the exogenous $X_W$, overriding $X_S$ by $\varphi_S(X_W)$ adds only exogenous dependence to the $O$-mechanisms and leaves the directed edges among $O$ exactly those of $G(\Dcal)$; hence the SCCs of $G(\Dcal_{\Do(X_S = \varphi_S(X_W))})$ are those of the subgraph of $G(\Dcal)$ induced on $O$, each contained in an SCC $\bar C$ of $G(\Dcal)$. Fix such an SCC $C \subseteq \bar C$. By Assumption~\ref{ass:uniquely_solvable}(i) on $\bar C$, $\alpha(v) \cap C \subseteq \alpha(v) \cap \bar C = \emptyset$ for every $v \in C$; and by Assumption~\ref{ass:uniquely_solvable}(ii) on $\bar C$, $g_v$ satisfies the linear-growth and Lipschitz conditions in $(X_v, X_{\beta(v) \cap \bar C})$ given $X_{\beta(v) \setminus \bar C}$, which implies the conditions in $(X_v, X_{\beta(v) \cap C})$ given $X_{\beta(v) \setminus C}$ by considering the variables $\beta(v) \cap (\bar C \setminus C)$ as the given inputs. Hence $C$ satisfies the hypotheses of Theorem~\ref{thm:ito_map}, and thus $\Dcal_{\Do(X_S = \varphi_S(X_W))}$ is essentially uniquely solvable w.r.t.\ $C$. Since the resulting solution function $I^{[O]}$ does not depend on $\varphi_S$ we obtain the result.
\end{proof}

\esssolvableunderadditive*
\begin{proof}
    For each $v \in V$, writing $\Sigma_{v\cdot}$ for the (constant) $v$-th row of $\Sigma$ and $\gamma(v) := \{w \in \gamma : \Sigma_{vw} \neq 0\}$ for the Brownian components driving $v$, the causal SDE
    \begin{equation*}
        X_v(t) = X_{\alpha_v}^0 + \int_0^t \mu_v(s, X_V(s)) \diff s + \int_0^t \Sigma_{v\cdot} \diff X_{\gamma(v)}(s)
    \end{equation*}
    matches the form $f_v(t, X_{\alpha(v)}) + \int_0^t g_v(s{-}, X_v, X_{\beta(v)}) \diff h_v(s, X_{\gamma(v)})$ of \eqref{eqn:sde} with $\alpha(v) := \{\alpha_v\} \subseteq \alpha \subseteq W$ the exogenous initial-condition coordinate, $\beta(v) := V \setminus \{v\}$, $\gamma(v) \subseteq \gamma \subseteq W$ the exogenous Brownian coordinates driving $v$, $f_v(t, X_{\alpha(v)}) := X_{\alpha_v}^0$, $g_v(s, X_v, X_{\beta(v)}) := (\mu_v(s, X_V(s)), \Sigma_{v\cdot})$, and $h_v(s, X_{\gamma(v)}) := (s, X_{\gamma(v)}(s))$.

    For every $v \in V$, $\alpha(v) \subseteq W$, so $v$ has no endogenous functional parents, and thus $\alpha(v) \cap S = \emptyset$ for every SCC $S$ and every $v \in S$.
    The drift component $\mu_v(s, X_V(s))$ inherits the linear-growth and Lipschitz conditions from $\mu$ by restriction. The constant component $\Sigma_{v\cdot}$ of $g_v$ has no $X$-dependence and does not contribute to the Lipschitz constant, and adds a term $\|\Sigma_{v\cdot}\|$ to the linear-growth constant.
\end{proof}

\section{Proofs of Section \ref{sec:marginalisation}}
\margclosuresimple*
\begin{proof}
    Fix the solution function $I^{[L]} : \Xcal_{V \setminus L} \times \Xcal_W \to \Xcal_L$ used in Definition~\ref{def:marginalisation}, and write $S := V \setminus (L \cup O)$. Since the integrator $h_v$ of each $v \in V \setminus L$ depends only on the exogenous $X_{\gamma(v)}$, marginalising $L$ leaves it unchanged and substitutes $I^{[L]}$ only into $f_v, g_v$; consequently the pathwise mechanism of $\Dcal_{\marg(L)}$ satisfies
    \begin{equation}\label{eqn:marg_mechanism}
        \tilde\Phi_O(x_O, x_S, x_W) = \Phi_O\bigl(I^{[L]}(x_S, x_O, x_W), x_O, x_S, x_W\bigr).
    \end{equation}

    Fix a solution function $I^{[L \cup O]}$ of $\Dcal$ w.r.t.\ $L \cup O$. We will show that $\tilde I^{[O]} := I^{[L \cup O]}_O$ is an essentially unique solution function for $\Dcal_{\marg(L)}$. To this end, let $\varphi_S : \Xcal_W \to \Xcal_S$ be measurable and adapted.

    By the fixed-point property of $I^{[L \cup O]}$ at $\varphi_S$ we have
    \begin{equation*}
        I^{[L \cup O]}(\varphi_S(X_W), X_W) = \Phi_{L \cup O}\bigl(I^{[L \cup O]}(\varphi_S(X_W), X_W), \varphi_S(X_W), X_W\bigr) \quad \PP(X_W)\text{-a.s.}
    \end{equation*}
    Reading off the $L$-block and regrouping its $O$- and $S$-arguments into the map $\varphi_{S\cup O}(x_W) := (\varphi_S(x_W), \tilde I^{[O]}(\varphi_S(x_W), x_W))$ we get that $I^{[L \cup O]}_L(\varphi_S(\cdot), \cdot)$ satisfies the fixed-point property
    \begin{equation*}
        I^{[L \cup O]}_L(\varphi_S(X_W), X_W) = \Phi_L\bigl(I^{[L \cup O]}_L(\varphi_S(X_W), X_W), \varphi_{S\cup O}(X_W), X_W\bigr) \quad \PP(X_W)\text{-a.s.}
    \end{equation*}
    Since $\Dcal$ is essentially uniquely solvable w.r.t.\ $L$, the uniqueness property of $I^{[L]}$ at $\varphi_{S\cup O}$ then gives
    \begin{equation}\label{eqn:marg_Lcoord}
        I^{[L \cup O]}_L(\varphi_S(x_W), x_W) = I^{[L]}(\varphi_{S\cup O}(x_W), x_W) \quad \PP(X_W)\text{-a.s.}
    \end{equation}
    Now the $O$-component of the fixed-point property of $I^{[L \cup O]}$ at $\varphi_S$ reads, using $\tilde I^{[O]} = I^{[L \cup O]}_O$ and $I^{[L \cup O]}(\varphi_S(X_W), X_W) = (I^{[L \cup O]}_L(\varphi_S(X_W), X_W), \tilde I^{[O]}(\varphi_S(X_W), X_W))$, as $\PP(X_W)$-a.s.:
    \begin{equation*}
        \tilde I^{[O]}(\varphi_S(X_W), X_W) = \Phi_O\bigl((I^{[L \cup O]}_L(\varphi_S(X_W), X_W),\ \tilde I^{[O]}(\varphi_S(X_W), X_W)),\ \varphi_S(X_W), X_W\bigr).
    \end{equation*}
    By \eqref{eqn:marg_Lcoord}, $I^{[L \cup O]}_L(\varphi_S(X_W), X_W) = I^{[L]}(\varphi_{S\cup O}(X_W), X_W) = I^{[L]}((\varphi_S(X_W), \tilde I^{[O]}(\varphi_S(X_W), X_W)), X_W)$; substituting this for the $L$-coordinate, the right-hand side is precisely \eqref{eqn:marg_mechanism} evaluated at $x_O = \tilde I^{[O]}(\varphi_S(X_W), X_W)$ and $x_S = \varphi_S(X_W)$. Hence, $\PP(X_W)$-a.s.,
    \begin{equation*}
        \tilde I^{[O]}(\varphi_S(x_W), x_W) = \tilde\Phi_O\bigl(\tilde I^{[O]}(\varphi_S(x_W), x_W), \varphi_S(x_W), x_W\bigr).
    \end{equation*}

    For essential uniqueness, let $\varphi_O : \Xcal_W \to \Xcal_O$ be adapted and measurable with $\varphi_O(x_W) = \tilde\Phi_O\bigl(\varphi_O(x_W), \varphi_S(x_W), x_W\bigr)$ for $\PP(X_W)$-a.a.\ $x_W$. Redefine $\varphi_{S\cup O} := (\varphi_S, \varphi_O)$ and set
    \begin{equation*}
        \varphi_{L\cup O}(x_W) := \bigl(I^{[L]}(\varphi_{S\cup O}(x_W), x_W),  \varphi_O(x_W)\bigr).
    \end{equation*}
    Then $\varphi_{L\cup O}$ is a fixed point of $x_{L\cup O} = \Phi_{L \cup O}(x_{L\cup O}, \varphi_S(X_W), X_W)$ $\PP(X_W)$-a.s., since the $L$-component satisfies
    \begin{equation*}
        \Phi_L(\varphi_{L\cup O}(x_W), \varphi_S(x_W), x_W) = \Phi_L\bigl(I^{[L]}(\varphi_{S\cup O}(x_W), x_W), \varphi_{S\cup O}(x_W), x_W\bigr) = I^{[L]}(\varphi_{S\cup O}(x_W), x_W)
    \end{equation*}
    by the fixed-point property of $I^{[L]}$ at $\varphi_{S\cup O}$, and the $O$-component satisfies $\Phi_O(\varphi_{L\cup O}(x_W), \varphi_S(x_W), x_W) = \tilde\Phi_O\bigl(\varphi_O(x_W), \varphi_S(x_W), x_W\bigr) = \varphi_O(x_W)$, by \eqref{eqn:marg_mechanism} and the hypothesis.
    Since $\varphi_{L\cup O}$ is adapted and $\Dcal$ is essentially uniquely solvable w.r.t.\ $L \cup O$, the essential uniqueness of $I^{[L \cup O]}$ at $\varphi_S$ gives $\varphi_{L\cup O}(X_W) = I^{[L \cup O]}(\varphi_S(X_W), X_W)$ $\PP(X_W)$-a.s. Reading off the $O$-component, we have
    \begin{equation*}
        \varphi_O(X_W) = I^{[L \cup O]}_O(\varphi_S(X_W), X_W) = \tilde I^{[O]}(\varphi_S(X_W), X_W)
    \end{equation*}
    $\PP(X_W)$-a.s.
\end{proof}

\begin{lemapp}\label{lem:marg_well_defined}
    Let $\Dcal$ be a system of causal SDEs, and let $L, O, S\subseteq V$ partition $V$. If $\Dcal$ is essentially uniquely solvable w.r.t.\ $L$ and $L \cup O$, then
    \begin{equation*}
        \PP_{\Dcal_{\marg(L)}}\bigl(X_O \given \Do(X_S = x_S)\bigr) = \PP_{\Dcal}\bigl(X_O \given \Do(X_S = x_S)\bigr),
    \end{equation*}
    and any two marginalisations $\Dcal_{\marg(L)}$ and $\Dcal'_{\marg(L)}$ yield the same distributions.
\end{lemapp}
\begin{proof}
    Let $I^{[L \cup O]} : \Xcal_S \times \Xcal_W \to \Xcal_{L \cup O}$ be a solution function of $\Dcal$ w.r.t.\ $L \cup O$, and set $\tilde I^{[O]}(x_S, x_W) := I^{[L \cup O]}_{O}(x_S, x_W)$. By Theorem~\ref{thm:marg_closure_simple}, $\Dcal_{\marg(L)}$ is essentially uniquely solvable w.r.t.\ $O$ and $\tilde I^{[O]}$ is a solution function of $\Dcal_{\marg(L)}$ w.r.t.\ $O$. By Theorem~\ref{thm:obs_int_dist} the interventional distribution of $\Dcal_{\marg(L)}$ is the pushforward via any solution function. Hence
    \begin{equation*}
        \PP_{\Dcal_{\marg(L)}}\bigl(X_{O} \given \Do(X_S = x_S)\bigr) = \PP(\tilde I^{[O]}(x_S,X_W)) = \PP(I^{[L \cup O]}_{O}(x_S, X_W)),
    \end{equation*}
    which by Theorem~\ref{thm:obs_int_dist} applied to $\Dcal$ equals $\PP_{\Dcal}(X_{O} \given \Do(X_S = x_S))$.
\end{proof}

\begin{lemapp}[Sequential marginalisation]\label{lem:marginalisation_steps}
    Let $\Dcal$ be a system of causal SDEs, let $L_1, L_2, O, S \subseteq V$ partition $V$. Assume that $\Dcal$ is essentially uniquely solvable w.r.t.\ $L_1$, $L_1 \cup L_2$, and $L_1 \cup L_2 \cup O$. Then $\Dcal_{\marg(L_1 \cup L_2)}$ and $(\Dcal_{\marg(L_1)})_{\marg(L_2)}$ are well-defined and essentially uniquely solvable w.r.t.\ $O$, and for every $x_S \in \Xcal_S$,
    \begin{equation*}
        \PP_{(\Dcal_{\marg(L_1)})_{\marg(L_2)}}\bigl(X_O \given \Do(X_S = x_S)\bigr) = \PP_{\Dcal_{\marg(L_1 \cup L_2)}}\bigl(X_O \given \Do(X_S = x_S)\bigr).
    \end{equation*}
\end{lemapp}
\begin{proof}
    Let $I^{[L_1 \cup L_2 \cup O]}$ be a solution function of $\Dcal$ w.r.t.\ $L_1 \cup L_2 \cup O$. Throughout we apply Theorem~\ref{thm:marg_closure_simple}, whose two hypotheses on a system are essential unique solvability w.r.t.\ the marginalised set and w.r.t.\ its union with the retained set. Applied to $\Dcal$ with marginalised set $L_1 \cup L_2$ and retained set $O$, it gives that $\Dcal_{\marg(L_1 \cup L_2)}$ is essentially uniquely solvable w.r.t.\ $O$ with solution function $I^{[L_1 \cup L_2 \cup O]}_O$. Applied to $\Dcal$ with marginalised set $L_1$ and retained set $L_2$, it gives that $\Dcal_{\marg(L_1)}$ is essentially uniquely solvable w.r.t.\ $L_2$, so $(\Dcal_{\marg(L_1)})_{\marg(L_2)}$ is defined; applied to $\Dcal$ with marginalised set $L_1$ and retained set $L_2 \cup O$, it gives that $I^{[L_1 \cup L_2 \cup O]}_{L_2 \cup O}$ is a solution function of $\Dcal_{\marg(L_1)}$ w.r.t.\ $L_2 \cup O$. Applied once more to $\Dcal_{\marg(L_1)}$ with marginalised set $L_2$ and retained set $O$, it yields $I^{[L_1 \cup L_2 \cup O]}_O$ as a solution function of $(\Dcal_{\marg(L_1)})_{\marg(L_2)}$ w.r.t.\ $O$.

    Both $\Dcal_{\marg(L_1 \cup L_2)}$ and $(\Dcal_{\marg(L_1)})_{\marg(L_2)}$ therefore admit $I^{[L_1 \cup L_2 \cup O]}_O$ as a solution function w.r.t.\ $O$, and by Theorem~\ref{thm:obs_int_dist} the interventional distributions coincide as pushforwards of $\PP(X_W)$ under this map.
\end{proof}

\begin{lemapp}[Marginalisation commutes with intervention]\label{lem:intervention_marg_commute}
    Let $\Dcal$ be a system of causal SDEs, let $L, O , S, T \subseteq V$ partition $V$, and let $x_T \in \Xcal_T$. Assume $\Dcal$ is essentially uniquely solvable w.r.t.\ $L$ and w.r.t.\ $L \cup O$. Then $(\Dcal_{\marg(L)})_{\Do(X_T = x_T)}$ and $(\Dcal_{\Do(X_T = x_T)})_{\marg(L)}$ are well-defined and essentially uniquely solvable w.r.t.\ $O$, and for every $x_S \in \Xcal_S$,
    \begin{equation*}
        \PP_{(\Dcal_{\marg(L)})_{\Do(X_T = x_T)}}\bigl(X_O \given \Do(X_S = x_S)\bigr) = \PP_{(\Dcal_{\Do(X_T = x_T)})_{\marg(L)}}\bigl(X_O \given \Do(X_S = x_S)\bigr).
    \end{equation*}
\end{lemapp}
\begin{proof}
    Let $I^{[L \cup O]} : \Xcal_{T \cup S} \times \Xcal_W \to \Xcal_{L \cup O}$ be a solution function of $\Dcal$ w.r.t.\ $L \cup O$.

    By Theorem \ref{thm:obs_int_dist}, since $\Dcal$ is essentially uniquely solvable w.r.t.\ $L$, $\Dcal_{\Do(X_T = x_T)}$ is essentially uniquely solvable w.r.t.\ $L$, and thus the marginalisation $(\Dcal_{\Do(X_T = x_T)})_{\marg(L)}$ is well-defined. Similarly $\Dcal$ and thus also $\Dcal_{\Do(X_T = x_T)}$ are essentially uniquely solvable w.r.t.\ $L \cup O$, where the latter has solution function $I^{[L \cup O]}(x_T, \cdot, \cdot) : \Xcal_S \times \Xcal_W \to \Xcal_{L \cup O}$. By Theorem~\ref{thm:marg_closure_simple}, $(\Dcal_{\Do(X_T = x_T)})_{\marg(L)}$ is then essentially uniquely solvable w.r.t.\ $O$ with solution function $I^{[L \cup O]}_O(x_T, \cdot, \cdot)$.

    By Theorem~\ref{thm:marg_closure_simple}, $\Dcal_{\marg(L)}$ is essentially uniquely solvable w.r.t.\ $O$ with solution function $I^{[L \cup O]}_O : \Xcal_{T \cup S} \times \Xcal_W \to \Xcal_O$. By Theorem \ref{thm:obs_int_dist}, the system $(\Dcal_{\marg(L)})_{\Do(X_T = x_T)}$ is essentially uniquely solvable w.r.t.\ $O$ with solution function $I^{[L \cup O]}_O(x_T, \cdot, \cdot) : \Xcal_S \times \Xcal_W \to \Xcal_O$.

    Both systems therefore admit $I^{[L \cup O]}_O(x_T, \cdot, \cdot) : \Xcal_S \times \Xcal_W \to \Xcal_O$ as a solution function w.r.t.\ $O$, so by Theorem~\ref{thm:obs_int_dist} their interventional distributions under $\Do(X_S = x_S)$ coincide, both being the pushforward of $\PP(X_W)$ under $I^{[L \cup O]}_O(x_T, x_S, \cdot)$.
\end{proof}

\section{Proofs of Section \ref{sec:markov_properties}}

\begin{theoremapp}\label{thm:self_cyclic_graph_separations}
    Let $G = (V, E)$ be a directed graph and let $G^-\subseteq G$ be the directed graph with all self-loops $v\to v$ removed. Let $A, B, C\subseteq V$, then we have
    \begin{align*}
        A \Perp^d_{G} B \given C      & \iff A \Perp^d_{G^-} B \given C      \\
        A \Perp^\sigma_{G} B \given C & \iff A \Perp^\sigma_{G^-} B \given C
    \end{align*}
\end{theoremapp}
\begin{proof}
    Since $G^-$ is a subgraph of $G$, any separation in $G$ immediately implies a separation in $G^-$.

    For the implication $A \nPerp^d_{G} B \given C \implies A \nPerp^d_{G^-} B \given C$ we follow the proof of Proposition 3.5 in \cite{mogensen2020markov}. Let $\pi$ be an active walk in $G$. For any self-loop at $v$ on the walk, let $\pi'$ be the walk in which this self-loop is removed. If $v\in A$ or $v\in B$, then the walk $\pi'$ is still active. Otherwise, if $v$ is a non-collider in $\pi'$, we must have had one of $...\to v \to v \to ...$, $...\to v \ot v \to ...$, $...\ot v \to v \to ...$, $...\ot v \ot v \to ...$, $...\ot v \to v \ot ...$ or $...\ot v \ot v \ot ...$ in $\pi$; in each of these the occurrence of $v$ at the tail of the self-loop is a non-collider, so we must have had $v\notin C$ for $\pi$ to be active, so $\pi'$ is active. If $v$ is a collider in $\pi'$, then in $\pi$ we either had $...\to v \ot v \ot ...$ or $...\to v \to v \ot ...$, so $v\in \Anc(C)$ for these walks to be open, so $\pi'$ is open as well. Iterating this for every self-loop on $\pi$ gives an active walk in $G^-$, and so we get $A \nPerp^d_{G^-} B \given C$. For $\sigma$-separation, the same case analysis applies. The only additional consideration is blockability of non-colliders: a non-collider $v$ is blockable if it has a child on the walk not in the same SCC. Removing self-loops does not affect the SCCs, and therefore the SCCs of $G^-$ are identical to those of $G$. In particular, blockability of non-colliders is the same in $G$ and $G^-$, so the argument carries over.
\end{proof}

\subsection{Proofs of Section \ref{sec:sigmasep_mp}}
We carry out the proof of Theorem~\ref{thm:sigma_sep_mp} via the acyclification strategy of \cite{forre2017markov,bongers2021foundations,forre2025mathematical} adapted to the essentially uniquely solvable setting.
\begin{definitionapp}[Acyclification]\label{def:weak_acyclification}
    Let $\Dcal$ be a system of causal SDEs that is essentially uniquely solvable w.r.t.\ every SCC $C$ of $G(\Dcal)$. An \emph{acyclification} is the SCM $\Dcal^{\mathrm{acy}} = \angs{V, W, \Xcal_V, \Xcal_W, \tilde\Phi, \PP(X_W)}$ with sample spaces $\Xcal_v = D([0,T], \RR)$ and, for each SCC $C$ of $G(\Dcal)$ and each $v\in C$, structural equation
    \begin{equation*}
        x_v = I^{[C]}_v(x_{\pa(C)}, x_W),
    \end{equation*}
    where $I^{[C]}$ is an adapted solution function from Lemma~\ref{lem:acy_factorisation}.
\end{definitionapp}
By construction, the acyclification $\Dcal^{\mathrm{acy}}$ is an acyclic SCM \citep{pearl2009causality}, and hence essentially uniquely solvable with respect to $V$ by recursive substitution of the structural equations. Moreover, it satisfies the $d$-separation Markov property \citep[Theorem 6.3]{bongers2021foundations}.

Recall from \cite[Definition 3.5.1]{forre2025mathematical} that a \emph{graphical acyclification} $\tilde G$ of a graph $G$ on $V$ is for example obtained from $G$ by, for each SCC $C$, replacing all edges within $C$ and all directed edges $u \to v$ into $C$ (with $u \notin C$, $v \in C$) by bidirected edges between every pair of distinct vertices in $C$ and directed edges $u \to v'$ for every $v' \in C$, respectively. The resulting graph is acyclic on the SCCs of $G$, and by \cite{forre2025mathematical}, Proposition 3.5.2, $d$-separation in $\tilde G$ corresponds to $\sigma$-separation in $G$: for all $A, B, C \subseteq V$ we have $A \Perp_{G}^\sigma B \given C \iff A \Perp_{\tilde G}^d B \given C$.

\begin{lemapp}\label{lem:weak_acy}
    Let $\Dcal$ be a system of causal SDEs that is essentially uniquely solvable w.r.t.\ every SCC of $G(\Dcal)$, and $\Dcal^{\mathrm{acy}}$ an acyclification. Then:
    \begin{enumerate}[label=(\roman*)]
        \item $\PP_{\Dcal}(X_V, X_W) = \PP_{\Dcal^{\mathrm{acy}}}(X_V, X_W)$;
        \item $G^+(\Dcal^{\mathrm{acy}}) \subseteq \tilde G^+(\Dcal)$, the graphical acyclification of $G^+(\Dcal)$.
    \end{enumerate}
\end{lemapp}
\begin{proof}
    Let $C_1 < \cdots < C_k$ be a topological ordering of the DAG of SCCs of $G(\Dcal)$.

    \emph{(i)}
    Define $I^{[V]} : \Xcal_W \to \Xcal_V$ recursively along the topological order of SCCs:
    \begin{equation*}
        I^{[V]}_{C_i}(x_W) := I^{[C_i]}\bigl(I^{[V]}_{\pa(C_i)}(x_W), x_W\bigr), \qquad i = 1, \ldots, k.
    \end{equation*}
    By acyclicity of $G(\Dcal^{\mathrm{acy}})$, this recursive construction yields an adapted solution function of $\Dcal^{\mathrm{acy}}$. As shown in the proof of Lemma~\ref{lem:scc_composition} the same $I^{[V]}$ is an essentially unique solution function of $\Dcal$ w.r.t.\ $V$. Hence under both $\Dcal$ and $\Dcal^{\mathrm{acy}}$ the endogenous variables are given $\PP$-a.s.\ by the same measurable map $X_V = I^{[V]}(X_W)$ of the common exogenous $X_W$, so the joint law of $(X_V, X_W)$ is in either case the pushforward of $\PP(X_W)$ under $x_W \mapsto (I^{[V]}(x_W), x_W)$; therefore $\PP_{\Dcal}(X_V, X_W) = \PP_{\Dcal^{\mathrm{acy}}}(X_V, X_W)$.

    \emph{(ii)}
    Let $v \in V$ and write $C$ for the SCC of $v$ in $G(\Dcal)$. By construction, $I^{[C]}_v$ depends only on $x_{\pa(C)}$ and $x_W$, so any directed edge $u \to v$ in $G^+(\Dcal^{\mathrm{acy}})$ has $u \in \pa(C) \cup W$. If $u \in \pa(C)$, then by definition of $\pa(C)$ there exists $v' \in C$ with $u \to v'$ in $G(\Dcal)$, hence $u \to v$ lies in $\tilde G^+(\Dcal)$. If $u \in W$, then since $I^{[C]}$ is the essentially unique solution of $X_C = \Phi_C(X_C, X_{\pa(C)}, X_W)$, its component $I^{[C]}_v$ can depend on $x_u$ only if some mechanism $\Phi_{v'}$ with $v' \in C$ essentially depends on $x_u$, that is, $u \to v'$ in $G^+(\Dcal)$ for some $v' \in C$; the graphical acyclification therefore puts $u \to v$ in $\tilde G^+(\Dcal)$.
\end{proof}

\sigmasepmarkov*
\begin{proof}
    Let $\tilde G^+(\Dcal)$ be a graphical acyclification of $G^+(\Dcal)$. By \cite[Proposition 3.5.2]{forre2025mathematical}, $A \Perp_{G^+(\Dcal)}^\sigma B \given C \iff A \Perp_{\tilde G^+(\Dcal)}^d B \given C$. By Lemma~\ref{lem:weak_acy}(ii), $G^+(\Dcal^{\mathrm{acy}}) \subseteq \tilde G^+(\Dcal)$, and removing edges preserves $d$-separation so $A \Perp_{\tilde G^+(\Dcal)}^d B \given C \implies A \Perp_{G^+(\Dcal^{\mathrm{acy}})}^d B \given C$. Since $G(\Dcal^{\mathrm{acy}})$ is acyclic, the $d$-separation global Markov property for acyclic SCMs \cite[Theorem 6.3]{bongers2021foundations} applies:
    \begin{equation*}
        A \Perp_{G^+(\Dcal^{\mathrm{acy}})}^d B \given C \implies X_A \Indep_{\PP_{\Dcal^{\mathrm{acy}}}} X_B \given X_C.
    \end{equation*}
    By Lemma~\ref{lem:weak_acy}(i), $\Dcal$ and $\Dcal^{\mathrm{acy}}$ induce the same joint law $\PP(X_V, X_W)$, so the conditional independence---which may involve exogenous variables, as $A, B, C \subseteq V \cup W$---transfers to $\Dcal$.
\end{proof}

\subsection{Proofs of Section \ref{sec:dsep_mp}}

\dsepsummaryeuler*
\begin{proof}
    Let $A, B, C \subseteq V$ be such that $X_A^\Delta \nPerp_{G(\Mcal_\Dcal^\Delta)}^d X_B^\Delta \given X_C^\Delta$, then we also have a $d$-connection in the augmented graph $X_A^\Delta \nPerp_{G^+(\Mcal_\Dcal^\Delta)}^d X_B^\Delta \given X_C^\Delta$, so let $\pi^{\Delta} = (v_0^{t_0}, ..., v_n^{t_n})$ be an active path in $G^+(\Mcal_\Dcal^\Delta)$ from $X_A^\Delta$ to $X_B^\Delta$ (given $X_C^\Delta$), where for ease of notation we write for any component $B_v^{[t_k, t_{k+1}]} = v^{t_{k+1}}$, for $v\in W$. This induces a walk $\pi := (v_0, ..., v_n)$ that is active given $C$:
    \begin{itemize}
        \item for every collider $v$ on $\pi$, there is a $t_j$ such that $v^{t_j}$ is a collider on $\pi^\Delta$, and since $v^{t_j} \in \Anc(X_C^\Delta)$, a directed path from $v^{t_j}$ to some $c^{t_i}$ with $c \in C$ in $G^+(\Mcal_\Dcal^\Delta)$ projects onto a directed path from $v$ to $c$ in $G^+(\Dcal)$, so $v\in \Anc(C)$;
        \item for every non-collider $v$ on $\pi$, we have
        \begin{itemize}
            \item $w\to v\to z$ with possibly $w=v$ and possibly $z=v$, and a corresponding non-collider $v^{t_j}$ in $w^{t_{j-1}} \to v^{t_j} \to z^{t_{j+1}}$ we have $v^{t_j}\notin X_C^{\Delta}$ and hence $v\notin C$;
            \item $w\ot v\to z$ then either $v\notin V$ so $v\notin C$, or if $v\in V$ then (with possibly either $w=v$ or $z=v$) there is a corresponding non-collider $v^{t_j}$ in $w^{t_{j+1}} \ot v^{t_j} \to z^{t_{j+1}}$ we have $v^{t_j}\notin X_C^{\Delta}$ and hence $v\notin C$;
            \item we have $v_0^{t_0}\notin X_C^\Delta$ and $v_n^{t_n}\notin X_C^\Delta$, and hence $v_0 \notin C$ and $v_n \notin C$.
        \end{itemize}
    \end{itemize}
    From $\pi$, construct the walk $\pi'$ in $G^+(\Dcal)$ between $A$ and $B$ by replacing every maximal subwalk of the form $v_j \to \cdots \to v_j$ by the single vertex $v_j$, which retains the boundary edges the subwalk had with the rest of $\pi$. Such a subwalk is a directed cycle $v_j \to w_1 \to \cdots \to w_k \to v_j$, visiting $v_j$ at its start and end. Its first occurrence of $v_j$ is the tail of the cycle --- incident to the outgoing edge $v_j \to w_1$ --- so it is a non-collider in $\pi$ and $v_j \notin C$; hence if $v_j$ is a non-collider in $\pi'$, $\pi'$ stays active there. If $v_j$ is a collider in $\pi'$ (both boundary edges point into $v_j$), then its last occurrence is the head of the cycle --- incident to the incoming edge $w_k \to v_j$ --- and, together with the incoming boundary edge, has two incoming edges, so it is a collider in $\pi$, giving $v_j \in \Anc(C)$. The interior cycle vertices are non-colliders on a directed path, hence not in $C$, so removing them does not affect activeness. (This is the self-loop argument of Theorem~\ref{thm:self_cyclic_graph_separations} (above), applied to the directed cycle.) Hence $\pi'$ is active given $C$, and so $X_A \nPerp_{G^+(\Dcal)}^d X_B \given X_C$, or equivalently $X_A \nPerp_{G(\Dcal)}^d X_B \given X_C$.

    Since $\Mcal_\Dcal^\Delta$ is acyclic, the $d$-separation Markov property gives $X_A^\Delta \Indep X_B^\Delta \given X_C^\Delta$.
\end{proof}

\begin{lemapp}\label{lem:euler_decomposition}
    Let $\Dcal$ satisfy Assumption~\ref{ass:additive_noise} and fix a grid $t_k = Tk/n$ ($k = 0, \ldots, n$). For each cell $[t_k, t_{k+1}]$ let $\bar X_\gamma^k$ be the Brownian bridge of $X_\gamma$,
    \begin{equation*}
        \bar X_\gamma^k(t) := \bigl(X_\gamma(t) - X_\gamma(t_k)\bigr) - \tfrac{t-t_k}{t_{k+1}-t_k}\bigl(X_\gamma(t_{k+1}) - X_\gamma(t_k)\bigr),
    \end{equation*}
    and set $\tilde X_V^k := \Sigma\bar X_\gamma^k$; for $S \subseteq V$ write $\tilde X_S^k := (\tilde X_v^k)_{v \in S}$ and $\tilde X_S := (\tilde X_S^k)_{k=0}^{n-1}$. Then the continuous Euler scheme $X_V^n$ of~\eqref{eqn:euler_continuous} satisfies, for every $S \subseteq V$,
    \begin{equation*}
        \sigma(X_S^n) = \sigma(X_S^\Delta) \vee \sigma(\tilde X_S).
    \end{equation*}
    Moreover, $X_V^\Delta$ and $\tilde X_V$ are independent.
\end{lemapp}
\begin{proof}
    Write $\Delta_k X_\gamma := X_\gamma(t_{k+1}) - X_\gamma(t_k)$. At the grid points $X_V^n(t_k) = X_V^\Delta(t_k)$, so on $[t_k, t_{k+1}]$ the continuous Euler scheme~\eqref{eqn:euler_continuous} reads $X_V^n(t) = X_V^\Delta(t_k) + \mu(t_k, X_V^\Delta(t_k))(t-t_k) + \Sigma\bigl(X_\gamma(t) - X_\gamma(t_k)\bigr)$. Substituting the gridpoint recursion $\mu(t_k, X_V^\Delta(t_k))(t_{k+1}-t_k) = X_V^\Delta(t_{k+1}) - X_V^\Delta(t_k) - \Sigma\Delta_k X_\gamma$ and collecting terms,
    \begin{equation*}
        X_V^n(t) = X_V^\Delta(t_k) + \tfrac{t-t_k}{t_{k+1}-t_k}\bigl(X_V^\Delta(t_{k+1}) - X_V^\Delta(t_k)\bigr) + \tilde X_V^k(t).
    \end{equation*}
    For $C\subseteq V$, $X_C^n$ is determined by $X_C^\Delta$ and $\tilde X_C$, and conversely $X_C^\Delta(t_k) = X_C^n(t_k)$ and $\tilde X_C^k(t) = X_C^n(t) - X_C^n(t_k) - \tfrac{t-t_k}{t_{k+1}-t_k}\bigl(X_C^n(t_{k+1}) - X_C^n(t_k)\bigr)$, so $\sigma(X_C^n) = \sigma(X_C^\Delta, \tilde X_C)$.

    On each cell the Brownian bridge $\bar X_\gamma^k$ and the increment $\Delta_k X_\gamma$ are jointly Gaussian with $\mathrm{Cov}\bigl(\bar X_\gamma^k(t), \Delta_k X_\gamma\bigr) = (t - t_k) I_d - \tfrac{t-t_k}{t_{k+1}-t_k}(t_{k+1}-t_k) I_d = 0$ for every $t \in [t_k, t_{k+1}]$, hence independent; bridges and increments on disjoint cells are independent by the independence of Brownian increments. Together with the independence of $X_\alpha^0$ and $X_\gamma$, the bridges $\tilde X_V$ are independent of $\sigma(X_\alpha^0, \{\Delta_k X_\gamma\}_k) \supseteq \sigma(X_V^\Delta)$, i.e.\ $X_V^\Delta \Indep \tilde X_V$.
\end{proof}

\begin{lemapp}\label{lem:mp_brownian_bridge}
    Let $\Dcal$ satisfy Assumption~\ref{ass:additive_noise}, and let $\tilde X_V$ be the Brownian bridges of Lemma~\ref{lem:euler_decomposition}. Then for all $A, B, C \subseteq V$,
    \begin{equation*}
        X_A \Perp^d_{G(\Dcal)} X_B \given X_C \implies \tilde X_A \Indep \tilde X_B \given \tilde X_C.
    \end{equation*}
\end{lemapp}
\begin{proof}
    Fix an interval $k$. The components $\bar X_{\gamma_w}^k$ ($w \in \gamma$) of the vector Brownian bridge $\bar X_\gamma^k$ are mutually independent and jointly Gaussian (as $X_\gamma$ has independent components), and $\tilde X_v^k = \sum_{w\in\gamma}\Sigma_{vw}\bar X_{\gamma_w}^k$. Hence $\tilde X_V^k = (\tilde X_v^k)_{v\in V}$ is the solution of the acyclic linear structural causal model with mutually independent exogenous variables $(\bar X_{\gamma_w}^k)_{w\in\gamma}$ and structural equations $\tilde X_v^k = \Sigma_{v\cdot}\bar X_\gamma^k$, whose latent projection onto $V$ has no directed edges and a bidirected edge $u \leftrightarrow v$ exactly when $u$ and $v$ share a driving component, i.e.\ $\Sigma_{uw} \neq 0 \neq \Sigma_{vw}$ for some $w \in \gamma$. Since the initial-condition exogenous variables $\alpha$ are in bijection with $V$ and contribute no bidirected edges, these coincide with the bidirected edges of $G(\Dcal)$ (Definition~\ref{def:sde_graph}); as the projection has no further edges, it is a subgraph of $G(\Dcal)$. Removing edges preserves $d$-separation, so $X_A \Perp^d_{G(\Dcal)} X_B \given X_C$ implies the $d$-separation of $A$ and $B$ given $C$ in this projection, and the $d$-separation global Markov property for acyclic SCMs \citep[Theorem 6.3]{bongers2021foundations} gives $\tilde X_A^k \Indep \tilde X_B^k \given \tilde X_C^k$. Since the bridges on distinct intervals are independent, $\tilde X_A \Indep \tilde X_B \given \tilde X_C$.
\end{proof}

\continuouseulerdsepmp*
\begin{proof}
    By Lemma~\ref{thm:dsep_summary_euler}, $X_A \Perp^d_{G(\Dcal)} X_B \given X_C$ implies $X_A^\Delta \Indep X_B^\Delta \given X_C^\Delta$ and by Lemma~\ref{lem:mp_brownian_bridge} we have $\tilde X_A \Indep \tilde X_B \given \tilde X_C$. Since $X_V^\Delta$ and $\tilde X_V$ are independent by Lemma~\ref{lem:euler_decomposition}, this combines into the conditional independence $(X_A^\Delta, \tilde X_A) \Indep (X_B^\Delta, \tilde X_B) \given (X_C^\Delta, \tilde X_C)$. By Lemma~\ref{lem:euler_decomposition}, $\sigma(X_S^n) = \sigma(X_S^\Delta) \vee \sigma(\tilde X_S)$ for each $S \in \{A, B, C\}$, so that $X_A^n$ and $X_B^n$ are measurable functions of $(X_A^\Delta, \tilde X_A)$ and $(X_B^\Delta, \tilde X_B)$, while conditioning on $X_C^n$ coincides with conditioning on $(X_C^\Delta, \tilde X_C)$; hence this gives $X_A^n \Indep X_B^n \given X_C^n$.
\end{proof}

\begin{lemapp}\label{thm:gauss_bound}
    Let $\Dcal$ satisfy Assumption~\ref{ass:additive_noise}. Then for any $\delta > 0$, the process $G(t) := \Sigma X_\gamma(t)$ satisfies
    \begin{equation*}
        \EE\left[\exp(3\delta \sup_{0\leq s\leq T}\|G(s)\|^2)\right] < \infty
    \end{equation*}
    whenever $6\delta T\|\Sigma\|_{\mathrm{op}}^2 < 1$.
\end{lemapp}
\begin{proof}
    Note that $\|G(s)\| \leq \|\Sigma\|_{\mathrm{op}}\|X_\gamma(s)\|$ and $\|X_\gamma(s)\|^2 = \sum_{i=1}^d W_i(s)^2$, where $W_i := X_{\gamma_i}$ are independent standard Brownian motions. Writing $M_i := \sup_{0\leq s\leq T}|W_i(s)|$, we have $\sup_{0\leq s\leq T}\|G(s)\|^2 \leq \|\Sigma\|_{\mathrm{op}}^2 \sum_{i=1}^d M_i^2$, and by independence of the components
    \begin{equation*}
        \EE\left[\exp\left(3\delta \sup_{0\leq s\leq T}\|G(s)\|^2\right)\right] \leq \prod_{i=1}^d \EE\left[\exp\left(3\delta\|\Sigma\|_{\mathrm{op}}^2 M_i^2\right)\right].
    \end{equation*}
    The event $\{M_i \geq a\}$ is contained in $\{\sup_{0\leq s\leq T} W_i(s) \geq a\} \cup \{\sup_{0\leq s\leq T}(-W_i(s)) \geq a\}$, so by a union bound and the symmetry of $W_i$, the reflection principle $\PP(\sup_{0\leq s\leq T} W_i(s) \geq a) = 2\PP(W_i(T) \geq a)$, and the Gaussian tail bound $\PP(W_i(T) \geq a) \leq \tfrac{1}{2}e^{-a^2/2T}$,
    \begin{equation*}
        \PP(M_i \geq a) \leq 2\,\PP\Big(\sup_{0\leq s\leq T} W_i(s) \geq a\Big) = 4\,\PP(W_i(T) \geq a) \leq 2e^{-a^2/2T}.
    \end{equation*}
    With $\lambda := 3\delta\|\Sigma\|_{\mathrm{op}}^2 > 0$,
    \begin{equation*}
        \EE\left[\exp(\lambda M_i^2)\right] = 1 + \int_0^\infty 2\lambda a\, e^{\lambda a^2}\,\PP(M_i \geq a)\diff a \leq 1 + 4\lambda\int_0^\infty a\, e^{-(\frac{1}{2T} - \lambda)a^2}\diff a = 1 + \frac{4T\lambda}{1 - 2T\lambda},
    \end{equation*}
    which is finite whenever $\lambda < 1/(2T)$, i.e.\ $6\delta T\|\Sigma\|_{\mathrm{op}}^2 < 1$. As a finite product of finite factors, $\EE[\exp(3\delta \sup_{0\leq s\leq T}\|G(s)\|^2)]$ is then finite.
\end{proof}

\eulertvconvergence*
\begin{proof}
    Without loss of generality, assume that the linear growth coefficient $K(t)$ of $\mu$ satisfies $1 \leq K := \sup_{0\leq t \leq T}|K(t)| < \infty$. For readability, write $X := X_V$, $X_0 := X_\alpha^0$ and $W := X_\gamma$. Define $M := \|\Sigma^{-1}\|_{\mathrm{op}}$, the process $G(t) := \Sigma W(t)$, and $\theta(s, x) := \Sigma^{-1}\mu(s,x)$ for $x \in \RR^d$; for a path or process $y$ we abbreviate $\mu(s, y) := \mu(s, y(s))$ and $\theta(s, y) := \Sigma^{-1}\mu(s, y(s))$.

    \paragraph{Step 1: Disintegration with respect to initial condition.}
    By Theorem \ref{thm:ito_map}, $X = I^{[V]}(X_0, W)$ where $I^{[V]}$ is the adapted solution function. For any deterministic $x_0\in\RR^d$, the process $X^{x_0} := I^{[V]}(x_0, W)$ is therefore $\Fcal^W$-adapted (and hence so is $\theta(s, X^{x_0})$). We write $\PP_{X^{x_0}}$ for the law of the solution with deterministic initial condition $x_0$. Since $X_0$ is independent of $W$, the law with deterministic initial condition is a version of the conditional law, i.e.\ $\PP(X\in\cdot\given X_0 = x_0) = \PP(I^{[V]}(x_0, W)\in\cdot) = \PP_{X^{x_0}}(\cdot)$, and hence the law of $X$ disintegrates as $\PP(X \in \cdot) = \int \PP_{X^{x_0}}(\cdot)\diff\PP_{X_0}(x_0)$. The same holds for the continuous Euler scheme $X^{x_0,n}$ with initial condition $x_0$.

    \paragraph{Step 2: Density of the centred process $\PP_{Y^{x_0}}$ with respect to $\PP_G$.}
    Fix $x_0\in\RR^d$ and let $Y^{x_0} := X^{x_0} - x_0$, which starts at $Y^{x_0}(0) = 0$; we compute its density with respect to $\PP_G$ (the law $\PP_{X^{x_0}}$ itself is supported on paths started at $x_0$ and is mutually singular with $\PP_G$ for $x_0\neq 0$). The SDE for $X^{x_0}$ gives
    \begin{equation*}
        \|X^{x_0}(t)\| \leq \|x_0\| + \int_0^t\|\mu(s,X^{x_0})\|\diff s + \|G(t)\| \leq (\|x_0\| + Kt + \sup_{0\leq s\leq t}\|G(s)\|) + K\int_0^t\sup_{0\leq u\leq s}\|X^{x_0}(u)\|\diff s,
    \end{equation*}
    where we used the linear growth condition $\|\mu(s,X^{x_0})\| \leq K(1+\sup_{0\leq u\leq s}\|X^{x_0}(u)\|)$. Gr\"onwall's inequality then gives $\sup_{0\leq t \leq T}\|X^{x_0}(t)\| \leq (\|x_0\| + KT + \sup_{0\leq s\leq T}\|G(s)\|)e^{KT}$, so for $\delta>0$ we have
    \begin{align*}
        \EE\left[\exp\left(\delta e^{-2KT} \sup_{0\leq t \leq T}\|X^{x_0}(t)\|^2\right)\right] & \leq \exp(3\delta \|x_0\|^2)\exp(3\delta K^2T^2)\EE\left[\exp(3\delta \sup_{0\leq s\leq T}\|G(s)\|^2)\right]
    \end{align*}
    which is finite for sufficiently small $\delta$ by Lemma \ref{thm:gauss_bound}. Partitioning $[0,T]$ into intervals of width $t_k - t_{k-1} \leq \delta e^{-2KT} / K^2M^2$, we bound
    \begin{align*}
        \EE\left[\exp\left(\frac{1}{2}\int_{t_{k-1}}^{t_k}\|\theta(s, X^{x_0})\|^2\diff s\right)\right] & \leq \EE\left[\exp\left(\frac{1}{2}M^2 K^2(t_k-t_{k-1})(1+\sup_{0\leq t\leq T}\|X^{x_0}(t)\|)^2\right)\right] \\
                                                                                                        & \leq \EE\left[\exp\left(\delta e^{-2KT}(1+\sup_{0\leq t\leq T}\|X^{x_0}(t)\|^2)\right)\right] < \infty
    \end{align*}
    for each $k$. Since $X^{x_0} = I^{[V]}(x_0, W)$ is $\Fcal^W$-adapted, the piecewise Novikov condition (\citealp{karatzas1988brownian}, Section 3.5, Corollary 5.14) gives $\EE[Z^{x_0}]=1$ for the random variable
    \begin{equation*}
        Z^{x_0} :=
        \exp\left(-\int_0^T \theta(s, X^{x_0})\diff W(s) - \frac{1}{2}\int_0^T \|\theta(s, X^{x_0})\|^2\diff s\right).
    \end{equation*}
    Defining the process $\xi^{x_0}(t) := \int_0^t \theta(s, X^{x_0}) \diff s + W(t)$, by \cite{liptser2001statistics} Theorem 7.3 the laws $\PP_{\xi^{x_0}}$ and $\PP_W$ are equivalent with density
    \begin{equation*}
        \frac{\diff\PP_{\xi^{x_0}}}{\diff\PP_W}(\omega) = \exp\left(\int_0^T \theta(s, x_0 + \int_0^s\Sigma\diff\xi^{x_0}(u)(\omega))\diff \xi^{x_0}(s)(\omega) - \frac{1}{2}\int_0^T \|\theta(s, x_0 + \int_0^s\Sigma\diff\xi^{x_0}(u)(\omega))\|^2\diff s\right),
    \end{equation*}
    where $\omega \in \Omega$ denotes a point in the underlying probability space and we used that $\theta(s, X^{x_0}) = \theta(s, x_0 + \int_0^s\Sigma\diff\xi^{x_0}(u))$ is $\Fcal_T^{\xi^{x_0}}$-measurable. Let $F(W)$ be the solution function of $\int_0^t\Sigma\diff W(s)$, with inverse $F^{-1}(G)$ the solution function of $\int_0^t\Sigma^{-1}\diff G(s)$. Since $Y^{x_0} = X^{x_0} - x_0 = F(\xi^{x_0})$ and $G = F(W)$, pushing forward through $F$ gives $\PP_{Y^{x_0}} \sim \PP_G$ with density
    \begin{equation}\label{eqn:density_x}
        \frac{\diff\PP_{Y^{x_0}}}{\diff\PP_{G}}(G(\omega)) = \exp\left(\int_0^T \theta(s, x_0 + G(\omega))\diff W(s)(\omega) - \frac{1}{2}\int_0^T \|\theta(s, x_0 + G(\omega))\|^2\diff s\right).
    \end{equation}

    \paragraph{Step 3: Density of $\PP_{Y^{x_0,n}}$ with respect to $\PP_G$.}
    The continuous Euler scheme $X^{x_0,n}$ satisfies $X^{x_0,n}(t) = x_0 + \int_0^t \mu_n(s, X^{x_0,n}(s))\diff s + \int_0^t\Sigma\diff W(s)$ with $\mu_n(t, x) = \mu(t_k, x(t_k))$ for $t\in [t_k, t_{k+1})$. Since $t_k \leq t$, the linear growth bound $\|\mu_n(t, x)\| = \|\mu(t_k, x(t_k))\| \leq K(1+\|x(t_k)\|) \leq K(1+\sup_{u\leq t}\|x(u)\|)$ holds, and similarly $\mu_n$ inherits the Lipschitz constant of $\mu$. Defining $\theta_n(t, x) = \Sigma^{-1}\mu_n(t, x)$ and the centred process $Y^{x_0,n} := X^{x_0,n} - x_0$, the same Gr\"onwall and Novikov arguments as in Step 2 give $\PP_{Y^{x_0,n}} \sim \PP_G$ with density
    \begin{equation}\label{eqn:density_xn}
        \frac{\diff\PP_{Y^{x_0,n}}}{\diff\PP_{G}}(G(\omega)) = \exp\left(\int_0^T \theta_n(s, x_0 + G(\omega))\diff W(s)(\omega) - \frac{1}{2}\int_0^T \|\theta_n(s, x_0 + G(\omega))\|^2\diff s\right).
    \end{equation}

    \paragraph{Step 4: Comparing the densities.}
    Comparing (\ref{eqn:density_x}) and (\ref{eqn:density_xn}), the log-density ratio is
    \begin{align}
        \begin{split}
            \label{eqn:log_density_ratio}
            \log\frac{\diff\PP_{Y^{x_0,n}}}{\diff\PP_{Y^{x_0}}}(G(\omega)) &= \int_0^T \big(\theta_n(s, x_0+G(\omega)) - \theta(s, x_0+G(\omega))\big)\diff W(s)(\omega) \\
            &\quad - \frac{1}{2}\int_0^T \big(\|\theta_n(s, x_0+G(\omega))\|^2- \|\theta(s, x_0+G(\omega))\|^2\big)\diff s.
        \end{split}
    \end{align}

    Our goal is to show that the densities $\frac{\diff\PP_{Y^{x_0,n}}}{\diff\PP_G}$ converge to $\frac{\diff\PP_{Y^{x_0}}}{\diff\PP_G}$ almost surely, for almost all $x_0$. By Scheff\'e's Theorem, this gives total variation convergence $\PP_{Y^{x_0,n}} \tvto \PP_{Y^{x_0}}$ for almost all $x_0$. Since $X^{x_0,n} = Y^{x_0,n} + x_0$ and $X^{x_0} = Y^{x_0} + x_0$, and total variation distance is invariant under the bi-measurable translation $y\mapsto y + x_0$, this gives $\PP_{X^{x_0,n}} \tvto \PP_{X^{x_0}}$ for almost all $x_0$, after which the disintegration from the beginning of the proof yields the desired total variation convergence $\PP_{X^n} \tvto \PP_X$. Since the densities are exponentials, it suffices to show that the log-density ratio (\ref{eqn:log_density_ratio}) converges to 0 for $(\PP_{X_0}\otimes\PP)$-almost all $(x_0, \omega)$.

    Our strategy is as follows. We bound both terms in (\ref{eqn:log_density_ratio}) in terms of the single quantity $\int_0^T\|\theta_n(s, x_0+G) - \theta(s, x_0+G)\|^2\diff s$ (Step 5). We then show that this quantity converges to 0 in $L^1(\PP_{X_0}\otimes\PP)$ (Step 6), which by the bounds implies that the log-density ratio converges to 0 in $L^1(\PP_{X_0}\otimes\PP)$ as well. $L^1$ convergence implies convergence in probability, which guarantees the existence of a subsequence $n_m$ along which the log-density ratio converges to 0 for $(\PP_{X_0}\otimes\PP)$-almost all $(x_0, \omega)$, as desired (Step 7).

    \paragraph{Step 5: Bounding the log-density ratio.}
    For the stochastic integral in (\ref{eqn:log_density_ratio}), the It\^o isometry gives
    \begin{equation*}
        \EE\left[\left(\int_0^T \big(\theta_n(s, x_0+G)- \theta(s, x_0+G)\big)\diff W(s)\right)^2\right] = \EE\left[\int_0^T\|\theta_n(s, x_0+G) - \theta(s, x_0+G)\|^2\diff s\right]
    \end{equation*}
    and thus $\EE[|\int(\theta_n-\theta)\diff W(s)|] \leq \EE[|\int(\theta_n-\theta)\diff W(s)|^2]^{1/2} = \EE[\int\|\theta_n-\theta\|^2\diff s]^{1/2}$ by Jensen's inequality.
    For the Lebesgue integral in (\ref{eqn:log_density_ratio}), by the reverse triangle inequality $\big|\|\theta_n\|^2 - \|\theta\|^2\big| = (\|\theta_n\| + \|\theta\|)\big|\|\theta_n\| - \|\theta\|\big| \leq (2\|\theta\| + \|\theta_n - \theta\|)\|\theta_n - \theta\|$, so combined with Cauchy--Schwarz on the product $\|\theta\|\|\theta_n - \theta\|$ we have
    \begin{align*}
        \left|\int_0^T \left(\|\theta_n\|^2 - \|\theta\|^2\right)\diff s\right| & \leq 2\left(\int_0^T\|\theta\|^2\diff s\right)^{1/2}\left(\int_0^T\|\theta_n - \theta\|^2\diff s\right)^{1/2} + \int_0^T\|\theta_n - \theta\|^2\diff s,
    \end{align*}
    and thus by Cauchy--Schwarz on the product $2\left(\int_0^T\|\theta\|^2\diff s\right)^{1/2}\left(\int_0^T\|\theta_n - \theta\|^2\diff s\right)^{1/2}$ we get
    \begin{align*}
        \EE\left[\left|\int_0^T \left(\|\theta_n\|^2 - \|\theta\|^2\right)\diff s\right|\right] & \leq L(x_0)\EE\left[\int_0^T\|\theta_n - \theta\|^2\diff s\right]^{1/2} + \EE\left[\int_0^T\|\theta_n - \theta\|^2\diff s\right],
    \end{align*}
    where $\EE[\int_0^T\|\theta\|^2\diff s] \leq TM^2K^2\EE[(1+\|x_0\|+\sup_u\|G(u)\|)^2] := \frac{1}{4}L(x_0)^2 < \infty$ by linear growth, whose finiteness follows from Doob's maximal inequality (\citealp{protter2005stochastic}, Theorem I.20), applied to the nonnegative submartingale $\|G\|$ (a convex function of the continuous martingale $G = \Sigma W$): $\EE[\sup_u\|G(u)\|^2] \leq 4\EE[\|G(T)\|^2] = 4\mathrm{tr}(\Sigma\Sigma^\top T) < \infty$, where $\Sigma\Sigma^\top T$ is the covariance of $G(T)$. Combining these bounds we obtain
    \begin{multline}
        \label{eqn:log_density_bound}
        \EE\left[\left|\log\frac{\diff\PP_{Y^{x_0,n}}}{\diff\PP_{Y^{x_0}}}(G)\right|\right] \leq \left(1+ \tfrac{1}{2}L(x_0)\right)\EE\left[\int_0^T\|\theta_n - \theta\|^2\diff s\right]^{1/2} + \tfrac{1}{2}\EE\left[\int_0^T\|\theta_n - \theta\|^2\diff s\right].
    \end{multline}
    We aim at showing $L^1(\PP_{X_0}\otimes\PP)$-convergence of the log-density ratio from which we then obtain a subsequence which converges for $(\PP_{X_0}\otimes\PP)$-almost all $(x_0, \omega)$.
    Integrating (\ref{eqn:log_density_bound}) over $X_0$ and applying Cauchy--Schwarz over $\PP_{X_0}$ to the first term gives
    \begin{multline*}
        \EE_{X_0}\left[\EE\left[\left|\log\frac{\diff\PP_{Y^{X_0,n}}}{\diff\PP_{Y^{X_0}}}(G)\right|\right]\right] \leq \EE_{X_0}\left[(1+\tfrac{1}{2}L(X_0))^2\right]^{1/2} \EE_{X_0}\left[\EE\left[\int_0^T\|\theta_n - \theta\|^2\diff s\right]\right]^{1/2} \\
        + \tfrac{1}{2}\EE_{X_0}\left[\EE\left[\int_0^T\|\theta_n - \theta\|^2\diff s\right]\right],
    \end{multline*}
    where $\EE_{X_0}[(1+\tfrac{1}{2}L(X_0))^2] < \infty$ since $\EE[\|X_0\|^2] < \infty$ by assumption. Hence $L^1(\PP_{X_0}\otimes\PP)$-convergence of the log-density ratio reduces to showing
    \begin{equation}\label{eqn:convergence}
        \EE_{X_0}\left[\EE\left[\int_0^T\|\theta_n(s, X_0 + G) - \theta(s, X_0 + G)\|^2\diff s\right]\right] \to 0.
    \end{equation}

    \paragraph{Step 6: $L^1(\PP_{X_0}\otimes\PP\otimes\lambda)$-convergence of $\|\theta_n - \theta\|^2$.}
    To obtain (\ref{eqn:convergence}) we show pointwise convergence of $\|\theta_n(s, x_0 + G(\omega)) - \theta(s, x_0 + G(\omega))\|^2 \to 0$ for $\PP_{X_0}\otimes\PP\otimes\lambda$-almost all $(x_0, \omega, s)$, and apply dominated convergence on $(\RR^d \times \Omega \times [0,T], \PP_{X_0}\otimes\PP\otimes\lambda)$ to obtain the result.

    Denote with $t_k^n(s)$ the $t_k$ such that $s\in [t_k, t_{k+1})$. Since $\theta_n(s, x) = \Sigma^{-1}\mu(t_k^n(s), x)$ and $\theta(s, x) = \Sigma^{-1}\mu(s, x)$, we have
    \begin{equation}\label{eqn:theta_bound}
        \|\theta_n(s, x_0 + G) - \theta(s, x_0 + G)\|^2 \leq M^2\|\mu(t_k^n(s), x_0 + G) - \mu(s, x_0 + G)\|^2.
    \end{equation}

    \emph{Pointwise convergence:}
    Fix $(x_0, \omega)$. By the abbreviation above, $\mu(t_k^n(s), x_0+G) = \mu(t_k^n(s), x_0+G(t_k^n(s)))$ and $\mu(s, x_0+G) = \mu(s, x_0+G(s))$. As $t_k^n(s) \to s$ we have $x_0 + G(t_k^n(s))(\omega) \to x_0 + G(s)(\omega)$ by continuity of $G$, and hence
    \begin{equation*}
        \|\mu(t_k^n(s), x_0+G) - \mu(s, x_0+G)\| \to 0
    \end{equation*}
    by joint continuity of $\mu$ (continuous in $t$ and Lipschitz, hence continuous, in the state).

    \emph{Dominating function:} By the triangle inequality and linear growth $\|\mu(t, x_0+G)\| \leq K(1+\|x_0\| + \sup_u\|G(u)\|)$, so
    \begin{equation*}
        \|\theta_n(s, x_0+G) - \theta(s, x_0+G)\|^2 \leq 4M^2 K^2(1+\|x_0\|+\sup_u\|G(u)\|)^2,
    \end{equation*}
    uniformly in $n$ and $s$. Since $\EE[\sup_u\|G(u)\|^2] < \infty$ by Doob's maximal inequality (\citealp{protter2005stochastic}, Theorem I.20) and $\EE[\|X_0\|^2] < \infty$ by assumption, this dominating function is integrable against $\PP_{X_0}\otimes \PP$. Dominated convergence on $(\RR^d \times \Omega \times [0,T], \PP_{X_0}\otimes\PP\otimes\lambda)$ gives (\ref{eqn:convergence}).

    \paragraph{Step 7: Conclusion.}
    As argued in Step 5, (\ref{eqn:convergence}) gives $L^1(\PP_{X_0}\otimes\PP)$-convergence of the log-density ratio, and hence a subsequence $n_m$ along which the log-density ratio converges to 0 for $(\PP_{X_0}\otimes\PP)$-a.a.\ $(x_0,\omega)$. By Fubini's theorem, for $\PP_{X_0}$-a.a.\ $x_0$ the log-density ratio converges $\PP$-a.s.\ to 0, so the densities $\frac{\diff\PP_{Y^{x_0,n_m}}}{\diff\PP_G}$ converge to $\frac{\diff\PP_{Y^{x_0}}}{\diff\PP_G}$ $\PP_G$-a.s., and by Scheff\'e's Theorem $\PP_{Y^{x_0,n_m}} \tvto \PP_{Y^{x_0}}$; since total variation is invariant under the translation $y\mapsto y + x_0$, also $\PP_{X^{x_0,n_m}} \tvto \PP_{X^{x_0}}$.
    By the disintegration established in Step 1 we have
    \begin{equation*}
        d_{TV}(\PP_{X^{n_m}}, \PP_X) \leq \int d_{TV}(\PP_{X^{x_0,n_m}}, \PP_{X^{x_0}})\diff\PP_{X_0}(x_0) \to 0,
    \end{equation*}
    where the convergence to 0 follows by dominated convergence (since $d_{TV} \leq 1$).
\end{proof}

\section{Proofs of Section \ref{sec:causal_inference}}

Rules 2 and 3 of the do-calculus (Theorem \ref{thm:do_calculus}) involve $\sigma$-separation conditions in a graph augmented with intervention nodes $R_v$ (Definition \ref{def:intervention_variables_graph}). One approach to formalise this is to introduce a separate framework of SCMs with free input nodes --- the iSCM framework of \cite{forre2025mathematical} --- and establish the Markov property for that framework. Instead, we model the intervention variables as genuine endogenous variables in an extended system of causal SDEs with an auxiliary exogenous distribution $\nu$. This allows us to apply the $\sigma$-separation Markov property (Theorem \ref{thm:sigma_sep_mp}) directly to the extended system, without having to re-prove it for a different framework.

\begin{definitionapp}[System of causal SDEs with intervention variables]\label{def:intervention_variables}
    Let $\Dcal$ be a system of causal SDEs, $B \subseteq V$, and $\nu$ a probability measure on $\Xcal_{E_B} := \prod_{v \in B}(\Xcal_v \cup \{\star\})$. The \emph{system with intervention variables} $\Dcal_{\Do(R_B \sim \nu)}$ has:
    \begin{itemize}
        \item endogenous variables $V \cup R_B$, where each indicator $X_{R_v}$ has sample space $\Xcal_{R_v} := \Xcal_v \cup \{\star\}$;
        \item exogenous variables $W \cup E_B$, where each $X_{E_v}$ has sample space $\Xcal_{E_v} := \Xcal_v \cup \{\star\}$, with joint distribution $\PP(X_W) \otimes \nu$;
        \item for each $v \in B$, the equation $X_{R_v} = X_{E_v}$ for the intervention variable, and for each $v \in V$ the modified causal SDE
    \end{itemize}
    \begin{equation*}
        X_v(t) = \begin{cases}
            X_{R_v}(t)                                                                              & \text{if } v \in B \text{ and } X_{R_v} \in \Xcal_v, \\
            f_v(t, X_{\alpha(v)}) + \int_0^t g_v(s-, X_v, X_{\beta(v)}) \diff h_v(s, X_{\gamma(v)}) & \text{otherwise}.
        \end{cases}
    \end{equation*}
\end{definitionapp}
For non-product $\nu$ the intervention variables $X_{R_v}$ ($v \in B$) are dependent. For the causal graph and the Markov properties we regard their exogenous parents as a single node $E_B$ with law $\nu$, each $X_{R_v}$ reading its coordinate; this adds the bidirected edges $R_u \leftrightarrow R_v$ ($u \neq v \in B$) to $G(\Dcal_{\Do(R_B \sim \nu)})$ and makes $\Dcal_{\Do(R_B \sim \nu)}$ a system of causal SDEs and thus equipped with a $\sigma$-separation Markov property.\footnote{Treating $E_B$ as a single node makes it a higher-dimensional exogenous process, a relaxation of the one-dimensional exogenous process convention of Definition~\ref{def:sdes}.}

For any partition $B = B_1 \cup B_2$ and $x_{B_2} \in \Xcal_{B_2}$, we will abuse notation and write $\Dcal_{\Do(X_{R_{B_1}} = \star, X_{R_{B_2}} = x_{B_2})}$ for $(\Dcal_{\Do(R_B \sim \nu)})_{\Do(X_{R_{B_1}} = \star, X_{R_{B_2}} = x_{B_2})}$.

Before proving the lemma we record a localisation property: on an event decided at time $0$, essential unique solvability pins down the solution on that event alone.

\begin{lemapp}[Localisation of essential unique solvability]\label{lem:f0_local_solvability}
    Let $\Dcal$ be a system of causal SDEs that is essentially uniquely solvable w.r.t.\ $O \subseteq V$, with solution function $I^{[O]}$ and $S := V \setminus O$. Fix an adapted input $\varphi_S : \Xcal_W \to \Xcal_S$ and a measurable set $G \subseteq \Xcal_W$ such that $\I_G(X_W) = \I_G(X_W^{\wedge 0})$. If an adapted measurable map $\varphi_O:\Xcal_W \to \Xcal_O$ satisfies $\varphi_O(X_W) = \Phi_O(\varphi_O(X_W), \varphi_S(X_W), X_W)$ $\PP$-a.s.\ on $G$, then $\varphi_O(X_W) = I^{[O]}(\varphi_S(X_W), X_W)$ $\PP$-a.s.\ on $G$.
\end{lemapp}
\begin{proof}
    Write $I := I^{[O]}(\varphi_S(X_W), X_W)$ and set $\tilde \varphi_O := \varphi_O\,\I_G + I\,\I_{G^c}$. Since $I$, $\varphi_O$ and $\I_G$ are adapted (the latter since $\I_G(X_W) = \I_G(X_W^{\wedge 0})$), $\tilde \varphi_O$ is adapted as well. By construction we then have
    \begin{align*}
        \Phi_O(\tilde \varphi_O(X_W), \varphi_S(X_W), X_W) = \begin{cases}
                                                                 \Phi_O(\varphi_O(X_W), \varphi_S(X_W), X_W) & \text{ on } G    \\
                                                                 \Phi_O(I, \varphi_S(X_W), X_W)              & \text{ on } G^c.
                                                             \end{cases}
    \end{align*}
    Since both right-hand sides are fixed points on their respective events --- the first by hypothesis, the second because $I$ is a fixed point in general --- we have that $\tilde \varphi_O(X_W) = \Phi_O(\tilde \varphi_O(X_W), \varphi_S(X_W), X_W)$ $\PP$-a.s. The essential uniqueness of $I$ then gives $\tilde \varphi_O = I$ $\PP$-a.s., and thus $\varphi_O = I$ $\PP$-a.s.\ on $G$.
\end{proof}

\begin{restatable}[]{lemapp}{interventionvariablesweaklysimple}\label{lem:intervention_variables_lemma}
    Let $\Dcal$ be a system of causal SDEs that is essentially uniquely solvable w.r.t.\ every $O \subseteq V$, let $B \subseteq V$, and let $\nu$ be a probability measure on $\Xcal_{E_B}$.
    \begin{itemize}
        \item[(i)]
        $\Dcal_{\Do(R_B \sim \nu)}$ is essentially uniquely solvable w.r.t.\ every SCC of $G(\Dcal_{\Do(R_B\sim \nu)})$.\footnote{This can be strengthened to being essentially uniquely solvable w.r.t.\ every $O\subseteq V\cup R_B$, but since this is more complicated and not necessary for our purposes, we do not prove this stronger result.}
        \item[(ii)]
        We have $\nu$-a.s.
        \begin{equation*}
            \PP_{\Dcal_{\Do(R_B \sim \nu)}}(X_V \mid \Do(X_{R_B} = x_{R_B})) = \PP_{\Dcal_{\Do(R_B \sim \nu)}}(X_V \mid X_{R_B} = x_{R_B}).
        \end{equation*}
        \item[(iii)]
        For any partition $B = B_1 \cup B_2$ and $x_{B_2} \in \Xcal_{B_2}$, we have
        \begin{equation*}
            \PP_{\Dcal_{\Do(R_B \sim \nu)}}(X_V \mid \Do(X_{R_{B_1}} = \star, X_{R_{B_2}} = x_{B_2})) = \PP_{\Dcal}(X_V \mid \Do(X_{B_2} = x_{B_2})).
        \end{equation*}
    \end{itemize}
\end{restatable}
\begin{proof}
    \textbf{(i)}
    We verify essential unique solvability w.r.t.\ each SCC $C$ of $G(\Dcal)$. The singletons are immediate: for $O = \{R_v\}$ the mechanism is $X_{R_v} = X_{E_v}$, so $\tilde I^{[R_v]} := x_{E_v}$ is the unique adapted solution function. It remains to treat an SCC $C$ of $G(\Dcal)$.

    So fix an SCC $C\subseteq V$ of $G(\Dcal)$, and let $\Phi$ denote the pathwise mechanism of the augmented system $\Dcal_{\Do(R_B\sim \nu)}$. The complement of $C$ in $V \cup R_B$ splits into the endogenous inputs $S := V \setminus C$ and the regime indicators $R_B$; accordingly fix adapted measurable maps $\varphi_S : \Xcal_W \times \Xcal_{E_B} \to \Xcal_S$ and $\varphi_{R_B} : \Xcal_W \times \Xcal_{E_B} \to \Xcal_{R_B}$. We must exhibit a measurable adapted map $\tilde I^{[C]} : \Xcal_S \times \Xcal_{R_B} \times \Xcal_W \times \Xcal_{E_B} \to \Xcal_C$ that is the following fixed point:
    \begin{equation*}
        \tilde I^{[C]}\bigl(\varphi_S, \varphi_{R_B}, X_W, X_{E_B}\bigr) = \Phi_C\bigl(\tilde I^{[C]}(\varphi_S, \varphi_{R_B}, X_W, X_{E_B}),\, \varphi_S,\, \varphi_{R_B},\, X_W, X_{E_B}\bigr)
    \end{equation*}
    $(\PP(X_W)\otimes\nu)$-a.s.\ (evaluating $\varphi_S, \varphi_{R_B}$ both at $(X_W, X_{E_B})$), and the a.s.-unique such fixed point among adapted processes: every adapted $\varphi_C$ with $\varphi_C = \Phi_C(\varphi_C, \varphi_S, \varphi_{R_B}, \cdot)$ $(\PP(X_W)\otimes\nu)$-a.s.\ equals $\tilde I^{[C]}(\varphi_S, \varphi_{R_B}, \cdot)$ $(\PP(X_W)\otimes\nu)$-a.s.

    A regime configuration $x_{R_B} \in \Xcal_{R_B}$ specifies for each $w \in B$ whether $X_{R_w}$ is idle ($x_{R_w} = \star$) or set to a value $x_{R_w} \in \Xcal_w$. Only the coordinates $w \in C \cap B$ act on the $C$-block, so for each partition $\pi = B_1^\pi \cup B_2^\pi$ of $C \cap B$, let
    \begin{equation*}
        F_\pi := \bigl\{x_{R_B} = (x_{R_{B_1^\pi}}, x_{R_{B_2^\pi}}, x_{R_{B\setminus C}}) \in \Xcal_{R_B} : x_{R_{B_1^\pi}} = \star \text{ and } x_{R_{B_2^\pi}} \neq \star\bigr\}
    \end{equation*}
    be the cell of regime configurations on which exactly the variables in $B_2^\pi$ are intervened on; the $2^{|C \cap B|}$ Borel sets $\{F_\pi\}$ partition $\Xcal_{R_B}$. In particular, for $x_{R_B} \in F_\pi$, all variables in $C\cap B_2^{\pi}$ are intervened on, and all variables in $C^\pi := C\setminus B_2^\pi$ are governed by the original mechanism $\Phi_{C^\pi}$ of $\Dcal$. Let $I^{[C^\pi]}$ be a solution function of $\Dcal$ w.r.t.\ $C^\pi$, and define $\tilde I^{[C]} : \Xcal_S \times \Xcal_{R_B} \times \Xcal_W \times \Xcal_{E_B} \to \Xcal_C$ on each cell $F_\pi$ by matching this mechanism,
    \begin{equation*}
        \tilde I^{[C]}_v(x_S, x_{R_B}, x_W)\big|_{F_\pi} := \begin{cases}
            x_{R_v}                                  & v \in C \cap B_2^\pi, \\
            I^{[C^\pi]}_v(x_S, x_{R_{B_2^\pi}}, x_W) & v \in C^\pi,
        \end{cases}
    \end{equation*}
    constant in $x_{E_B}$, and define the adapted and measurable map $\tilde I^{[C]} := \sum_\pi \I_{F_\pi}(x_{R_B})\, \tilde I^{[C]}\big|_{F_\pi}$.

    Composing $\tilde I^{[C]}$ with $\varphi_{R_B}$ shifts the choice of which variables in $R_B$ are intervened on to the realised exogenous values $(x_W, x_{E_B})$. To analyse the fixed-point and uniqueness properties of $\tilde I^{[C]}$, define
    \begin{equation*}
        G_\pi := \bigl\{(x_W, x_{E_B}) : \varphi_{R_B}(x_W, x_{E_B}) \in F_\pi\bigr\},
    \end{equation*}
    such that we can express
    \begin{equation*}
        \tilde I^{[C]}_v(\varphi_S, \varphi_{R_B}, x_W)\big|_{G_\pi} := \begin{cases}
            \varphi_{R_v}                                        & v \in C \cap B_2^\pi, \\
            I^{[C^\pi]}_v(\varphi_S, \varphi_{R_{B_2^\pi}}, x_W) & v \in C^\pi.
        \end{cases}
    \end{equation*}
    We analyse the fixed-point and uniqueness properties on a single cell $G_\pi$. For each $x_{E_B}$ write $G_\pi(x_{E_B}) := \{x_W : (x_W, x_{E_B})\in G_\pi\} \subseteq \Xcal_W$. For each $v\in C^{\pi}$, by the fixed-point property of $I_v^{[C^\pi]}$ the process $I_v^{[C^\pi]}(\varphi_{S\cup R_{B_2^\pi}}(\cdot, x_{E_B}), \cdot)$ is a fixed point of $\Phi_v(\cdot, \varphi_{S\cup R_{B_2^\pi}}(\cdot, x_{E_B}), \cdot)$ on $G_\pi(x_{E_B})$ outside of some $\PP(X_W)$-null set, and for every other solution $\varphi_v : \Xcal_W\times \Xcal_{E_B} \to \Xcal_v$ we have by Lemma~\ref{lem:f0_local_solvability} that $\varphi_v = I_v^{[C^\pi]}$ $\PP(X_W)$-a.s.\ on $G_\pi(x_{E_B})$. The set of $(x_W, x_{E_B}) \in G_\pi$ where these properties fail is Borel with each section $G_\pi(x_{E_B})$ being $\PP(X_W)$-null, hence is $(\PP(X_W)\otimes\nu)$-null by Fubini. For each $v\in C\cap B_2^\pi$, the mechanism $\Phi_v$ in $\Dcal_{\Do(R_B\sim \nu)}$ does not depend on $X_C$: it sets $X_v = X_{R_v}$ (Definition~\ref{def:intervention_variables}), which on $G_\pi$ equals $\varphi_{R_v}\in\Xcal_v$. The $v$-component of the fixed-point equation is thus $X_v = \varphi_{R_v}$, whose unique adapted solution is $\tilde I^{[C]}_v = \varphi_{R_v}$ (on the whole of $G_\pi$).

    Since $\{G_\pi\}_{\pi}$ partitions $\Xcal_W \times \Xcal_{E_B}$, $\tilde I^{[C]}(\varphi_S(\cdot), \varphi_{R_B}(\cdot), \cdot)$ is $(\PP(X_W)\otimes\nu)$-a.s.\ a fixed point of $\Phi_C$, and any adapted fixed point $\varphi_C$ equals it $(\PP(X_W)\otimes\nu)$-a.s. Since $(\varphi_S, \varphi_{R_B})$ was an arbitrary adapted input, $\Dcal_{\Do(R_B \sim \nu)}$ is essentially uniquely solvable w.r.t.\ $C$.

    \medskip
    \textbf{(ii)}
    For a regime $x_{R_B} \in \Xcal_{R_B}$, let $B_2 := \{v \in B : x_{R_v} \neq \star\}$ be its intervened coordinates and $x_{B_2} := (x_{R_v})_{v \in B_2} \in \Xcal_{B_2}$ their values. By $\Dcal$'s essential unique solvability and Theorem~\ref{thm:obs_int_dist} (intervening on $B_2 \subseteq V$), the system $\Dcal_{\Do(X_{B_2} = x_{B_2})}$ is essentially uniquely solvable w.r.t.\ $V$; write $h(x_W, x_{R_B}) \in \Xcal_V$ for its $V$-solution, giving a measurable map $h : \Xcal_W \times \Xcal_{R_B} \to \Xcal_V$.
    By part (i) and Lemma~\ref{lem:scc_composition}, $\Dcal_{\Do(R_B \sim \nu)}$ is essentially uniquely solvable w.r.t.\ $V \cup R_B$, so it admits the solution function $I^{[V \cup R_B]} : \Xcal_W \times \Xcal_{E_B} \to \Xcal_V \times \Xcal_{R_B}$. By its essential unique solvability we have for each $x_{E_B}$ that $I^{[V\cup R_B]}_V(x_W, x_{E_B}) = h(x_W, x_{E_B})$ holds $\PP(X_W)$-a.s.

    For each $v\in V$ the causal SDE for $v$ in $\Dcal_{\Do(X_{R_B}=x_{R_B})}$ coincides with the causal SDEs for $v$ in $\Dcal_{\Do(X_{B_2} = x_{B_2})}$. Hence, for every $x_{R_B}$,
    \begin{equation}\label{eqn:regime_intervention}
        \PP_{\Dcal_{\Do(R_B \sim \nu)}}(X_V \given \Do(X_{R_B} = x_{R_B})) = \PP\bigl(h(X_W, x_{R_B})\bigr).
    \end{equation}
    As $X_V = h(X_W, X_{R_B})$ with $X_{R_B} = X_{E_B}$ a.s.\ and $X_W \Indep X_{E_B}$ under $\PP(X_W) \otimes \nu$, we have for $\nu$-a.a.\ $x_{R_B}$
    \begin{equation*}
        \PP_{\Dcal_{\Do(R_B \sim \nu)}}(X_V \given X_{R_B} = x_{R_B}) = \PP\bigl(h(X_W, x_{R_B}) \given X_{E_B} = x_{R_B}\bigr) = \PP\bigl(h(X_W, x_{R_B})\bigr).
    \end{equation*}
    The interventional and conditional distributions therefore agree $\nu$-a.s.

    \medskip
    \textbf{(iii)}
    Taking $x_{R_B} = (\star, x_{B_2})$ in \eqref{eqn:regime_intervention} gives
    \begin{equation*}
        \PP_{\Dcal_{\Do(R_B \sim \nu)}}(X_V \given \Do(X_{R_{B_1}} = \star, X_{R_{B_2}} = x_{B_2})) = \PP\bigl(h(X_W, (\star, x_{B_2}))\bigr).
    \end{equation*}
    By definition $h(\cdot, (\star, x_{B_2}))$ is the $V$-solution of $\Dcal_{\Do(X_{B_2} = x_{B_2})}$, so the right-hand side equals $\PP_\Dcal(X_V \given \Do(X_{B_2} = x_{B_2}))$.
\end{proof}

\docalculus*
\begin{proof}
    The proof follows \cite[Theorem 5.1.2]{forre2025mathematical}.

    \emph{Rule 1.} By Theorem~\ref{thm:sigma_sep_mp}, the $\sigma$-separation $A \Perp^\sigma_{G(\Dcal)} B \given C$ gives $X_A \Indep X_B \given X_C$ under $\PP_\Dcal$. By \cite{kallenberg2021foundations}, Theorem 8.9, this gives $\PP_\Dcal(X_A \given X_B, X_C) = \PP_\Dcal(X_A \given X_C)$ $\PP_\Dcal(X_B, X_C)$-a.s.

    \emph{Rule 2.}
    Define $\nu := \tfrac{1}{2}\delta_{(\star, ..., \star)} + \tfrac{1}{2}\PP_\Dcal(X_B)$ on $\Xcal_{R_B}$.
    Write $\Dcal_\nu := \Dcal_{\Do(R_B\sim\nu)}$ and let $\tilde G(\Dcal)_{\Do(R_B)}$ be the graph $G(\Dcal)_{\Do(R_B)}$ appended with bidirected edges $R_u \leftrightarrow R_v$ for all $u \neq v\in B$. Then we also have the $\sigma$-separation $A \Perp^\sigma_{\tilde G(\Dcal)_{\Do(R_B)}} R_B \given B \cup C$, since compared to $G(\Dcal)_{\Do(R_B)}$, adding bidirected edges between the intervention variables does not open up a walk between $A$ and $R_B$ (conditioned on $B\cup C$). Since $G(\Dcal_\nu)$ is a subgraph of $\tilde G(\Dcal)_{\Do(R_B)}$ we also  have $A \Perp^\sigma_{G(\Dcal_\nu)} R_B \given B \cup C$. By Lemma~\ref{lem:intervention_variables_lemma}(i) $\Dcal_\nu$ is essentially uniquely solvable with respect to every strongly connected component of $G(\Dcal_\nu)$, so the $\sigma$-separation Markov property holds and we have under $\PP_{\Dcal_\nu}$ the conditional independence $X_A \Indep X_{R_B} \given X_B, X_C$, which gives the factorisation
    \begin{equation*}
        \PP_{\Dcal_\nu}(X_A, X_B, X_C, X_{R_B}) = \QQ(X_A \given X_B, X_C) \otimes \PP_{\Dcal_\nu}(X_B, X_C, X_{R_B}),
    \end{equation*}
    for some Markov kernel $\QQ : \Xcal_B \times \Xcal_C \to \mathcal{P}(\Xcal_A)$. By conditioning the factorisation above on $X_{R_B}$ and applying Lemma~\ref{lem:intervention_variables_lemma}(ii) we get
    \begin{equation}\label{eq:factorisation}
        \PP_{\Dcal_\nu}(X_A, X_B, X_C \given \Do(X_{R_B})) = \QQ(X_A \given X_B, X_C) \otimes \PP_{\Dcal_\nu}(X_B, X_C \given \Do(X_{R_B}))
    \end{equation}
    $\nu$-a.s. Under $\Do(X_{R_B} = \star)$, Lemma~\ref{lem:intervention_variables_lemma}(iii) and \eqref{eq:factorisation} give
    \begin{equation*}
        \PP_\Dcal(X_A, X_B, X_C) = \QQ(X_A \given X_B, X_C) \otimes \PP_\Dcal(X_B, X_C),
    \end{equation*}
    and hence $\QQ(X_A \given X_B, X_C) = \PP_\Dcal(X_A \given X_B, X_C)$ holds $\PP_\Dcal(X_B, X_C)$-a.s.
    Under $\Do(X_{R_B} = x_B)$ for $x_B \in \Xcal_B$, Lemma~\ref{lem:intervention_variables_lemma}(iii) and \eqref{eq:factorisation} give, for $\PP_\Dcal(X_B)$-a.a.\ $x_B$ that
    \begin{equation}\label{eqn:rule_2_eqn}
        \PP_\Dcal(X_A, X_C \given \Do(X_B = x_B)) = \QQ(X_A \given X_B = x_B, X_C) \otimes \PP_\Dcal(X_C \given \Do(X_B = x_B))
    \end{equation}
    and hence $\QQ(X_A \given X_B, X_C) = \PP_\Dcal(X_A \given \Do(X_B=x_B), X_C)$ holds for $\PP_\Dcal(X_B)$-almost all $x_B$, $\PP_\Dcal(X_C \given \Do(X_B = x_B))$-a.s. If for $\PP_\Dcal(X_B)$-a.a.\ $x_B$ we also have $\PP_\Dcal(X_C \mid X_B = x_B) \ll \PP_\Dcal(X_C \mid \Do(X_B = x_B))$, then $\QQ(X_A \given X_B, X_C) = \PP_\Dcal(X_A \given \Do(X_B), X_C)$ holds $\PP_\Dcal(X_B, X_C)$-a.s.\ as well, in which case $\PP_\Dcal(X_A \mid X_B, X_C) = \PP_\Dcal(X_A \mid \Do(X_B), X_C)$ holds $\PP_\Dcal(X_B, X_C)$-a.s.

    \emph{Rule 3.}
    Let $x_B \in \Xcal_B$ and define $\nu := \bigotimes_{v\in B}\left(\tfrac{1}{2}\delta_{\star} + \tfrac{1}{2}\delta_{x_v}\right)$ on $\Xcal_{R_B}$. Since $\nu$ is a product measure we have that $G(\Dcal_\nu)$ is a subgraph of $G(\Dcal)_{\Do(R_B)}$, so $A \Perp^\sigma_{G(\Dcal_\nu)} R_B \given C$. By the $\sigma$-separation Markov property for $\Dcal_\nu$ we have the conditional independence $X_A \Indep X_{R_B} \given X_C$, giving the factorisation
    \begin{equation*}
        \PP_{\Dcal_\nu}(X_A, X_C, X_{R_B}) = \QQ(X_A \given X_C) \otimes \PP_{\Dcal_\nu}(X_C, X_{R_B}),
    \end{equation*}
    for some Markov kernel $\QQ : \Xcal_C \to \mathcal{P}(\Xcal_A)$, and similarly to Rule 2 this yields
    \begin{equation*}
        \PP_{\Dcal_\nu}(X_A, X_C \given \Do(X_{R_B})) = \QQ(X_A \given X_C) \otimes \PP_{\Dcal_\nu}(X_C \given \Do(X_{R_B})) \quad \nu\text{-a.s.}
    \end{equation*}
    Under $\Do(X_{R_B} = \star)$, Lemma~\ref{lem:intervention_variables_lemma}(iii) gives $\QQ(X_A \given X_C) = \PP_\Dcal(X_A \given X_C)$ $\PP_\Dcal(X_C)$-a.s. Since $\nu$ assigns positive mass to $x_B$, we have by Lemma~\ref{lem:intervention_variables_lemma}(iii) under $\Do(X_{R_B} = x_B)$ that
    \begin{equation*}
        \QQ(X_A \given X_C) = \PP_\Dcal(X_A \given \Do(X_B = x_B), X_C) \quad \PP_\Dcal(X_C \given \Do(X_B = x_B))\text{-a.s.}
    \end{equation*}
    If $\PP_\Dcal(X_C) \ll \PP_\Dcal(X_C \given \Do(X_B = x_B))$ then this last equality holds $\PP_\Dcal(X_C)$-a.s., and combining with the previous equality yields $\PP_\Dcal(X_A \given \Do(X_B = x_B), X_C) = \PP_\Dcal(X_A \given X_C)$ $\PP_\Dcal(X_C)$-a.s.
\end{proof}

\causationimpliespath*
\begin{proof}
    If there is no directed path from $u$ to $v$ in $G(\Dcal)$, then there is no directed path from $u$ to $v$ in $G(\Dcal)_{\Do(R_u)}$ (since $R_u$ has only the edge $R_u \to u$), and hence $v \Perp^\sigma_{G(\Dcal)_{\Do(R_u)}} R_u$. Applying Rule 3 of Theorem~\ref{thm:do_calculus} with $A = \{v\}$, $B = \{u\}$, $C = \emptyset$ (the absolute continuity condition is vacuous since $C = \emptyset$) gives $\PP(X_v \given \Do(X_u)) = \PP(X_v)$.
\end{proof}

\section{Proofs of Section \ref{sec:time-evaluations}}
Define the \emph{restriction map} $\mathrm{res}^\Ical : D([0,T], \RR^n) \to D(\Ical, \RR^n)$ as $\mathrm{res}^\Ical(x) := x|_\Ical$, and extend as $\mathrm{res}^\Ecal = (\mathrm{res}^{\Ical_1}, ..., \mathrm{res}^{\Ical_m})$.
\begin{lemapp}\label{lem:timesplit_equivalent}
    Let $\Dcal = \angs{V, W, X_W, f, g, h}$ be a system of causal SDEs, $\Ecal$ a finite partition of $[0,T]$ into intervals and $\Dcal^\Ecal$ the time-split SDE. Then:
    \begin{enumerate}[label=(\roman*)]
        \item If $X_V$ is a solution of $\Dcal$, then $X_V^\Ecal := \mathrm{res}^\Ecal(X_V)$ is a solution of $\Dcal^\Ecal$.
        \item Conversely, if $X_V^\Ecal$ is a solution of $\Dcal^\Ecal$, then $X_V := \mathrm{cat}(X_V^\Ecal)$ is a solution of $\Dcal$.
    \end{enumerate}
\end{lemapp}
\begin{proof}
    \emph{(i)} Let $X_V$ solve $\Dcal$ and set $X_V^\Ecal := \mathrm{res}^\Ecal(X_V)$; let $X_W^\Ecal$ denote the corresponding tuple of exogenous increment processes from Definition \ref{def:sdes_timesplit}. Fix $v\in V$, $\Ical_k\in\Ecal$, $s \in \Ical_k$. Evaluating the SDE at $s$ and at $b_{k-1, k}$ gives
    \begin{equation*}
        X_v(s) = X_v(b_{k-1, k}) + \big(f_v(s, X_{\alpha(v)}) - f_v(b_{k-1, k}, X_{\alpha(v)})\big) + \int_{b_{k-1, k}}^s g_v(u-, X_v, X_{\beta(v)}) \diff h_v(u, X_{\gamma(v)}).
    \end{equation*}
    We have $X_v(b_{k-1, k}) = X_v^{\Ical_{k-1}}(b_{k-1, k})$, which for $k=1$ equals $X_v(0) = f_v(0, X_{\alpha(v)}^{\Ical_1})$ by Definition~\ref{def:sdes_timesplit}. Under the reconstructions $X_{\alpha(v)} = (\mathrm{cat}(X_{\alpha(v)\cap V}^\Ecal), \mathrm{rec}(X_{\alpha(v)\cap W}^\Ecal))$ and similarly for $\{v\}\cup\beta(v)$ and $\gamma(v)$, writing out the definitions of $f_v^{\Ical_k}, g_v^{\Ical_k}, h_v^{\Ical_k}$ yields the time-split SDE for $X_v^{\Ical_k}$.

    \emph{(ii)} Let $X_V^\Ecal$ solve $\Dcal^\Ecal$ and set $X_V := \mathrm{cat}(X_V^\Ecal)$, $X_W := \mathrm{rec}(X_W^\Ecal)$. Under these reconstructions, $f_v^{\Ical_k}, g_v^{\Ical_k}, h_v^{\Ical_k}$ evaluate via the original $f_v, g_v, h_v$ as in Definition \ref{def:sdes_timesplit}. We show by induction on $k$ that $X_v(s) = f_v(s, X_{\alpha(v)}) + \int_0^s g_v(u-, X_v, X_{\beta(v)}) \diff h_v(u, X_{\gamma(v)})$ for all $s \in \Ical_k$.

    \emph{Base ($k=1$).} The time-split SDE on $\Ical_1$ straightforwardly reduces to $X_v^{\Ical_1}(s) = f_v(s, X_{\alpha(v)}) + \int_0^s g_v(u-, X_v, X_{\beta(v)}) \diff h_v(u, X_{\gamma(v)})$.

    \emph{Inductive step ($k \geq 2$).} The time-split SDE on $\Ical_k$ reads
    \begin{equation*}
        X_v^{\Ical_k}(s) = X_v^{\Ical_{k-1}}(b_{k-1, k}) + \big(f_v(s, X_{\alpha(v)}) - f_v(b_{k-1, k}, X_{\alpha(v)})\big) + \int_{b_{k-1, k}}^s g_v(u-, X_v, X_{\beta(v)}) \diff h_v(u, X_{\gamma(v)}).
    \end{equation*}
    We have $X_v^{\Ical_{k-1}}(b_{k-1, k}) = X_v(b_{k-1, k})$, which by the induction hypothesis equals $f_v(b_{k-1, k}, X_{\alpha(v)}) + \int_0^{b_{k-1, k}} g_v(u-, X_v, X_{\beta(v)}) \diff h_v(u, X_{\gamma(v)})$. Substituting and combining the integrals gives $X_v(s) = f_v(s, X_{\alpha(v)}) + \int_0^s g_v(u-, X_v, X_{\beta(v)}) \diff h_v(u, X_{\gamma(v)})$.
\end{proof}

\begin{lemapp}[Solvability of the time-split system]\label{lem:timesplit_solvable}
    Let $\Dcal$ be a system of causal SDEs satisfying Assumption~\ref{ass:uniquely_solvable} whose exogenous processes have independent increments, and let $\Ecal$ be a finite partition of $[0,T]$ into intervals. Then the time-split system $\Dcal^\Ecal$ is essentially uniquely solvable w.r.t.\ every $O \subseteq V \times \Ecal$.
\end{lemapp}
\begin{proof}
    By adaptedness, $G(\Dcal^\Ecal)$ has no edges from later to earlier intervals, so every SCC of $G(\Dcal^\Ecal)$ is of the form $S \times \{\Ical_k\}$ for some SCC $S$ of $G(\Dcal)$. On each such SCC, the functions $f_v^{\Ical_k}, g_v^{\Ical_k}, h_v^{\Ical_k}$ are compositions of $f_v, g_v, h_v$ with $\mathrm{cat}$ and $\mathrm{rec}$, and the linear-growth and Lipschitz bounds transfer pointwise since $\|\mathrm{cat}(x^\Ecal)\|_\infty \leq \max_k \|x^{\Ical_k}\|_\infty$ and $\|\mathrm{rec}(y^\Ecal)\|_\infty \leq \sum_k \|y^{\Ical_k}\|_\infty \leq |\Ecal| \max_k \|y^{\Ical_k}\|_\infty$. Hence $\Dcal^\Ecal$ satisfies Assumption~\ref{ass:uniquely_solvable}, and by Theorem~\ref{thm:ess_solvable_under_ass} it is essentially uniquely solvable w.r.t.\ every $O \subseteq V \times \Ecal$.
\end{proof}

For any subset $O\subseteq (V\cup W)\times \Ecal$ and tuple of processes $\bar x_O = (\bar x_u^{\Ical})_{(u,\Ical)\in O} \in D([0,T], \RR)^{|O|}$, let the restriction map $\mathrm{res}(\bar x_O) = (\mathrm{res}^{\Ical}(\bar x_u^{\Ical}))_{(u,\Ical)\in O}$ be applied element-wise. As a main tool for embedding each variable $x_{v}^{\Ical_k}\in D(\Ical_k, \RR)$ of $\Dcal^{\Ecal}$ in $D([0,T], \RR)$, consider the following map:
\begin{equation*}
    \mathrm{ext}(x_v^{\Ical_k})(t)  := \begin{cases}
        0                                               & \text{if } t< \inf\Ical_k                              \\
        \lim_{s\downarrow\inf\Ical_k}x_{v}^{\Ical_k}(s) & \text{if } t = \inf\Ical_k \text{ and } t\notin\Ical_k \\
        x_v^{\Ical_k}(t)                                & \text{if } t\in \Ical_k                                \\
        x_v^{\Ical_k}(b_{k, k+1})                       & \text{if } t\geq \sup\Ical_k,
    \end{cases}
\end{equation*}
extended componentwise. Note that $\mathrm{res}\circ\mathrm{ext} = \mathrm{id}$.

\begin{definitionapp}[Embedded time-split system]\label{def:embedded}
    Given $\Dcal$ such that each exogenous process $X_w$ has independent increments, let $\Dcal^\Ecal = \angs{V\times\Ecal, W\times\Ecal, X_{W\times\Ecal}, f_V^\Ecal, g_V^\Ecal, h_V^\Ecal}$ be the time-split system of causal SDEs (Definition~\ref{def:sdes_timesplit}) for some finite partition $\Ecal$ of $[0,T]$ into intervals, and let $\Phi_v^{\Ical_k}$ denote its pathwise mechanism. We define the \emph{embedded time-split system} $\bar \Dcal^\Ecal := \angs{V\times\Ecal, W\times\Ecal, \bar\Xcal_{V\times\Ecal}, \bar\Xcal_{W\times\Ecal}, \bar\Phi_V^\Ecal, \bar\PP}$ on $[0,T]$ as follows:
    \begin{itemize}
        \item Sample spaces $\bar \Xcal_{V\times\Ecal} := D([0,T], \RR)^{|V\times\Ecal|}$ and $\bar\Xcal_{W\times\Ecal} := D([0,T], \RR)^{|W\times\Ecal|}$.
        \item For each $(v, \Ical_k) \in V\times\Ecal$, the embedded causal SDE $\bar X_v^{\Ical_k}(t) = \bar\Phi_v^{\Ical_k}(\bar x_V^\Ecal, \bar x_W^\Ecal)(t)$ with
        \begin{equation*}
            \bar\Phi_v^{\Ical_k}(\bar x_V^\Ecal, \bar x_W^\Ecal) := \mathrm{ext}\bigl(\Phi_v^{\Ical_k}(\mathrm{res}(\bar x_V^\Ecal), \mathrm{res}(\bar x_W^\Ecal))\bigr).
        \end{equation*}
        \item As exogenous distribution $\PP(\bar X_W^\Ecal) := \bigotimes_{(w,\Ical)\in W\times \Ecal}\PP(\bar X_w^\Ical)$ the distribution of the random variable
        \begin{equation*}
            \bar X_w^{\Ical_k} := \mathrm{ext}(X_w^{\Ical_k})
        \end{equation*}
        where $X_W^\Ecal$ is any random variable with distribution $\PP(X_W^\Ecal)$.
    \end{itemize}
\end{definitionapp}

\begin{lemapp}\label{lem:timesplit_embedding}
    Let $\Dcal^\Ecal$ be the time-split system of causal SDEs from Definition \ref{def:sdes_timesplit} with exogenous processes $X_w$ ($w \in W$) having independent increments, then:
    \begin{enumerate}[label=(\roman*)]
        \item If $\Dcal^\Ecal$ is essentially uniquely solvable w.r.t.\ $O \subseteq V\times\Ecal$, then so is $\bar\Dcal^\Ecal$.
        \item $G(\Dcal^\Ecal) = G(\bar\Dcal^\Ecal)$.
        \item For every $S \subseteq V\times\Ecal$, with $O := (V\times\Ecal)\setminus S$, and every $x_S \in \Xcal_S$,
        \begin{equation*}
            \PP_{\Dcal^\Ecal}\bigl(X_O \given \Do(X_S = x_S)\bigr) = \PP_{\bar\Dcal^\Ecal}\bigl(\mathrm{res}(\bar X_O) \given \Do(\bar X_S = \mathrm{ext}(x_S))\bigr).
        \end{equation*}
    \end{enumerate}
\end{lemapp}
\begin{proof}
    \emph{(i)} Fix $O \subseteq V \times \Ecal$ and write $S := (V \times \Ecal) \setminus O$. By essential unique solvability of $\Dcal^\Ecal$ w.r.t.\ $O$, there is an adapted measurable solution function $I^{[O]} : \Xcal_S \times \Xcal_{W\times\Ecal} \to \Xcal_O$, that satisfies for every admissible adapted measurable $\varphi_S : \Xcal_{W\times\Ecal} \to \Xcal_S$ the fixed-point property
    \begin{equation}\label{eqn:embedding_baseeqn}
        I^{[O]}(\varphi_S(X_W^\Ecal), X_W^\Ecal) = \Phi_O\bigl(I^{[O]}(\varphi_S(X_W^\Ecal), X_W^\Ecal), \varphi_S(X_W^\Ecal), X_W^\Ecal\bigr)
    \end{equation}
    $\PP$-a.s., and is essentially unique.
    Define $\bar I^{[O]} : \bar\Xcal_S \times \bar\Xcal_{W\times\Ecal} \to \bar\Xcal_O$ as the componentwise extension of the time-split solution to $[0,T]$,
    \begin{equation*}
        \bar I^{[O]}(\bar x_S, \bar x_W^\Ecal) := \mathrm{ext}\bigl(I^{[O]}(\mathrm{res}(\bar x_S), \mathrm{res}(\bar x_W^\Ecal))\bigr),
    \end{equation*}
    which is adapted and measurable.
    Fix an admissible, adapted measurable $\bar\varphi_S : \bar\Xcal_{W\times\Ecal} \to \bar\Xcal_S$ and let $\varphi_S := \mathrm{res} \circ \bar\varphi_S \circ \mathrm{ext} : \Xcal_{W\times\Ecal} \to \Xcal_S$, which is admissible, adapted and measurable as well.
    Write $X_W^\Ecal := \mathrm{res}(\bar X_W^\Ecal)$, so that $\mathrm{res}(\bar\varphi_S(\bar X_W^\Ecal)) = \varphi_S(x_W^\Ecal)$. Using $\bar\Phi_O = \mathrm{ext}\circ\Phi_O\circ\mathrm{res}$, $\mathrm{res}\circ\mathrm{ext} = \mathrm{id}$, and \eqref{eqn:embedding_baseeqn},
    \begin{align*}
        \bar\Phi_O\bigl(\bar I^{[O]}(\bar\varphi_S(\bar X_W^\Ecal), \bar X_W^\Ecal), \bar\varphi_S(\bar X_W^\Ecal), \bar X_W^\Ecal\bigr)
         & = \mathrm{ext}\bigl(\Phi_O(I^{[O]}(\varphi_S(x_W^\Ecal), x_W^\Ecal), \varphi_S(x_W^\Ecal), x_W^\Ecal)\bigr)                      \\
         & = \mathrm{ext}\bigl(I^{[O]}(\varphi_S(x_W^\Ecal), x_W^\Ecal)\bigr) = \bar I^{[O]}(\bar\varphi_S(\bar X_W^\Ecal), \bar X_W^\Ecal)
    \end{align*}
    $\bar\PP$-a.s., so $\bar I^{[O]}(\bar\varphi_S(\bar X_W^\Ecal), \bar X_W^\Ecal)$ is a fixed point of $\bar\Phi_O$.

    Let $\bar\varphi_O : \bar\Xcal_{W\times\Ecal} \to \bar\Xcal_O$ be a measurable adapted function with $\bar\varphi_O(\bar X_W^\Ecal) = \bar\Phi_O(\bar\varphi_O(\bar X_W^\Ecal), \bar\varphi_S(\bar X_W^\Ecal), \bar X_W^\Ecal)$ $\bar\PP$-a.s., and set $\varphi_O := \mathrm{res} \circ \bar\varphi_O \circ \mathrm{ext}$, so that $\mathrm{res}(\bar\varphi_O(\bar X_W^\Ecal)) = \varphi_O(x_W^\Ecal)$ $\bar\PP$-a.s. Using $\bar\Phi_O = \mathrm{ext}\circ\Phi_O\circ\mathrm{res}$ and applying $\mathrm{res}$ (with $\mathrm{res}\circ\mathrm{ext} = \mathrm{id}$) gives $\varphi_O(x_W^\Ecal) = \Phi_O(\varphi_O(x_W^\Ecal), \varphi_S(x_W^\Ecal), x_W^\Ecal)$; since $\bar\PP(\mathrm{res}(\bar X_W^\Ecal)) = \PP(X_W^\Ecal)$ this holds $\PP$-a.s., so by the essential uniqueness of $I^{[O]}$, $\varphi_O(x_W^\Ecal) = I^{[O]}(\varphi_S(x_W^\Ecal), x_W^\Ecal)$ $\PP$-a.s. Substituting back and using \eqref{eqn:embedding_baseeqn},
    \begin{equation*}
        \bar\varphi_O(\bar X_W^\Ecal) = \mathrm{ext}\bigl(\Phi_O(I^{[O]}(\varphi_S(x_W^\Ecal), x_W^\Ecal), \varphi_S(x_W^\Ecal), x_W^\Ecal)\bigr) = \mathrm{ext}\bigl(I^{[O]}(\varphi_S(x_W^\Ecal), x_W^\Ecal)\bigr) = \bar I^{[O]}(\bar\varphi_S(\bar X_W^\Ecal), \bar X_W^\Ecal)
    \end{equation*}
    $\bar\PP$-a.s.

    \emph{(ii)} By Definition~\ref{def:embedded}, $\bar\Phi_v^{\Ical_k} = \mathrm{ext}\circ\Phi_v^{\Ical_k}\circ\mathrm{res}$, and $\mathrm{ext}, \mathrm{res}$ are deterministic with $\mathrm{res}^{\Ical_j}(\bar x_u^{\Ical_j}) = x_u^{\Ical_j}$; hence $\bar\Phi_v^{\Ical_k}$ essentially depends on $\bar x_u^{\Ical_j}$ if and only if $\Phi_v^{\Ical_k}$ essentially depends on $x_u^{\Ical_j}$, so the directed edges into $(v,\Ical_k)$ coincide in $G(\bar\Dcal^\Ecal)$ and $G(\Dcal^\Ecal)$ --- including the initial-value edge $(v,\Ical_{k-1})\to(v,\Ical_k)$, carried by the term $X_v^{\Ical_{k-1}}(b_{k-1,k})$ inside $\Phi_v^{\Ical_k}$ (Definition~\ref{def:sdes_timesplit}). Bidirected edges match via the product factorisation $\bar\PP = \bigotimes_{(w,\Ical)} \PP(\bar X_w^\Ical)$. Hence $G(\bar\Dcal^\Ecal) = G(\Dcal^\Ecal)$.

    \emph{(iii)} Fix a solution function $I^{[O]} : \Xcal_S \times \Xcal_{W\times\Ecal} \to \Xcal_O$ of $\Dcal^\Ecal$ w.r.t.\ $O$, and let $\bar I^{[O]} = \mathrm{ext}\circ I^{[O]}\circ\mathrm{res}$ be the solution function of $\bar\Dcal^\Ecal$ from part (i). Since $\mathrm{res}\circ\mathrm{ext} = \mathrm{id}$,
    \begin{equation}\label{eqn:rho_bar_g}
        \mathrm{res}\bigl(\bar I^{[O]}(\bar x_S, \bar x_W^\Ecal)\bigr) = I^{[O]}\bigl(\mathrm{res}(\bar x_S), \mathrm{res}(\bar x_W^\Ecal)\bigr).
    \end{equation}
    By part (i), $\bar\Dcal^\Ecal$ is essentially uniquely solvable, so its interventional distributions are well-defined as pushforwards under the (intervened) solution function. Applied to $\Dcal^\Ecal$ and to $\bar\Dcal^\Ecal$, this gives
    \begin{align*}
        \PP_{\Dcal^\Ecal}\bigl(X_O \given \Do(X_S = x_S)\bigr)                             & = I^{[O]}(x_S, \cdot)_* \PP(X_W^\Ecal),                             \\
        \PP_{\bar\Dcal^\Ecal}\bigl(\bar X_O \given \Do(\bar X_S = \mathrm{ext}(x_S))\bigr) & = \bar I^{[O]}(\mathrm{ext}(x_S), \cdot)_* \bar\PP(\bar X_W^\Ecal).
    \end{align*}
    The law of $\mathrm{res}(\bar X_O)$ under the embedded interventional distribution is the pushforward of $\bar\PP(\bar X_W^\Ecal)$ under
    \begin{align*}
        \bar x_W^\Ecal \mapsto \mathrm{res}\bigl(\bar I^{[O]}(\mathrm{ext}(x_S), \bar x_W^\Ecal)\bigr) & \overset{\eqref{eqn:rho_bar_g}}{=} I^{[O]}\bigl(\mathrm{res}(\mathrm{ext}(x_S)), \mathrm{res}(\bar x_W^\Ecal)\bigr) = I^{[O]}\bigl(x_S, \mathrm{res}(\bar x_W^\Ecal)\bigr),
    \end{align*}
    where the final equality uses $\mathrm{res}\circ\mathrm{ext} = \mathrm{id}$. By Definition \ref{def:embedded}, $\bar\PP(\mathrm{res}(\bar X_W^\Ecal)) = \PP(X_W^\Ecal)$, so this pushforward equals $I^{[O]}(x_S, \cdot)_* \PP(X_W^\Ecal) = \PP_{\Dcal^\Ecal}\bigl(X_O \given \Do(X_S = x_S)\bigr)$ as required.
\end{proof}

\markovtimesplit*
\begin{proof}
    \emph{(i)} By Lemma~\ref{lem:timesplit_solvable}, $\Dcal^\Ecal$ is essentially uniquely solvable w.r.t.\ every $O \subseteq V \times \Ecal$. By Lemma~\ref{lem:timesplit_embedding}(i), the embedded time-split system $\bar\Dcal^\Ecal$ is also essentially uniquely solvable as a system of causal SDEs on $[0,T]$, so the $\sigma$-separation Markov property applies to it. Suppose $X_A^{\Ical_A} \Perp^\sigma_{G(\Dcal^\Ecal)} X_B^{\Ical_B} \given X_{C_1}^{\Ical_{C_1}}, \ldots, X_{C_r}^{\Ical_{C_r}}$. By Lemma \ref{lem:timesplit_embedding}(ii), $G(\bar\Dcal^\Ecal) = G(\Dcal^\Ecal)$, so the same $\sigma$-separation holds in $G(\bar\Dcal^\Ecal)$ between the corresponding embedded variables: $\bar X_A^{\Ical_A} \Perp^\sigma_{G(\bar\Dcal^\Ecal)} \bar X_B^{\Ical_B} \given \bar X_{C_1}^{\Ical_{C_1}}, \ldots, \bar X_{C_r}^{\Ical_{C_r}}$. The $\sigma$-separation Markov property in $\bar\Dcal^\Ecal$ then gives $\bar X_A^{\Ical_A} \Indep \bar X_B^{\Ical_B} \given \bar X_{C_1}^{\Ical_{C_1}}, \ldots, \bar X_{C_r}^{\Ical_{C_r}}$ under $\bar\PP$. By Lemma \ref{lem:timesplit_embedding}(iii) with $S = \emptyset$, $\bar\PP(\mathrm{res}(\bar X_V)) = \PP(X_V)$, and hence $X_A^{\Ical_A} \Indep X_B^{\Ical_B} \given X_{C_1}^{\Ical_{C_1}}, \ldots, X_{C_r}^{\Ical_{C_r}}$ under $\PP$.

    \emph{(ii)} Let $E_\Ecal := \bigcup_{\Ical\in\Ecal}\{\inf\Ical, \sup\Ical\}$ be the set of end-points of the partition $\Ecal$, and consider the Euler scheme $X^{\Delta}_V$ on a grid containing $E_\Ecal$. Write $X_{C}^{\Ical_C} := (X_{C_1}^{\Ical_{C_1}},\ldots,X_{C_r}^{\Ical_{C_r}})$ for brevity, and similarly $(X_{C}^\Delta)^{\Ical_C}$ and $(X_{C}^n)^{\Ical_C}$ for the Euler-grid and continuous-Euler analogues.
    Suppose $X_A^{\Ical_A}\Perp^d X_B^{\Ical_B} \given X_{C}^{\Ical_C}$ in the time-split graph $G(\Dcal^\Ecal)$. By the same argument as Lemma \ref{thm:dsep_summary_euler} adapted to the time-split graph (collapsing time-points within each interval to its $(v,\Ical)$-vertex respects interval membership since the Euler grid contains $E_\Ecal$), this gives
    $(X_A^\Delta)^{\Ical_A}\Perp^d (X_B^\Delta)^{\Ical_B} \given (X_{C}^\Delta)^{\Ical_C}$ in the Euler graph $G(\Mcal_\Dcal^\Delta)$, where we write $(X_A^\Delta)^{\Ical_A} = (X_A^\Delta(t_k) : t_k \in \Ical_A)$. Since the Euler graph is acyclic, the $d$-separation Markov property gives $(X_A^\Delta)^{\Ical_A}\Indep (X_B^\Delta)^{\Ical_B} \given (X_{C}^\Delta)^{\Ical_C}$. Let $X_V^n$ be the continuous Euler scheme from~\eqref{eqn:euler_continuous}. By the argument in the proof of Theorem~\ref{thm:continuous_euler_dsep_mp} applied to the interval-restricted blocks --- the gridpoints (functions of the increments) being independent of the Brownian bridges, with the gridpoint conditional independence above and the bridge conditional independence from the shared-noise structure --- the continuous Euler scheme satisfies $(X_A^n)^{\Ical_A}\Indep (X_B^n)^{\Ical_B} \given (X_{C}^n)^{\Ical_C}$. Let $\mathrm{proj}(X_V) := (X_A^{\Ical_A}, X_B^{\Ical_B}, X_{C}^{\Ical_C})$, then
    \begin{align*}
        d_{TV}(\PP(\mathrm{proj}(X^n_V)), \PP(\mathrm{proj}(X_V))) & = \sup_{A \in \sigma(\mathrm{proj})} \left|\PP(X^n_V \in A) - \PP(X_V\in A)\right|     \\
                                                                   & \leq \sup_{A \in \Bcal(C([0,T], \RR^d))} \left|\PP(X^n_V \in A) - \PP(X_V\in A)\right| \\
                                                                   & = d_{TV}(\PP(X^n_V), \PP(X_V)).
    \end{align*}
    By Theorem \ref{thm:euler_tv_convergence} there exists a subsequence $n_m$ such that the right-hand side converges to 0, and by Theorem \ref{thm:ci_tv_closed} we conclude that $X_A^{\Ical_A} \Indep X_B^{\Ical_B} \given X_{C}^{\Ical_C}$.
\end{proof}

\docalctimesplit*
\begin{proof}
    We prove Rule 2 explicitly; Rules 1 and 3 follow from similar arguments.
    Suppose that for $\Dcal^\Ecal$ we have $X_A^{\Ical_A} \Perp^\sigma_{G(\Dcal^\Ecal)_{\Do(R_B^{\Ical_B})}} R_B^{\Ical_B} \given X_B^{\Ical_B}, X_{C}^{\Ical_C}$, and $\PP_{\Dcal^\Ecal}(X_{C}^{\Ical_C} \mid X_B^{\Ical_B} = x_B) \ll \PP_{\Dcal^\Ecal}(X_{C}^{\Ical_C} \mid \Do(X_B^{\Ical_B} = x_B))$ for $\PP_{\Dcal^\Ecal}(X_B^{\Ical_B})$-a.a.\ $x_B$. By Lemma \ref{lem:timesplit_embedding}(ii), $G(\bar\Dcal^\Ecal) = G(\Dcal^\Ecal)$, and so $X_A^{\Ical_A} \Perp^\sigma_{G(\bar \Dcal^\Ecal)_{\Do(R_B^{\Ical_B})}} R_B^{\Ical_B} \given X_B^{\Ical_B}, X_{C}^{\Ical_C}$ as well. By Lemma \ref{lem:timesplit_embedding}(iii), for every $S\subseteq V\times\Ecal$ (write $O := (V\times\Ecal)\setminus S$) and every $x_S\in\Xcal_S$,
    \begin{equation}\label{eqn:docalctimesplit_pushforward}
        \PP_{\Dcal^\Ecal}\bigl(X_O \given \Do(X_S = x_S)\bigr) = \PP_{\bar\Dcal^\Ecal}\bigl(\mathrm{res}(\bar X_O) \given \Do(\bar X_S = \mathrm{ext}(x_S))\bigr),
    \end{equation}
    and similarly:
    \begin{equation}\label{eqn:docalctimesplit_pushforward_2}
        \PP_{\Dcal^\Ecal}\bigl(\mathrm{ext}(X_O) \given \Do(X_S = x_S)\bigr) = \PP_{\bar\Dcal^\Ecal}\bigl(\bar X_O \given \Do(\bar X_S = \mathrm{ext}(x_S))\bigr).
    \end{equation}
    Here \eqref{eqn:docalctimesplit_pushforward_2} follows by applying $\mathrm{ext}$ to \eqref{eqn:docalctimesplit_pushforward}: by Lemma~\ref{lem:timesplit_embedding}(i) the intervened solution of $\bar\Dcal^\Ecal$ is $\bar X_O = \bar I^{[O]}(\mathrm{ext}(x_S), \bar X_W^\Ecal) = \mathrm{ext}\bigl(I^{[O]}(x_S, X_W^\Ecal)\bigr)$, which lies in the image of $\mathrm{ext}$, so $\bar X_O = \mathrm{ext}(\mathrm{res}(\bar X_O))$.
    Since absolute continuity is preserved under push-forwards, \eqref{eqn:docalctimesplit_pushforward_2} gives $\PP_{\bar\Dcal^\Ecal}(\bar X_{C}^{\Ical_C} \mid \bar X_B^{\Ical_B}=\bar x_B) \ll \PP_{\bar\Dcal^\Ecal}(\bar X_{C}^{\Ical_C} \mid \Do(\bar X_B^{\Ical_B}=\bar x_B))$ for $\PP_{\bar\Dcal^\Ecal}(\bar X_B^{\Ical_B})$-almost all $\bar x_B$.
    The embedded time-split system $\bar\Dcal^\Ecal$ is a system of causal SDEs on $[0,T]$; since $\Dcal^\Ecal$ is essentially uniquely solvable w.r.t.\ every $O \subseteq V\times\Ecal$ (Lemma~\ref{lem:timesplit_solvable}), so is $\bar\Dcal^\Ecal$ by Lemma~\ref{lem:timesplit_embedding}(i), and Theorem~\ref{thm:do_calculus} applies to $\bar\Dcal^\Ecal$.
    Applying Rule 2 to $\bar\Dcal^\Ecal$ then gives
    \begin{equation*}
        \PP_{\bar\Dcal^\Ecal}\bigl(\bar X_A^{\Ical_A} \given \bar X_B^{\Ical_B}, \bar X_{C}^{\Ical_C}\bigr) = \PP_{\bar\Dcal^\Ecal}\bigl(\bar X_A^{\Ical_A} \given \Do(\bar X_B^{\Ical_B}), \bar X_{C}^{\Ical_C}\bigr) \quad \PP_{\bar\Dcal^\Ecal}(\bar X_B^{\Ical_B}, \bar X_{C}^{\Ical_C})\text{-a.s.}
    \end{equation*}
    Componentwise pushforward by $\mathrm{res}$ (using \eqref{eqn:docalctimesplit_pushforward}) yields
    $\PP_{\Dcal^\Ecal}(X_A^{\Ical_A} \given X_B^{\Ical_B}, X_{C}^{\Ical_C}) = \PP_{\Dcal^\Ecal}(X_A^{\Ical_A} \given \Do(X_B^{\Ical_B}), X_{C}^{\Ical_C})$ $\PP_{\Dcal^\Ecal}(X_B^{\Ical_B}, X_{C}^{\Ical_C})$-a.s., establishing Rule 2 for $\Dcal^\Ecal$.
\end{proof}

\subsection{Proofs of Section \ref{sec:granger_causality}}\label{app:granger}
\grangeriscausation*
\begin{proof}
    \emph{(i)} Let $u = z_0 \to z_1 \to \cdots \to z_n = v$ be a directed path in $G(\Dcal)$, and fix any $0\leq s < T$. By Definition~\ref{def:sdes_timesplit}, the causal SDE for $X_u^{(s,T]}$ contains the initial-value term $X_u^{[0,s]}(b)$, giving the self-edge $u^{[0,s]} \to u^{(s,T]}$ in $G(\Dcal^{\{[0,s], (s,T]\}})$. Within $(s,T]$, the causal SDE for each $X_{z_{i+1}}^{(s,T]}$ has parent set $\alpha(z_{i+1})\cup\beta(z_{i+1})\cup\gamma(z_{i+1})$, inheriting the edge $z_i^{(s,T]} \to z_{i+1}^{(s,T]}$ from $z_i \to z_{i+1}$ in $G(\Dcal)$. Composing these yields the walk
    \begin{equation*}
        u^{[0,s]} \to u^{(s,T]} \to z_1^{(s,T]} \to \cdots \to z_{n-1}^{(s,T]} \to v^{(s,T]}
    \end{equation*}
    in $G(\Dcal^{\{[0,s], (s,T]\}})$. Every non-endpoint node lies in $(s,T]$ and so is not in the conditioning set $X_{V\setminus\{u\}}^{[0,s]}$, so the walk is $\sigma$-active. By faithfulness of $\Dcal^{\{[0,s], (s,T]\}}$, $X_u^{[0,s]} \nIndep X_v^{(s,T]} \given X_{V\setminus\{u\}}^{[0,s]}$, so $u$ is a Granger cause of $v$.

    \emph{(ii)} We prove the contrapositive. Suppose there is no directed path from $u$ to $v$ in $G(\Dcal)$, and there is no latent confounding (so $G(\Dcal^{\{[0,s], (s,T]\}})$ contains no bidirected edges). Fix $0\leq s < T$ and write $C := X_{V\setminus\{u\}}^{[0,s]}$.
    Let $\pi$ be any walk from $u^{[0,s]}$ to $v^{(s,T]}$; we show that it is $\sigma$-blocked given $C$. Let $w^{[0,s]}$ be the \emph{last} node of $\pi$ lying in $[0,s]$, so all subsequent nodes lie in $(s,T]$. By adaptedness $G(\Dcal^{\{[0,s], (s,T]\}})$ has no edges from $(s,T]$ to $[0,s]$, so the edge from $w^{[0,s]}$ to the next node is of the form $w^{[0,s]} \to w'^{(s,T]}$; hence $w^{[0,s]}$ is a non-collider that points to $w'^{(s,T]}$, which lies in a different interval and thus a different strongly connected component (SCCs of the time-split graph never span intervals), so $w^{[0,s]}$ is blockable. If $w \neq u$, then $w^{[0,s]} \in C$ is a blockable non-collider, so $\pi$ is $\sigma$-blocked. If $w = u$, then every node of $\pi$ after $u^{[0,s]}$ lies in $(s,T]$. Suppose $\pi$ were $\sigma$-open; since $C \subseteq [0,s]$ and no $(s,T]$-node is an ancestor of a $[0,s]$-node, any collider on the portion after $u^{[0,s]}$ would lie outside $\Anc(C)$ and block $\pi$, so this portion is collider-free. As its first edge $u^{[0,s]} \to w'^{(s,T]}$ is outgoing, it is then a directed path $u^{[0,s]} \to w'^{(s,T]} \to \cdots \to v^{(s,T]}$, giving a directed path from $u$ to $v$ in $G(\Dcal)$, contradicting the assumption; hence $\pi$ is $\sigma$-blocked.
    Therefore every walk from $u^{[0,s]}$ to $v^{(s,T]}$ is $\sigma$-blocked given $C$, and by the $\sigma$-separation Markov property (Theorem~\ref{thm:markov_timesplit}(i)), $X_u^{[0,s]} \Indep X_v^{(s,T]} \given X_{V\setminus\{u\}}^{[0,s]}$ for every $0\leq s < T$. Hence $u$ is not a Granger cause of $v$.
\end{proof}

\subsubsection{Proofs of Section \ref{sec:local_independence}}\label{app:local_independence}
\begin{definition}[Quasimartingale]\label{def:quasimartingale}
    A real-valued càdlàg process $X$ adapted to a filtration $\Fcal$ on $[0,T]$ with $\EE[|X(t)|] < \infty$ for all $t\in[0,T]$ is a \emph{quasimartingale} with respect to $\Fcal$ if
    \begin{equation*}
        \mathrm{Var}(X, \Fcal) := \sup_{\tau} \mathrm{Var}_\tau(X, \Fcal) < \infty,
        \qquad \text{where} \quad
        \mathrm{Var}_\tau(X, \Fcal) := \EE\Bigl[\sum_{i=0}^{n-1} \bigl|\EE[X(t_{i+1}) - X(t_i) \given \Fcal_{t_i}]\bigr|\Bigr]
    \end{equation*}
    for a finite partition $\tau = \{0 = t_0 < t_1 < \cdots < t_n = T\}$ of $[0,T]$, and the supremum ranges over all such $\tau$. A vector-valued process is a quasimartingale if each of its components is.
\end{definition}
By Rao's theorem (\cite{protter2005stochastic}, III.18) every càdlàg quasimartingale is a special semimartingale. Since a quasimartingale remains a quasimartingale under optional projection \citep[Theorem 2.4]{follmer2011local}, the notion of local independence is well-defined for quasimartingales:
\begin{lemapp}\label{thm:bounded_implies_special_projection}
    If $X_B$ is a quasimartingale with respect to $\Fcal^V$, then for every $C \subseteq V$ the optional projection $\EE[X_B(t)\given \Fcal^C_t]$ exists and is a quasimartingale with respect to $\Fcal^C$, hence a special semimartingale. In particular the notion of local independence is well-defined.
\end{lemapp}

For proving Theorem \ref{thm:global_granger_implies_local_independence} we need some additional definitions and a technical lemma. For a càdlàg process $\Lambda$ we write
\begin{equation*}
    \int_0^t |\diff\Lambda_s| := \sup_{0 = t_0 < t_1 < \cdots < t_n = t} \sum_{i=1}^n \|\Lambda(t_i) - \Lambda(t_{i-1})\|
\end{equation*}
for its \emph{total variation} on $[0,t]$, a value in $[0,\infty]$. Then $\Lambda$ is \emph{of finite variation} if $\int_0^T |\diff\Lambda_s| < \infty$ almost surely, and \emph{of integrable variation} if $\EE\bigl[\int_0^T |\diff\Lambda_s|\bigr] < \infty$. Note that this total variation differs from the mean variation $\mathrm{Var}(\cdot, \Fcal)$ of Definition~\ref{def:quasimartingale} (applied componentwise to vector-valued processes), which conditions the increments on $\Fcal$ before taking absolute values. A family $\{\xi_i\}_{i\in I}$ of integrable random variables is \emph{uniformly integrable} if $\lim_{c\to\infty} \sup_{i\in I} \EE[\|\xi_i\| \I\{\|\xi_i\| > c\}] = 0$; a c\`adl\`ag adapted process $X$ with $\EE[|X(t)|] < \infty$ for all $t$ is \emph{of class~(D)} if the family $\{X(\tau) : \tau \text{ a finite-valued stopping time}\}$ is uniformly integrable; and a martingale $M$ is a \emph{uniformly integrable martingale} if the family $\{M(t) : t\in[0,T]\}$ is uniformly integrable.

\begin{lemapp}\label{lem:projection_class_D}
    Let $X_B$ be a quasimartingale with respect to $\Fcal^V$ (componentwise) such that $\EE\bigl[\sup_{t\in[0,T]}\|X_B(t)\|\bigr] < \infty$, and let $C \subseteq V$. Then the optional projection $Y(t) := \EE[X_B(t)\given \Fcal^C_t]$ is of class~(D), and its canonical decomposition $Y = Y(0) + \Lambda^C + M^C$ has $\Lambda^C$ of integrable variation, and $M^C$ a uniformly integrable martingale.
\end{lemapp}
\begin{proof}
    Write $\chi := \sup_{t\in[0,T]}\|X_B(t)\|$, which is integrable by assumption. For any finite-valued $\Fcal^C$-stopping time $\tau$ we have $\|Y(\tau)\| = \|\EE[X_B(\tau)\given \Fcal_\tau^C]\| \leq \EE[\chi \given \Fcal_\tau^C]$, and the family $\{\EE[\chi\given\Gcal] : \Gcal \subseteq \Fcal \text{ a sub-}\sigma\text{-algebra}\}$ is uniformly integrable since $\chi$ is integrable. Hence $\{Y(\tau)\}$ is uniformly integrable, i.e.\ $Y$ is of class~(D).

    By Rao's theorem (\cite{protter2005stochastic}, III.17) applied in the filtration $\Fcal^C$, we may write $Y = U - W$ with $U, W$ positive right-continuous $\Fcal^C$-supermartingales. By the Doob-Meyer decomposition (\cite{protter2005stochastic}, III.16) we have $U = U(0) + M^U - A^U$ with $M^U$ an $\Fcal^C$-local martingale and $A^U$ an increasing predictable process with $A^U_0 = 0$; since $U \geq 0$ we have $\lim_{t}\EE[U(t)] \geq 0 > -\infty$, so that same theorem gives $\EE[A^U_T] < \infty$. Writing $W = W(0) + M^W - A^W$ analogously, we obtain
    \begin{equation*}
        Y = Y(0) + (M^U - M^W) - (A^U - A^W),
    \end{equation*}
    a decomposition into an $\Fcal^C$-local martingale and a predictable finite-variation process vanishing at $0$; by uniqueness of the decomposition (\cite{protter2005stochastic}, III.34) we have $\Lambda^C = -(A^U - A^W)$ and $M^C = M^U - M^W$. Since $A^U$ and $A^W$ are increasing with $A^U_0 = A^W_0 = 0$, their total variations are $\int_0^T |\diff A^U_s| = A^U_T$ and $\int_0^T |\diff A^W_s| = A^W_T$; as $\Lambda^C$ is their difference, subadditivity of the total variation gives $\int_0^T|\diff\Lambda^C_s| \leq A^U_T + A^W_T$, so $\EE[\int_0^T |\diff\Lambda^C_s|] < \infty$.

    For any finite-valued $\Fcal^C$-stopping time $\tau$ we have $\|M^C(\tau)\| \leq \|Y(\tau)\| + \|Y(0)\| + \int_0^T|\diff\Lambda^C_s|$, where the first family is uniformly integrable by the above and the remaining two terms are integrable random variables not depending on $\tau$, hence $M^C$ is a local martingale of class~(D). Being a local martingale, $M^C$ admits a localising sequence $\tau_n \uparrow \infty$ for which each stopped process $M^C(\cdot\wedge\tau_n)$ is a martingale, so for $s\leq t$
    \begin{equation}\label{eqn:stopped_martingale}
        \EE[M^C(t\wedge\tau_n)\given\Fcal^C_s] = M^C(s\wedge\tau_n).
    \end{equation}
    Since $\tau_n \uparrow \infty$, both $M^C(t\wedge\tau_n) \to M^C(t)$ and $M^C(s\wedge\tau_n) \to M^C(s)$ almost surely as $n\to\infty$. Each $t\wedge\tau_n$ is a finite $\Fcal^C$-stopping time, so $\{M^C(t\wedge\tau_n) : n\geq 1\}$ is a subfamily of the class~(D) family and is uniformly integrable; combined with the almost-sure convergence, the $L^1$-convergence criterion (\cite{kallenberg2021foundations}, Theorem 5.12) gives $M^C(t\wedge\tau_n) \to M^C(t)$ in $L^1$. By the conditional Jensen inequality, $\EE\bigl|\EE[Z\given\Fcal^C_s]\bigr| \leq \EE\bigl[\EE[|Z|\given\Fcal^C_s]\bigr] = \EE|Z|$ for any integrable $Z$. Applying this to $Z = M^C(t\wedge\tau_n) - M^C(t)$ gives $\EE\big[|\EE[M^C(t\wedge\tau_n)\given\Fcal^C_s] - \EE[M^C(t)\given\Fcal^C_s]|\big] \leq \EE\big[|M^C(t\wedge\tau_n) - M^C(t)|\big] \to 0$, so the left-hand side of (\ref{eqn:stopped_martingale}) converges to $\EE[M^C(t)\given\Fcal^C_s]$ in $L^1$, and hence in probability. The right-hand side of (\ref{eqn:stopped_martingale}) converges to $M^C(s)$ almost surely, and hence in probability, and thus $\EE[M^C(t)\given\Fcal^C_s] = M^C(s)$ almost surely. Hence $M^C$ is a martingale.
\end{proof}

In the original work of \cite{florens1996noncausality}, it is claimed for the setting where $B\subseteq C$ that if $X_A^{[0,s]}\Indep X_B^{(s,t]} \given X_C^{[0,s]}$, then in the Doob-Meyer decomposition $X_B = \Lambda^C + M^C$, the $\Fcal_t^C$-local martingale $M^C$ is a $\Fcal_t^{A,C}$-local martingale, in which case we also have $\Lambda^C = \Lambda^{A,C}$ and hence local independence.
That $M^C$ is a $\Fcal_t^{A,C}$-local martingale should follow from their Corollary 2.1, saying that if $X_A^{[0,s]}\Indep X_B^{(s,t]} \given X_C^{[0,s]}$, then any $\Fcal_t^B$-adapted $\Fcal_t^C$-local martingale is an $\Fcal_t^{A, C}$-local martingale. However, the $\Fcal_t^C$-local martingale $M^C$ is not necessarily $\Fcal_t^B$-adapted, so one cannot apply Corollary 2.1 to obtain that $M^C$ is also a $\Fcal_t^{A,C}$-local martingale. The following proof corrects this error and extends the theorem to the setting where $B$ is not contained in $C$.

\globalgrangerimplieslocalindependence*
\begin{proof}
    We must show that $\Lambda_t^{C} = \Lambda_t^{A,C}$ a.s. -- we will derive this property via the definition of these processes as given in the proof of the Doob-Meyer decomposition from \cite{kallenberg2021foundations}, Theorem 10.5. Let $X_B^{C}(t) := \EE[X_B(t)\given \Fcal_t^C]$ and $X_B^{A,C}(t) := \EE[X_B(t)\given \Fcal_t^{A,C}]$. Consider the points $t_j^n := jT/2^n$ that define the dyadic partition of $[0,T]$. Define
    \begin{equation}\label{eqn:doob_meyer_last_element}
        (\alpha_T^{C})^n := \sum_{j=1}^{2^n} \EE[X_B^{C}(t_j^n) - X_B^C(t_{j-1}^n)\given \Fcal^{C}_{t_{j-1}^n}].
    \end{equation}
    We first show that $\{(\alpha_T^{C})^n : n\geq 1\}$ is uniformly integrable. By Lemma~\ref{thm:bounded_implies_special_projection}, $X_B^C$ is a special semimartingale with canonical decomposition $X_B^C = X_B^C(0) + \Lambda^C + M^C$ in which $\Lambda^C$ is of integrable variation and the martingale $M^C$ is uniformly integrable, by Lemma~\ref{lem:projection_class_D}. Since $M^C$ is a martingale, its increments vanish under the conditional expectations in (\ref{eqn:doob_meyer_last_element}), so
    \begin{equation*}
        (\alpha_T^{C})^n = \sum_{j=1}^{2^n}\EE\bigl[\Lambda^C(t_j^n) - \Lambda^C(t_{j-1}^n) \given \Fcal^C_{t_{j-1}^n}\bigr].
    \end{equation*}
    Write $|\Lambda^C|_t := \int_0^t |\diff \Lambda^C_s|$ for the total variation process of $\Lambda^C$, which is increasing with $\EE[|\Lambda^C|_T] < \infty$ and hence a submartingale of class~(D), and let
    \begin{equation*}
        A^n_t := \sum_{j\,:\,t_j^n \leq t} \EE\bigl[|\Lambda^C|_{t_j^n} - |\Lambda^C|_{t_{j-1}^n} \given \Fcal^C_{t_{j-1}^n}\bigr].
    \end{equation*}
    For each $j$ we have
    \begin{equation*}
        |\Lambda^C(t_j^n) - \Lambda^C(t_{j-1}^n)| = \Bigl|\int_{t_{j-1}^n}^{t_j^n}\diff\Lambda^C_s\Bigr| \leq \int_{t_{j-1}^n}^{t_j^n}|\diff\Lambda^C_s| = |\Lambda^C|_{t_j^n} - |\Lambda^C|_{t_{j-1}^n},
    \end{equation*}
    and thus $|(\alpha_T^{C})^n| \leq A^n_T$ by the triangle and the conditional Jensen inequality. The argument in the proof of Lemma 10.7 of \cite{kallenberg2021foundations} shows that the terminal values $\{A^n_T : n \geq 1\}$ form a uniformly integrable family, so $\{(\alpha_T^C)^n : n\geq 1\}$ is uniformly integrable by domination.

    By the Dunford-Pettis criterion (\cite{kallenberg2021foundations}, Lemma 5.13) there is a subsequence $n_k$ such that $(\alpha_T^{C})^{n_k}$ has a weak limit $\alpha_T^{C} \in L^1(\Fcal_{T}^{C})$. The Doob-Meyer decomposition is then given by
    \begin{equation}\label{eqn:compensator}
        \Lambda^{C}_t := X_B^C(t) - \EE[X_B^C(T) - \alpha^{C}_T \given \Fcal_t^{C}]
    \end{equation}
    and $M_t^{C} := X_B^C(t) - \Lambda^{C}_t$. That $M^{C}_t$ is a martingale is immediate, and see \cite{kallenberg2021foundations} for a proof that $\Lambda^{C}$ is a predictable finite-variation process. We obtain the Doob-Meyer decomposition of $X_B^{A,C}$ with respect to $\Fcal_t^{A,C}$ similarly, via the last element $\alpha_T^{A,C} \in \Fcal_T^{A,C}$ and defining $\Lambda^{A,C}$ in terms of $\alpha_T^{A,C}$. Since $X_B$ is c\`adl\`ag, $X_B(s) = \lim_{\epsilon\downarrow 0}X_B(s+\epsilon)$, so $X_B(s)$ is measurable with respect to $\sigma(X_B^{(s,t]})$, and hence, for $0<s<t$, the assumed conditional independence $X_A^{[0,s]}\Indep X_B^{(s,t]} \given X_C^{[0,s]}$ implies $X_A^{[0,s]}\Indep X_B(s) \given X_C^{[0,s]}$ and thus $X_B^{C}(s) = \EE[X_B(s)\given \Fcal_s^{C}] = \EE[X_B(s)\given \Fcal_s^{A,C}] = X_B^{A,C}(s)$. Since the optional projections are c\`adl\`ag and the filtrations right-continuous, letting $s\downarrow 0$ extends this identity to $s=0$ and yields $\EE[X_B(r)\given \Fcal_0^{C}] = \EE[X_B(r)\given \Fcal_0^{A,C}]$ for every $r\in(0,T]$; hence $X_B^{C} = X_B^{A,C}$. We then see from (\ref{eqn:doob_meyer_last_element}) (and the tower property) that $(\alpha_T^{C})^n = (\alpha_T^{A,C})^n$, and hence also $\alpha_T^{C} = \alpha_T^{A,C}$. We also have for any $t\in [0,T]$ and $i$ such that $t_{i-1}^n \leq t < t_{i}^n$ that
    \begin{equation*}
        \EE[(\alpha_T^{C})^n\given \Fcal^{C}_t] = \sum_{j=1}^i \EE[X_B^C(t_j^n)- X_B^C(t_{j-1}^n) \given \Fcal^{C}_{t_{j-1}^n}] + \sum_{j=i+1}^{2^n} \EE[X_B^C(t_j^n) - X_B^C(t_{j-1}^n)\given \Fcal^{C}_{t}],
    \end{equation*}
    which by the assumed conditional independence and the tower property is equal to $\EE[(\alpha_T^{C})^n\given \Fcal^{A,C}_t]$. Since the map $(\alpha_T^{C})^n \mapsto \EE[(\alpha_T^{C})^n \given \Fcal_t^{C}]$ is continuous from the weak topology $\sigma(L^1(\Fcal_T^{C}), L^\infty(\Fcal_T^{C}))$ to $\sigma(L^1(\Fcal_t^{C}), L^\infty(\Fcal_t^{C}))$ (see e.g.\ the discussion in \cite{dellacherie1978probabilities}, II.42), we get that $\EE[(\alpha_T^{C})^{n_k}\given \Fcal^{C}_t]$ converges weakly to $\EE[\alpha^{C}_T\given \Fcal_t^{C}]$, giving $\EE[\alpha^{C}_T\given \Fcal_t^{C}] = \EE[\alpha^{C}_T\given \Fcal_t^{A,C}]$. By the independence assumption we also have $\EE[X_B(T) \given \Fcal_t^{C}] = \EE[X_B(T) \given \Fcal_t^{A,C}]$, so from (\ref{eqn:compensator}) we obtain $\Lambda^{C} = \Lambda^{A,C}$. Since $X_B^{C} = X_B^{A,C}$ this also gives $M^{C} = M^{A,C}$.

    If $B=C$, the remaining implication follows directly from \cite{bremaud1978changes} Theorem~3, as also remarked by \cite{florens1996noncausality}.
\end{proof}

\begin{lemapp}\label{lem:additive_noise_quasimartingale}
    If $\Dcal$ satisfies Assumption~\ref{ass:additive_noise}, then $X_V$ is a quasimartingale with respect to $\Fcal^V$ and $\EE[\sup_{t\in[0,T]}\|X_V(t)\|] < \infty$.
\end{lemapp}
\begin{proof}
    By Theorems \ref{thm:ess_solvable_under_additive} and \ref{thm:ess_solvable_under_ass}, $\Dcal$ is essentially uniquely solvable w.r.t.\ $V$, so $X_V = I^{[V]}(X_W)$ a.s.\ for a measurable adapted solution function $I^{[V]}$, giving $\Fcal_t^V \subseteq \Fcal_t^W$ for all $t$. Since $X_\gamma$ is a Brownian motion independent of the remaining exogenous processes, its increments after $t$ are independent of $\Fcal_t^W$, so $X_\gamma$ is an $\Fcal^W$-Brownian motion, and by the tower property, for $s\leq t$,
    \begin{equation*}
        \EE\bigl[\Sigma(X_\gamma(t) - X_\gamma(s)) \given \Fcal_s^V\bigr] = \EE\bigl[\EE[\Sigma(X_\gamma(t) - X_\gamma(s))\given \Fcal_s^W] \given \Fcal_s^V\bigr] = 0.
    \end{equation*}
    Hence for any finite partition $\tau = \{t_0, \ldots, t_n\}$ of $[0,T]$, using \eqref{eqn:sde_additive} and the conditional Jensen inequality,
    \begin{align*}
        \sum_{i} \bigl\|\EE[X_V(t_{i+1}) - X_V(t_i) \given \Fcal^V_{t_i}]\bigr\| & = \sum_i \Bigl\|\EE\Bigl[\int_{t_i}^{t_{i+1}}\mu(s, X_V(s))\diff s \Bigm| \Fcal^V_{t_i}\Bigr]\Bigr\| \\
                                                                                 & \leq \sum_i \EE\Bigl[\int_{t_i}^{t_{i+1}}\|\mu(s, X_V(s))\|\diff s \Bigm| \Fcal^V_{t_i}\Bigr].
    \end{align*}
    Taking expectations and using the tower property, the intervals of $\tau$ reassemble into $[0,T]$, so by the linear-growth condition
    \begin{equation*}
        \mathrm{Var}_\tau(X_V, \Fcal^V) \leq \EE\Bigl[\int_0^T \|\mu(s, X_V(s))\|\diff s\Bigr] \leq KT\bigl(1 + \EE[\sup\nolimits_{t\in[0,T]}\|X_V(t)\|]\bigr) < \infty.
    \end{equation*}
    As this bound does not depend on $\tau$, we have $\mathrm{Var}(X_V, \Fcal^V) < \infty$, and $\EE[\|X_V(t)\|] < \infty$ for every $t$, so $X_V$ is a quasimartingale with respect to $\Fcal^V$.

    For the integrability condition, writing $K := \sup_{0\leq t\leq T}K(t) < \infty$ for the linear-growth constant of $\mu$ (finite as $K$ is c\`adl\`ag on a compact interval), Gr\"onwall's inequality gives $\sup_{0\leq t\leq T}\|X_V(t)\| \leq (\|X_\alpha^0\| + KT + \sup_{0\leq s\leq T}\|\Sigma X_\gamma(s)\|)e^{KT}$ (as in Step 2 of the proof of Theorem~\ref{thm:euler_tv_convergence}). Since $\EE[\|X_\alpha^0\|^2]<\infty$ by Assumption~\ref{ass:additive_noise} and $\EE[\sup_{0\leq s\leq T}\|\Sigma X_\gamma(s)\|^2]<\infty$ by Doob's maximal inequality (\citealp{protter2005stochastic}, Theorem I.20) applied to the nonnegative submartingale $\|\Sigma X_\gamma\|$ (a convex function of the continuous martingale $\Sigma X_\gamma$), we obtain $\EE[\sup_{0\leq t\leq T}\|X_V(t)\|] < \infty$.
\end{proof}
\end{appendix}

\bibliographystyle{apalike}
\bibliography{refs}       

\begin{thebibliography}{}

\bibitem[Aalen, 1978]{aalen1978nonparametric}
Aalen, O. (1978).
\newblock Nonparametric {{Inference}} for a {{Family}} of {{Counting
  Processes}}.
\newblock {\em The Annals of Statistics}, 6(4):701--726.

\bibitem[Aalen et~al., 2016]{aalen2016can}
Aalen, O., R{\o}ysland, K., Gran, J., Kouyos, R., and Lange, T. (2016).
\newblock Can we believe the {{DAGs}}? {{A}} comment on the relationship
  between causal {{DAGs}} and mechanisms.
\newblock {\em Statistical Methods in Medical Research}, 25(5):2294--2314.

\bibitem[Aalen et~al., 2008]{aalen2008survival}
Aalen, O.~O., Borgan, {\O}., and Gjessing, H.~K. (2008).
\newblock {\em Survival and Event History Analysis: A Process Point of View}.
\newblock Statistics for Biology and Health. Springer, New York.

\bibitem[Aalen et~al., 2012]{aalen2012causality}
Aalen, O.~O., R{\o}ysland, K., Gran, J.~M., and Ledergerber, B. (2012).
\newblock Causality, mediation and time: A dynamic viewpoint.
\newblock {\em Journal of the Royal Statistical Society: Series A (Statistics
  in Society)}, 175(4):831--861.

\bibitem[Andrews et~al., 2020]{andrews2020completeness}
Andrews, B., Spirtes, P., and Cooper, G.~F. (2020).
\newblock On the {{Completeness}} of {{Causal Discovery}} in the {{Presence}}
  of {{Latent Confounding}} with {{Tiered Background Knowledge}}.
\newblock In {\em Proceedings of the {{Twenty Third International Conference}}
  on {{Artificial Intelligence}} and {{Statistics}}}, pages 4002--4011. PMLR.

\bibitem[Assaad et~al., 2022]{assaad2022survey}
Assaad, C., Devijver, E., and Gaussier, E. (2022).
\newblock Survey and {{Evaluation}} of {{Causal Discovery Methods}} for {{Time
  Series}}.
\newblock {\em Journal of Artificial Intelligence Research}, 73:767--819.

\bibitem[Blom and Mooij, 2023]{blom2023causality}
Blom, T. and Mooij, J. (2023).
\newblock Causality and independence in perfectly adapted dynamical systems.
\newblock {\em Journal of Causal Inference}, 11(1).

\bibitem[Boeken and Mooij, 2024]{boeken2024dynamic}
Boeken, P. and Mooij, J.~M. (2024).
\newblock Dynamic {{Structural Causal Models}}.

\bibitem[Boeken et~al., 2026]{boeken2026topological}
Boeken, P., Skapinakis, E., Genin, K., and Mooij, J.~M. (2026).
\newblock Topological {{Criteria}} for {{Hypothesis Testing}} with
  {{Finite-Precision Measurements}}.

\bibitem[Bongers et~al., 2022]{bongers2022causal}
Bongers, S., Blom, T., and Mooij, J. (2022).
\newblock Causal {{Modeling}} of {{Dynamical Systems}}.

\bibitem[Bongers et~al., 2021]{bongers2021foundations}
Bongers, S., Forr{\'e}, P., Peters, J., and Mooij, J. (2021).
\newblock Foundations of structural causal models with cycles and latent
  variables.
\newblock {\em The Annals of Statistics}, 49(5).

\bibitem[Bowsher, 2010]{bowsher2010stochastic}
Bowsher, C.~G. (2010).
\newblock Stochastic kinetic models: {{Dynamic}} independence, modularity and
  graphs.
\newblock {\em The Annals of Statistics}, 38(4):2242--2281.

\bibitem[Br{\'e}maud and Yor, 1978]{bremaud1978changes}
Br{\'e}maud, P. and Yor, M. (1978).
\newblock Changes of filtrations and of probability measures.
\newblock {\em Zeitschrift f{\"u}r Wahrscheinlichkeitstheorie und Verwandte
  Gebiete}, 45(4):269--295.

\bibitem[Br{\"u}ck et~al., 2026]{bruck2026graph}
Br{\"u}ck, F., Engelke, S., and Volgushev, S. (2026).
\newblock Graph structure learning for stable processes.

\bibitem[Christgau et~al., 2023]{christgau2023nonparametric}
Christgau, A., Petersen, L., and Hansen, N. (2023).
\newblock Nonparametric conditional local independence testing.
\newblock {\em The Annals of Statistics}, 51(5).

\bibitem[Comte and Renault, 1996]{comte1996noncausality}
Comte, F. and Renault, E. (1996).
\newblock Noncausality in {{Continuous Time Models}}.
\newblock {\em Econometric Theory}, 12(2):215--256.

\bibitem[Cooper, 1997]{cooper1997simple}
Cooper, G. (1997).
\newblock A {{Simple Constraint-Based Algorithm}} for {{Efficiently Mining
  Observational Databases}} for {{Causal Relationships}}.
\newblock {\em Data Mining and Knowledge Discovery}.

\bibitem[Dawid, 2021]{dawid2021decisiontheoretic}
Dawid, P. (2021).
\newblock Decision-theoretic foundations for statistical causality.
\newblock {\em Journal of Causal Inference}, 9(1):39--77.

\bibitem[Dellacherie and Meyer, 1978]{dellacherie1978probabilities}
Dellacherie, C. and Meyer, P. (1978).
\newblock {\em Probabilities and Potential}.
\newblock Number~29 in North-{{Holland}} Mathematics Studies. Hermann [u.a.],
  Paris.

\bibitem[Didelez, 2008]{didelez2008graphical}
Didelez, V. (2008).
\newblock Graphical {{Models}} for {{Marked Point Processes Based}} on {{Local
  Independence}}.
\newblock {\em Journal of the Royal Statistical Society. Series B (Statistical
  Methodology)}, 70(1):245--264.

\bibitem[Eichler, 2012]{eichler2012graphical}
Eichler, M. (2012).
\newblock Graphical modelling of multivariate time series.
\newblock {\em Probability Theory and Related Fields}, 153(1):233--268.

\bibitem[Eichler and Didelez, 2010]{eichler2010granger}
Eichler, M. and Didelez, V. (2010).
\newblock On {{Granger}} causality and the effect of interventions in time
  series.
\newblock {\em Lifetime Data Analysis}, 16(1):3--32.

\bibitem[Elowitz and Leibler, 2000]{elowitz2000synthetic}
Elowitz, M.~B. and Leibler, S. (2000).
\newblock A synthetic oscillatory network of transcriptional regulators.
\newblock {\em Nature}, 403(6767):335--338.

\bibitem[Engelke et~al., 2024]{engelke2024levy}
Engelke, S., Ivanovs, J., and Th{\o}stesen, J. (2024).
\newblock L{\textbackslash}'evy graphical models.

\bibitem[Ferreira and Assaad, 2024]{ferreira2024identifying}
Ferreira, S. and Assaad, C. (2024).
\newblock Identifying macro conditional independencies and macro total effects
  in summary causal graphs with latent confounding.

\bibitem[Florens and Fougere, 1996]{florens1996noncausality}
Florens, J. and Fougere, D. (1996).
\newblock Noncausality in {{Continuous Time}}.
\newblock {\em Econometrica}, 64(5):1195--1212.

\bibitem[F{\"o}llmer and Protter, 2011]{follmer2011local}
F{\"o}llmer, H. and Protter, P. (2011).
\newblock Local martingales and filtration shrinkage.
\newblock {\em ESAIM: Probability and Statistics}, 15:S25--S38.

\bibitem[Forr{\'e} and Mooij, 2017]{forre2017markov}
Forr{\'e}, P. and Mooij, J. (2017).
\newblock Markov {{Properties}} for {{Graphical Models}} with {{Cycles}} and
  {{Latent Variables}}.

\bibitem[Forr{\'e} and Mooij, 2020]{forre2020causal}
Forr{\'e}, P. and Mooij, J. (2020).
\newblock Causal {{Calculus}} in the {{Presence}} of {{Cycles}}, {{Latent
  Confounders}} and {{Selection Bias}}.
\newblock In {\em PMLR}, pages 71--80. {PMLR}.

\bibitem[Forr{\'e} and Mooij, 2025]{forre2025mathematical}
Forr{\'e}, P. and Mooij, J. (2025).
\newblock A {{Mathematical Introduction}} to {{Causality}}.

\bibitem[Gill and Robins, 2001]{gill2001causal}
Gill, R. and Robins, J. (2001).
\newblock Causal {{Inference}} for {{Complex Longitudinal Data}}: {{The
  Continuous Case}}.
\newblock {\em The Annals of Statistics}, 29(6):1785--1811.

\bibitem[Granger, 1969]{granger1969investigating}
Granger, C. (1969).
\newblock Investigating {{Causal Relations}} by {{Econometric Models}} and
  {{Cross-spectral Methods}}.
\newblock {\em Econometrica}, 37(3):424--438.

\bibitem[Granger, 1980]{granger1980testing}
Granger, C. (1980).
\newblock Testing for causality: {{A}} personal viewpoint.
\newblock {\em Journal of Economic Dynamics and Control}, 2:329--352.

\bibitem[Guan et~al., 2024]{guan2024identifying}
Guan, V., Janssen, J., Rahmani, H., Warren, A., Zhang, S., Robeva, E., and
  Schiebinger, G. (2024).
\newblock Identifying {{Drift}}, {{Diffusion}}, and {{Causal Structure}} from
  {{Temporal Snapshots}}.

\bibitem[Hansen and Sokol, 2014]{hansen2014causal}
Hansen, N. and Sokol, A. (2014).
\newblock Causal interpretation of stochastic differential equations.
\newblock {\em Electronic Journal of Probability}, 19(none):1--24.

\bibitem[Jacod and Shiryaev, 2003]{jacod2003limit}
Jacod, J. and Shiryaev, A. (2003).
\newblock {\em Limit {{Theorems}} for {{Stochastic Processes}}}, volume 288 of
  {\em Grundlehren Der Mathematischen {{Wissenschaften}}}.
\newblock Springer Berlin Heidelberg, Berlin, Heidelberg.

\bibitem[Kallenberg, 1996]{kallenberg1996existence}
Kallenberg, O. (1996).
\newblock On the existence of universal functional solutions to classical
  {{SDE}}'s.
\newblock {\em The Annals of Probability}, 24(1).

\bibitem[Kallenberg, 2021]{kallenberg2021foundations}
Kallenberg, O. (2021).
\newblock {\em Foundations of {{Modern Probability}}}, volume~99 of {\em
  Probability {{Theory}} and {{Stochastic Modelling}}}.
\newblock Springer International Publishing, Cham.

\bibitem[Karandikar, 1995]{karandikar1995pathwise}
Karandikar, R.~L. (1995).
\newblock On pathwise stochastic integration.
\newblock {\em Stochastic Processes and their Applications}, 57(1):11--18.

\bibitem[Karatzas and Shreve, 1988]{karatzas1988brownian}
Karatzas, I. and Shreve, S. (1988).
\newblock {\em Brownian {{Motion}} and {{Stochastic Calculus}}}, volume 113 of
  {\em Graduate {{Texts}} in {{Mathematics}}}.
\newblock Springer US, New York, NY.

\bibitem[Laumann et~al., 2023]{laumann2023kernelbased}
Laumann, F., {von K{\"u}gelgen}, J., Park, J., Sch{\"o}lkopf, B., and Barahona,
  M. (2023).
\newblock Kernel-{{Based Independence Tests}} for {{Causal Structure Learning}}
  on {{Functional Data}}.
\newblock {\em Entropy}, 25(12):1597.

\bibitem[Lauritzen, 2024]{lauritzen2024total}
Lauritzen, S. (2024).
\newblock Total variation convergence preserves conditional independence.
\newblock {\em Statistics \& Probability Letters}, 214:110200.

\bibitem[Liptser and Shiryaev, 2001]{liptser2001statistics}
Liptser, R. and Shiryaev, A. (2001).
\newblock {\em Statistics of Random Processes: {{I}}: {{General}} Theory}.
\newblock Number~5 in Applications of {{Mathematics}}. Springer, Berlin,
  Germany ; Heidelberg, Germany.

\bibitem[Lok, 2008]{lok2008statistical}
Lok, J. (2008).
\newblock Statistical {{Modeling}} of {{Causal Effects}} in {{Continuous
  Time}}.
\newblock {\em The Annals of Statistics}, 36(3):1464--1507.

\bibitem[Lundborg et~al., 2022]{lundborg2022conditional}
Lundborg, A., Shah, R., and Peters, J. (2022).
\newblock Conditional {{Independence Testing}} in {{Hilbert Spaces}} with
  {{Applications}} to {{Functional Data Analysis}}.
\newblock {\em J. R. Stat. Soc. Ser. B Methodol.}, 84(5):1821--1850.

\bibitem[Lyons et~al., 2007]{lyons2007differential}
Lyons, T., Caruana, M., and L{\'e}vy, T. (2007).
\newblock {\em Differential {{Equations Driven}} by {{Rough Paths}}:
  {{{\'E}cole}} d'{{{\'E}t{\'e}}} de {{Probabilit{\'e}s}} de {{Saint-Flour
  XXXIV}} - 2004}, volume 1908 of {\em Lecture {{Notes}} in {{Mathematics}}}.
\newblock Springer Berlin Heidelberg, Berlin, Heidelberg.

\bibitem[Malinsky and Spirtes, 2018]{malinsky2018causala}
Malinsky, D. and Spirtes, P. (2018).
\newblock Causal {{Structure Learning}} from {{Multivariate Time Series}} in
  {{Settings}} with {{Unmeasured Confounding}}.
\newblock In {\em Proceedings of 2018 {{ACM SIGKDD Workshop}} on {{Causal
  Discovery}}}, pages 23--47. PMLR.

\bibitem[Mani, 2006]{mani2006bayesian}
Mani, S. (2006).
\newblock {\em A {{Bayesian Local Causal Discovery Framework}}}.
\newblock University of {{Pittsburgh ETD}}, University of Pittsburgh.

\bibitem[Manten et~al., 2024]{manten2024signature}
Manten, G., Casolo, C., Ferrucci, E., Mogensen, S., Salvi, C., and Kilbertus,
  N. (2024).
\newblock Signature {{Kernel Conditional Independence Tests}} in {{Causal
  Discovery}} for {{Stochastic Processes}}.

\bibitem[Manten et~al., 2025]{manten2025asymmetric}
Manten, G., Casolo, C., Mogensen, S., and Kilbertus, N. (2025).
\newblock An {{Asymmetric Independence Model}} for {{Causal Discovery}} on
  {{Path Spaces}}.

\bibitem[Mogensen and Hansen, 2022]{mogensen2022graphical}
Mogensen, S. and Hansen, N. (2022).
\newblock Graphical modeling of stochastic processes driven by correlated
  noise.
\newblock {\em Bernoulli}, 28(4).

\bibitem[Mogensen and Hansen, 2020]{mogensen2020markov}
Mogensen, S. and Hansen, N.~R. (2020).
\newblock Markov equivalence of marginalized local independence graphs.
\newblock {\em The Annals of Statistics}, 48(1).

\bibitem[Mogensen et~al., 2018]{mogensen2018causal}
Mogensen, S., Malinsky, D., and Hansen, N. (2018).
\newblock Causal {{Learning}} for {{Partially Observed Stochastic Dynamical
  Systems}}.

\bibitem[Mooij and Claassen, 2020]{mooij2020constraintbased}
Mooij, J. and Claassen, T. (2020).
\newblock Constraint-{{Based Causal Discovery}} using {{Partial Ancestral
  Graphs}} in the presence of {{Cycles}}.
\newblock In {\em UAI2020}, pages 1159--1168. PMLR.

\bibitem[Mooij et~al., 2013]{mooij2013ordinary}
Mooij, J., Janzing, D., and Sch{\"o}lkopf, B. (2013).
\newblock From {{Ordinary Differential Equations}} to {{Structural Causal
  Models}}: The deterministic case.

\bibitem[Mooij et~al., 2020]{mooij2020joint}
Mooij, J., Magliacane, S., and Claassen, T. (2020).
\newblock Joint causal inference from multiple contexts.
\newblock {\em The Journal of Machine Learning Research},
  21(1):99:3919--99:4026.

\bibitem[Nathaniel et~al., 2025]{nathaniel2025deep}
Nathaniel, J., Roesch, C., Buch, J., DeSantis, D., Rupe, A., Lamb, K., and
  Gentine, P. (2025).
\newblock Deep {{Koopman}} operator framework for causal discovery in nonlinear
  dynamical systems.

\bibitem[Neal, 2000]{neal2000deducing}
Neal, R.~M. (2000).
\newblock On {{Deducing Conditional Independence}} from d-{{Separation}} in
  {{Causal Graphs}} with {{Feedback}} ({{Research Note}}).
\newblock {\em Journal of Artificial Intelligence Research}, 12:87--91.

\bibitem[Niemiro, 2024]{niemiro2024causal}
Niemiro, W. (2024).
\newblock Causal graphs, composable stochastic processes and conditional
  independence.
\newblock {\em Applicationes Mathematicae}, pages 1--22.

\bibitem[Pearl, 1993]{pearl1993comment}
Pearl, J. (1993).
\newblock Comment: {{Graphical Models}}, {{Causality}} and {{Intervention}}.
\newblock {\em Statistical Science}, 8(3):266--269.

\bibitem[Pearl, 1995]{pearl1995causal}
Pearl, J. (1995).
\newblock Causal {{Diagrams}} for {{Empirical Research}}.
\newblock {\em Biometrika}, 82(4):669--688.

\bibitem[Pearl, 2009]{pearl2009causality}
Pearl, J. (2009).
\newblock {\em Causality}.
\newblock Cambridge University Press.

\bibitem[Pearl and Dechter, 1996]{pearl1996identifying}
Pearl, J. and Dechter, R. (1996).
\newblock Identifying independencies in causal graphs with feedback.
\newblock In {\em Proceedings of the {{Twelfth}} International Conference on
  {{Uncertainty}} in Artificial Intelligence}, {{UAI}}'96, pages 420--426, San
  Francisco, CA, USA. Morgan Kaufmann Publishers Inc.

\bibitem[Peters et~al., 2020]{peters2020causal}
Peters, J., Bauer, S., and Pfister, N. (2020).
\newblock Causal models for dynamical systems.

\bibitem[Peters et~al., 2013]{peters2013causal}
Peters, J., Janzing, D., and Sch{\"o}lkopf, B. (2013).
\newblock Causal {{Inference}} on {{Time Series}} using {{Restricted Structural
  Equation Models}}.
\newblock In {\em Advances in {{Neural Information Processing Systems}}},
  volume~26. Curran Associates, Inc.

\bibitem[Protter, 2005]{protter2005stochastic}
Protter, P. (2005).
\newblock {\em Stochastic {{Integration}} and {{Differential Equations}}},
  volume~21 of {\em Stochastic {{Modelling}} and {{Applied Probability}}}.
\newblock Springer Berlin Heidelberg, Berlin, Heidelberg.

\bibitem[Przyby{\l}owicz et~al., 2024]{przybylowicz2024skorohod}
Przyby{\l}owicz, P., Schwarz, V., Steinicke, A., and Sz{\"o}lgyenyi, M. (2024).
\newblock A {{Skorohod}} measurable universal functional representation of
  solutions to semimartingale {{SDEs}}.
\newblock {\em Stochastic Analysis and Applications}, 42(6):1137--1155.

\bibitem[Reisach et~al., 2025]{reisach2025case}
Reisach, A.~G., Su{\'a}rez, A., Weichwald, S., and Chambaz, A. (2025).
\newblock The {{Case}} for {{Time}} in {{Causal DAGs}}.

\bibitem[Reiter et~al., 2024]{reiter2024causal}
Reiter, N., Gerhardus, A., Wahl, J., and Runge, J. (2024).
\newblock Causal {{Inference}} on {{Process Graphs}}, {{Part I}}: {{The
  Structural Equation Process Representation}}.

\bibitem[Richardson, 1996]{richardson1996discovery}
Richardson, T. (1996).
\newblock A discovery algorithm for directed cyclic graphs.
\newblock In {\em Proceedings of the {{Twelfth}} International Conference on
  {{Uncertainty}} in Artificial Intelligence}, {{UAI}}'96, pages 454--461, San
  Francisco, CA, USA. Morgan Kaufmann Publishers Inc.

\bibitem[R{\o}ysland et~al., 2024]{roysland2024graphical}
R{\o}ysland, K., Ryalen, P., Nyg{\aa}rd, M., and Didelez, V. (2024).
\newblock Graphical criteria for the identification of marginal causal effects
  in continuous-time survival and event-history analyses.

\bibitem[Rubenstein et~al., 2018]{rubenstein2018deterministic}
Rubenstein, P., Bongers, S., Sch{\"o}lkopf, B., and Mooij, J. (2018).
\newblock From {{Deterministic ODEs}} to {{Dynamic Structural Causal Models}}.
\newblock In {\em Conference on {{Uncertainty}} in {{Artificial
  Intelligence}}}.

\bibitem[Runge et~al., 2019]{runge2019detecting}
Runge, J., Nowack, P., Kretschmer, M., Flaxman, S., and Sejdinovic, D. (2019).
\newblock Detecting and quantifying causal associations in large nonlinear time
  series datasets.
\newblock {\em Science Advances}, 5(11):eaau4996.

\bibitem[Ryalen et~al., 2026]{ryalen2026causal}
Ryalen, P.~C., Stensrud, M.~J., and R{\o}ysland, K. (2026).
\newblock On causal inference with marked point process data.

\bibitem[Rytgaard et~al., 2022]{rytgaard2022continuoustime}
Rytgaard, H., Gerds, T., and Van Der~Laan, M. (2022).
\newblock Continuous-time targeted minimum loss-based estimation of
  intervention-specific mean outcomes.
\newblock {\em The Annals of Statistics}, 50(5).

\bibitem[Scheff\'e, 1947]{scheffe1947useful}
Scheff\'e, H. (1947).
\newblock A {{Useful Convergence Theorem}} for {{Probability Distributions}}.
\newblock {\em The Annals of Mathematical Statistics}, 18(3):434--438.

\bibitem[Schwank and Drton, 2026]{schwank2026nonparametric}
Schwank, R. and Drton, M. (2026).
\newblock Non-parametric recovery of causal diffusion mechanisms from
  steady-state observations.

\bibitem[Schweder, 1970]{schweder1970composable}
Schweder, T. (1970).
\newblock Composable {{Markov Processes}}.
\newblock {\em Journal of Applied Probability}, 7(2):400--410.

\bibitem[Skorokhod, 1956]{skorokhod1956limit}
Skorokhod, A. (1956).
\newblock Limit {{Theorems}} for {{Stochastic Processes}}.
\newblock {\em Theory of Probability \& Its Applications}, 1(3):261--290.

\bibitem[Spirtes, 1994]{spirtes1994conditional}
Spirtes, P. (1994).
\newblock Conditional independence in directed cyclic graphical models for
  feedback.
\newblock {\em Technical Report CMU-PHIL-54, Carnegie Mellon University}.

\bibitem[Spirtes, 1995]{spirtes1995directed}
Spirtes, P. (1995).
\newblock Directed cyclic graphical representations of feedback models.
\newblock In {\em Proceedings of the {{Eleventh Conference}} on {{Uncertainty}}
  in {{Artificial Intelligence}} ({{UAI-95}})}, pages 499--506.

\bibitem[Spirtes et~al., 1993]{spirtes1993causation}
Spirtes, P., Glymour, C., and Scheines, R. (1993).
\newblock {\em Causation, {{Prediction}}, and {{Search}}}, volume~81 of {\em
  Lecture {{Notes}} in {{Statistics}}}.
\newblock Springer, New York, NY.

\bibitem[Spirtes et~al., 2000]{spirtes2000causation}
Spirtes, P., Glymour, C., and Scheines, R. (2000).
\newblock {\em Causation, Prediction, and Search}.
\newblock Adaptive Computation and Machine Learning. MIT Press, Cambridge,
  Mass, 2nd ed edition.

\bibitem[Spirtes et~al., 1999]{spirtes1999algorithm}
Spirtes, P., Meek, C., and Richardson, T. (1999).
\newblock An {{Algorithm}} for {{Causal Inference}} in the {{Presence}} of
  {{Latent Variables}} and {{Selection Bias}}.
\newblock In {\em Computation, {{Causation}} and {{Discovery}}}, pages
  211--252. The MIT Press.

\bibitem[Strobl, 2019]{strobl2019constraintbased}
Strobl, E. (2019).
\newblock A constraint-based algorithm for causal discovery with cycles, latent
  variables and selection bias.
\newblock {\em International Journal of Data Science and Analytics},
  8(1):33--56.

\bibitem[Weinberger, 2026]{weinberger2026homeostasis}
Weinberger, N. (2026).
\newblock Homeostasis and causal control.
\newblock {\em Biology \& Philosophy}, 41(2):18.

\bibitem[White and Lu, 2010]{white2010granger}
White, H. and Lu, X. (2010).
\newblock Granger {{Causality}} and {{Dynamic Structural Systems}}.
\newblock {\em Journal of Financial Econometrics}, 8(2):193--243.

\bibitem[Yamada and Watanabe, 1971]{yamada1971uniqueness}
Yamada, T. and Watanabe, S. (1971).
\newblock On the uniqueness of solutions of stochastic differential equations.
\newblock {\em Kyoto Journal of Mathematics}, 11(1).

\bibitem[Zhang, 2008]{zhang2008completeness}
Zhang, J. (2008).
\newblock On the completeness of orientation rules for causal discovery in the
  presence of latent confounders and selection bias.
\newblock {\em Artificial Intelligence}, 172(16-17):1873--1896.

\end{thebibliography}

\end{document}